\documentclass[11pt]{article}
\usepackage[compress]{natbib}
\usepackage{amsmath,amssymb,amsthm}
\usepackage{color}
\usepackage[margin=1in]{geometry}
\usepackage{setspace}
\usepackage{amsthm, amsmath, amssymb, amsfonts}
\usepackage{bm}
\usepackage[dvipsnames,table,xcdraw]{xcolor}
\usepackage{hyperref}
\hypersetup{pdftex, colorlinks=true, citecolor=MidnightBlue, linkcolor=red!80!black, urlcolor=MidnightBlue}
\usepackage{caption}
\usepackage{subfigure}

\usepackage{url}            % simple URL typesetting
\usepackage{booktabs}       % professional-quality tables
\usepackage{nicefrac}       % compact symbols for 1/2, etc.
\usepackage{graphicx}
\usepackage{caption}
\usepackage{thm-restate}
\usepackage[linesnumbered, ruled]{algorithm2e}
\usepackage{algpseudocode}
\usepackage{bbm}
\usepackage[export]{adjustbox}

\usepackage{microtype}
\usepackage{mathpazo} % math & rm
\usepackage[scaled]{helvet} % ss
\usepackage{courier} % tt
\normalfont
\usepackage[T1]{fontenc}

\usepackage{threeparttable}
\allowdisplaybreaks[4] % spread setup  
\usepackage[shortlabels]{enumitem} 

\usepackage{tikz}
\usetikzlibrary{shadows}
\usetikzlibrary{arrows,positioning} 

\usepackage{tcolorbox}

\def\O#1{\text{\ding{\the\numexpr#1+171}}}

\newcommand{\R}{\mathbb R}

\DeclareMathOperator{\prox}{prox}
\DeclareMathOperator{\proj}{proj}

\DeclareMathOperator{\dist}{dist}
\DeclareMathOperator*{\argmin}{argmin}
\DeclareMathOperator*{\argmax}{argmax}

\def\O{{\mathbb{O}}}

\def\R{{\mathbb{R}}}

\def\cR{{\cal \cR}}

\def\1{{\mathbf 1}}

\DeclareMathOperator{\st}{s.t.}

\usepackage{thmtools}

\declaretheoremstyle[parent=section]{definitionwithend}
\declaretheorem[style=definitionwithend]{corollary}
\declaretheorem[style=definitionwithend]{theorem}
\declaretheorem[style=definitionwithend]{proposition}
\declaretheorem[style=definitionwithend]{definition}
\declaretheorem[style=definitionwithend]{assumption}
\declaretheorem[style=definitionwithend]{example}
\declaretheorem[style=definitionwithend]{remark}

\declaretheorem[style=definitionwithend]{lemma}
\declaretheorem[style=definitionwithend]{fact} 

\def\keywords{\vspace{1.5em}
{\textbf{Keywords}:\,\relax%
}}

\usepackage{appendix}
\usepackage{pifont}

\title{ 
Nonsmooth Nonconvex-Nonconcave Minimax Optimization: Primal-Dual 
Balancing and Iteration Complexity Analysis\thanks{J. Li is supported by the University of British Columbia Start-Up Grant GR032687. A. M.-C. So is supported in part by the Hong Kong Research Grants Council (RGC) General Research Fund (GRF) project CUHK 14204823.}
}

\date{\today}
\author{%
    Jiajin Li\thanks{Sauder School of Business, University of British Columbia, Vancouver, BC, Canada. \texttt{jiajin.li@sauder.ubc.ca}} \and
    Linglingzhi Zhu\thanks{H. Milton Stewart School of Industrial and Systems Engineering, Georgia Institute of Technology, Atlanta, Georgia, USA. \texttt{llzzhu@gatech.edu}} \and
    Anthony Man-Cho So\thanks{Department of Systems Engineering and Engineering Management, The Chinese University of Hong Kong, Shatin, NT, Hong Kong. \texttt{manchoso@se.cuhk.edu.hk}}
}
\begin{document}
%\doublespacing
%\onehalfspacing
\maketitle

\begin{abstract}
Nonconvex-nonconcave minimax optimization has gained widespread interest over the last decade. However, most existing works focus on variants of gradient descent-ascent (GDA) algorithms, which are only applicable to smooth nonconvex-concave settings. 
To address this limitation, we propose a novel algorithm named smoothed proximal linear descent-ascent (smoothed PLDA), which can effectively handle a broad range of structured nonsmooth nonconvex-nonconcave minimax problems. 
Specifically, we consider the setting where the primal function has a nonsmooth composite structure and the dual problem possesses the Kurdyka-\L{}ojasiewicz (K\L{}) property with exponent $\theta \in [0,1)$. We introduce a novel convergence analysis framework for smoothed PLDA, the key components of which are our newly developed nonsmooth primal error bound and dual error bound. Using this framework,  we show that smoothed PLDA can find both $\epsilon$-game-stationary points and $\epsilon$-optimization-stationary points of the problems of interest in $\mathcal{O}(\epsilon^{-2\max\{2\theta,1\}})$ iterations.
Furthermore, when $\theta \in [0,\frac{1}{2}]$, smoothed PLDA achieves the optimal iteration complexity of $\mathcal{O}(\epsilon^{-2})$.
To further demonstrate the effectiveness and wide applicability of our analysis framework, we show that certain max-structured problem possesses the K\L{} property with exponent $\theta=0$ under mild assumptions. As a by-product, we establish algorithm-independent quantitative relationships among various stationarity concepts, which may be of independent interest.

\end{abstract}


\section{Introduction}
In this paper, we aim to solve the following minimax problem:
\begin{equation}
\label{eq:problem}
\min_{x\in \mathcal{X}} \max_{y\in\mathcal{Y}} F(x,y).
\tag{P}
\end{equation}
Here, $F:\R^n \times \R^d \rightarrow \R$ can be both nonconvex with respect to $x$ and nonconcave with respect to $y$, $\mathcal{X}\subseteq \R^n$ is a nonempty closed convex set, and $\mathcal{Y}\subseteq \R^d$ is a nonempty compact convex set. This problem has a wide range of applications in machine learning and operations research. Examples include learning with non-decomposable loss~\citep{shalev2016minimizing,zhang2018ell_1}, training generative adversarial networks~\citep{arjovsky2017wasserstein,goodfellow2020generative}, adversarial training~\citep{madry2017towards,sinha2017certifying}, and (distributionally) robust optimization (DRO)~\citep{bertsimas2011theory,rahimian2019distributionally}.

When $F$ is a smooth function, a natural and intuitive approach for solving \eqref{eq:problem} is gradient descent-ascent (GDA), which involves performing gradient descent on the primal variable $x$ and gradient ascent on the dual variable $y$ in each iteration. 
For nonconvex-strongly concave problems, GDA can find an $\epsilon$-stationary point with an iteration complexity of ${\mathcal{O}}(\epsilon^{-2})$~\citep{lin2020gradient}. 
This matches the lower bound for solving \eqref{eq:problem} using a first-order oracle, as shown in \citep{carmon2020lower,li2021complexity,zhang2021complexity}.
However, when $F(x,\cdot)$ is not strongly concave for some $x\in\mathcal{X}$, GDA may encounter oscillations, even for bilinear problems. To overcome this challenge, various techniques utilizing diminishing step sizes have been proposed to guarantee convergence. Nevertheless, these techniques can only achieve a suboptimal complexity of ${\mathcal{O}}(\epsilon^{-6})$ at best~\citep{jin2020local,lin2020gradient,lu2020hybrid}. 
To achieve a lower iteration complexity, the works \citep{xu2020unified,zhang2020single} introduce a smoothing technique to stabilize the iterates. As a result, the complexity is improved to ${\mathcal{O}}(\epsilon^{-4})$ for general nonconvex-concave problems. 
Moreover, the works \citep{yang2022faster,zhang2020single} achieve the ``optimal'' iteration complexity of $\mathcal{O}(\epsilon^{-2})$ for dual functions that either take a pointwise-maximum form or possess the Polyak-\L{}ojasiewicz (P{\L}) property~\citep{polyak1963gradient}.\footnote{As far as we know, 
the lowest complexity required for finding approximate stationary points of nonconvex-P\L{} minimax/pointwise maximum problems remains an open question. Nevertheless, as we mentioned, for smooth nonconvex-strongly concave problems, the $\mathcal{O}(\epsilon^{-2})$ complexity is already optimal.}
On another front, multi-loop-type algorithms have been developed to achieve a lower iteration complexity for general nonconvex-concave problems \citep{nouiehed2019solving,thekumparampil2019efficient,yang2020catalyst,ostrovskii2021efficient}. Among these, the two triple-loop algorithms in \citep{lin2020near,ostrovskii2021efficient} have the lowest iteration complexity of $\mathcal{O}(\epsilon^{-2.5})$.

If we take one step into the nonsmooth world with \eqref{eq:problem} possessing a separable nonsmooth structure, i.e., its objective function consists of a smooth term and a separable nonsmooth term whose proximal mapping can be readily computed, then the analytic and algorithmic framework from the purely smooth case can be easily adapted.
Specifically, a class of (accelerated) proximal-GDA type algorithms have been proposed, 
where the gradient step is substituted by the proximal gradient step \citep{barazandeh2020solving,boct2020alternating,chen2021proximal,huang2021efficient}. 
By utilizing the gradient Lipschitz continuity condition of the smooth term, these algorithms can be shown to achieve the same iteration complexity as those for the smooth case.

As we can see, most existing algorithms can only handle the \textit{almost} smooth case, which refers to scenarios where at least some gradient information is obtainable, such as in purely smooth or separable nonsmooth problems. By contrast, the \textit{general} nonsmooth problem has received relatively little attention in the minimax literature.
Only recently have two algorithms, namely the proximally guided stochastic subgradient method~\citep{rafique2022weakly} and two-timescale GDA~\citep{lin2022nonasymptotic}, been proposed to tackle general nonsmooth weakly convex-concave problems. Unfortunately, these methods suffer from the high iteration complexities of $\mathcal{O}(\epsilon^{-6})$ and $\mathcal{O}(\epsilon^{-8})$, respectively, 
% The reason is that these methods
since they only use subgradient information and neglect problem-specific structures. 
Thus, a natural question arises: 
% \begin{center}
{\bf 
Q1: Are there structured nonsmooth nonconvex-concave problems whose $\epsilon$-stationary points can be found with an iteration complexity of $\mathcal{O}(\epsilon^{-4})$, just like smooth nonconvex-concave problems?}
% \end{center}
Taking a step further, we would like to pose a more challenging and intriguing question:
% \begin{center}
{\bf Q2: Can we identify general regularity conditions to achieve the “optimal” rate of $\mathcal{O}(\epsilon^{-2})$ for nonsmooth nonconvex-nonconcave problems?
% \end{center}
}
% \vspace{-3mm}
\subsection{Main Contributions}
This paper provides affirmative answers to  both questions \textbf{Q1} and \textbf{Q2} for a particular class of nonsmooth nonconvex-nonconcave problems. We focus on a primal function that has the composite structure $F(\cdot, y) := h_y \circ c_y$ for each $y\in\mathcal{Y}$, where $h_y$ is convex, Lipschitz continuous, and possibly nonsmooth; $c_y$ is continuously differentiable with a Lipschitz continuous Jacobian mapping; and the dual function is continuously differentiable and gradient Lipschitz continuous on $\mathcal{X} \times \mathcal{Y}$. Additionally, we assume that either the dual function is concave or the dual problem possesses the K\L{} property with exponent $\theta \in [0, 1)$. As concavity alone cannot guarantee the K\L{} property~\citep{bolte2010characterizations}, we deal with the concave case separately.

To start, we introduce a new algorithm called smoothed proximal linear descent-ascent (smoothed PLDA), which builds on the smoothed GDA algorithm~\citep{zhang2020single}. Unlike its predecessor, smoothed PLDA is extended to handle nonsmooth composite nonconvex-nonconcave problems. To achieve this, we leverage the proximal linear scheme to effectively manage the nonsmooth composite structure of the primal function. This scheme has been extensively studied in recent literature~\citep{nesterov2007modified,cartis2011evaluation,hu2016convergence,lewis2016proximal,drusvyatskiy2018error,drusvyatskiy2019efficiency,hu2023linearized}, though its application to minimax optimization has not been explored.

Although the algorithmic extension may seem intuitive and straightforward, it introduces notable challenges in the convergence analysis. These challenges arise primarily from the lack of gradient Lipschitz continuity of the primal function and the nonconcavity of the dual function. To address these challenges, we first establish a tight Lipschitz-type primal error bound property of the proximal linear scheme, which holds even without the gradient Lipschitz continuity condition (see Proposition \ref{prop:lip}). This error bound is new and plays a crucial role in demonstrating the sufficient decrease property of the designed Lyapunov function. The sufficient decrease property serves as a starting point to ensure the global convergence of smoothed PLDA. 
% Also, this novel error bound is noteworthy for its implications in nonsmooth optimization, which could be of independent interest.

Next, we observe that the nonconcavity of the dual function poses a fundamental challenge in achieving a favorable tradeoff between the decrease in the primal and the increase in the dual. This challenge arises from the absence of inherent dominance between the primal and dual functions, making it difficult to find an optimal balance between the primal and dual updates.
As it turns out, the primal-dual tradeoff directly impacts the convergence rate. 
To quantify this tradeoff, we introduce a new dual error bound property based on the K\L{} exponent of the dual problem (see Proposition \ref{prop:dual_eb_KL}). 
This is a notable departure from the usual approach of utilizing the K\L{} exponent in pure primal nonconvex optimization~\citep{li2018calculus, attouch2013convergence}.

\begin{table*}[]
  \caption{Comparison of the iteration complexities of smoothed PLDA and other related methods under different settings for solving $\min_{x\in \mathcal{X}}\max_{y \in \mathcal{Y}} F(x,y)$. }
 \label{tab:summary}
 \centering
\begin{tikzpicture}
 \node[drop shadow,fill=white,inner sep=0pt] 
 {
 \resizebox{0.98\textwidth}{!}{ 
\begin{tabular}{|c|c|c|c|c|}
 \hline
 & \textbf{Primal Func.} & \textbf{Dual Func./Prob.} & \textbf{Iter. Compl.} & \textbf{Add. Asm.} \\ \hline
 Two-timescale GDA~\citep{lin2020gradient}  &    $L$-smooth                                  & concave  & $\mathcal{O}(\epsilon^{-6})$ & $\mathcal{X}=\R^n$               \\ \hline
Smoothed GDA~\citep{zhang2020single} & $L$-smooth & concave          & $\mathcal{O}(\epsilon^{-4})$ &    ---     \\ \hline
PG-SMD~\citep{rafique2022weakly} & weakly convex & concave          & $\mathcal{O}(\epsilon^{-6})$ &    $\mathcal{X}$ bounded        \\ \hline
 \rowcolor[HTML]{ECF4FF} 
Smoothed PLDA &      nonsmooth composite      & concave & $\mathcal{O}(\epsilon^{-4})$  & ---\\ \hline
GDA~\citep{lin2020gradient} &    $L$-smooth                                  & strongly concave  & $\mathcal{O}(\epsilon^{-2})$ & $\mathcal{X}=\R^n$               \\ \hline
Smoothed GDA~\citep{yang2022faster} & $L$-smooth & P\L{}      & $\mathcal{O}(\epsilon^{-2})$ &    $\mathcal{Y} = \R^d$   \\ \hline
 \rowcolor[HTML]{ECF4FF} 
Smoothed PLDA &      nonsmooth composite      &  K\L{} with exponent $\theta$   & $\mathcal{O}(\epsilon^{-2\max\{2\theta,1\}})$  &$\mathcal{X}$ bounded \\ \hline
\end{tabular}}
};

\end{tikzpicture}
% \\\vspace{1mm}{\raggedleft\footnotesize{$^1$ All methods use the $\epsilon$-game stationarity, see Section \ref{sec:sta} for detailed stationarity concepts.}\par}
\vspace{-4mm}
\end{table*}

The aforementioned primal and dual error bounds lie at the core of our convergence analysis framework. Our main result is that, when the dual problem possesses the K\L{} property with exponent $\theta \in [0,1)$, smoothed PLDA can find both $\epsilon$-game-stationary points and $\epsilon$-optimization-stationary points of the nonsmooth composite nonconvex-nonconcave problems introduced earlier in $\mathcal{O}(\epsilon^{-2\max\{2\theta,1\}})$ iterations. For concave functions, smoothed PLDA attains the same iteration complexity of $\mathcal{O}(\epsilon^{-4})$ as the smoothed GDA in~\citep{zhang2020single}. Therefore, we address \textbf{Q1} by extending smoothed GDA to a nonsmooth nonconvex-nonconcave setting while preserving at least the same iteration complexity. Table \ref{tab:summary} presents a summary of the iteration complexities of smoothed PLDA and other related methods in various scenarios. Using our analysis framework, we further show that when the K\L{} exponent of the dual problem lies between 0 and $\frac{1}{2}$, smoothed PLDA achieves the ``optimal'' iteration complexity of $\mathcal{O}(\epsilon^{-2})$, thus addressing \textbf{Q2}.

Interestingly, our analysis framework also reveals a phase transition phenomenon: The iteration complexity depends on the slower of the primal and dual variable updates, which can be characterized by the dual error bound. Specifically, when $\theta \in [0,\frac{1}{2}]$, the dual update is faster than the primal update, which causes the primal update to dominate the optimization process. However, since the primal update involves solving a strongly convex problem, we cannot achieve an iteration complexity better than $\mathcal{O}(\epsilon^{-2})$, which is optimal for nonconvex-strongly concave problems. By contrast, when $\theta \in (\frac{1}{2},1)$, the dual update dominates, and the iteration complexity explicitly depends on the K\L{} exponent $\theta$, i.e.,  $\mathcal{O}(\epsilon^{-4\theta})$. 

To further demonstrate the wide applicability and effectiveness of our analysis framework, we apply it to problems with a linear dual function and polytopal constraints on $y$. 
An important example is the max-structured problem, which takes the form
\begin{equation}
\label{eq:max_structure}
\min_{x\in\mathcal{X}} \max_{i\in[d]} G_i(x). 
\end{equation}
Here, $G_i: \mathbb{R}^n \rightarrow \mathbb{R}$, where $i \in [d]$, is a nonconvex function with the composite structure mentioned earlier. Such a problem
arises frequently in machine learning applications, including distributionally robust optimization (DRO) \citep{blanchet2019multivariate,shafieezadeh2019regularization,gao2022wasserstein}, adversarial training \citep{madry2017towards}, fairness training \citep{mohri2019agnostic,nouiehed2019solving}, and distribution-agnostic meta-learning \citep{collins2020distribution}.
Under some regularity conditions, we show that problem \eqref{eq:max_structure} possesses the  K\L{} property with exponent $\theta=0$.

Finally, the new dual error bound enables us to establish algorithm-independent quantitative relationships among different stationarity concepts (see Theorem \ref{thm-stationary}). The relationships are obtained as a by-product of our analysis and extend the scope of previous results in \citep{jin2020local,razaviyayn2020nonconvex,daskalakis2021complexity} to a wider range of settings. Furthermore, they hold great promise for demystifying various notions of stationarity in the context of minimax optimization.

\paragraph{Structure of the paper} The paper is organized as follows. In Section \ref{sec:motivation}, we provide several representative applications to demonstrate not only the prevalence of problem \eqref{eq:problem} when the primal function $F$ has the composite structure mentioned earlier but also the versatility of our proposed smoothed PLDA. In Section \ref{sec:pre}, we introduce the problem setup and key concepts used throughout the paper. 
We then present our proposed algorithm and a new convergence analysis framework for studying it in Sections \ref{sec:main} and \ref{sec:analysis}, respectively.
In Section \ref{sec:verify}, we verify that the max-structured problem \eqref{eq:max_structure} possesses the K\L{} property with exponent $\theta=0$.
In Section \ref{sec:sta}, we clarify the relationships among various stationarity concepts, both conceptually and quantitatively. Finally, we end with some closing remarks in Section \ref{sec:conclu}.

\section{Motivating Applications}
\label{sec:motivation}
One important class of problems that is mostly beyond the reach of existing algorithmic approaches but can be tackled by our proposed smoothed PLDA is ``two-layer'' nonsmooth composite optimization, which takes the form
\begin{equation}
\min_{x\in \mathcal{X}} f(h(c(x))).
\label{eq:one_non_s}
\end{equation}
Here, $f:\R^d \rightarrow \R$ is convex, $h:\R^m\rightarrow \R^d$ is convex and Lipschitz continuous, and $c:\R^n\rightarrow \R^m$ is continuously differentiable with a Lipschitz continuous Jacobian mapping. Note that both $f$ and $h$ can be nonsmooth. Now, observe that we can isolate the first-layer nonsmoothness by reformulating \eqref{eq:one_non_s} as the minimax problem
\begin{equation}
\min_{x\in \mathcal{X}} \max_{y\in \R^d} y^\top h(c(x))-f^\star(y),
\label{eq:two_non_s}
\end{equation}
where $f^\star: \R^d \rightarrow \R \cup \{+\infty\}$ is the conjugate function of $f$. Consequently, we can apply our proposed smoothed PLDA to solve \eqref{eq:two_non_s}. 

Various applications admit formulations that feature the nonsmooth composite structure in \eqref{eq:one_non_s}. Let us discuss two representative ones here.

\vspace{1ex}
{\rm (i)} \textbf{Variation-Regularized Wasserstein DRO~\citep{gao2022wasserstein}} 
\vspace{1ex}

The DRO methodology aims to find optimal decisions under the most adverse distribution in an ambiguity set, which consists of all probability distributions that fit the observed data with high confidence. In recent years, several works have interpreted regularization from a DRO perspective; see, e.g.,~\citep{namkoong2016stochastic,blanchet2019robust,shafieezadeh2019regularization,cranko2021generalised,rahimian2019distributionally,gao2022wasserstein}. These works provide a probabilistic justification for existing regularization techniques and offer an alternative approach to tackle risk minimization problems. 
For instance, the work \cite{gao2022wasserstein} establishes the asymptotic equivalence between Wasserstein distance-induced DRO and the following variation-regularized problem:
    \begin{equation}\label{eq:regu}
    \min_{\theta \in \Theta} \, \frac{1}{N}\sum_{i=1}^N  \ell(f_\theta(x_i),y_i)  +\rho \max_{i \in [N]} \|\nabla_x f_\theta(x_i)\|_q.
\end{equation}
Here, $\{(x_i,y_i)\}_{i\in [N]}$ are the feature-label pairs, $\ell:\R\rightarrow \R$ is the loss function, $f_\theta$ is the feature mapping (e.g., a neural network parameterized by $\theta$), $\rho>0$ is the constraint radius, and $q\in[1,+\infty]$ is a parameter. It can be shown that problem~\eqref{eq:regu} takes the form in~\eqref{eq:one_non_s}, where the assumptions on $f,h,c$ are satisfied when, e.g., $\ell$ is convex and Lipschitz, and 
$f_\theta(x_i)$ has a Lipschitz continuous gradient and $\nabla_x f_\theta(x_i)$ has a Lipschitz continuous Jacobian with respect to $\theta$ for all $i \in [N]$. The variation regularization can be regarded as an empirical alternative to Lipschitz regularization, where the goal is to promote smoothness of $f_{\theta}$ over the entire training dataset. Since problem \eqref{eq:regu} is closely related to the task of controlling the Lipschitz constant of deep neural networks, it has attracted significant interest.
In particular, problem \eqref{eq:regu} with $\ell(\cdot) =|\cdot|$ being the absolute loss and $f_{\theta}$ being a linear mapping is thoroughly investigated in \citep{blanchet2019multivariate}.

\vspace{1ex}
{\rm (ii)} \textbf{$\ell_1$-Regression with Heavy-Tailed Distributions~\citep{zhang2018ell_1}}
\vspace{1ex}

In linear regression, when the input and output follow heavy-tailed distributions, empirical risk minimization may no longer be a suitable approach. This observation has led to recent research on learning with heavy-tailed distributions~\citep{audibert2011robust}.  For instance, the work~\cite{zhang2018ell_1} proposes the following truncated minimization problem for $\ell_1$-regression with heavy-tailed distributions:
\begin{equation}
\label{eq:l1_regression}
\min_{\theta \in \Theta} \, \frac{1}{\alpha N}\sum_{i=1}^N  \psi\left(\alpha |y_i-\theta^\top x_i|\right). 
\end{equation}
Here, $\alpha > 0$ is a parameter and $\psi:\R\rightarrow \R$ is the nondecreasing truncation function defined by
\[
\psi(t)=\begin{cases} &\log \left(1+t+\frac{t^2}{2}\right), \quad t \geq 0 \\ &-\log \left(1-t+\frac{t^2}{2}\right),\quad t \leq 0 .\end{cases}
\]
Note that $\psi$ is a $\frac{1}{4}$-weakly convex function. Then, problem \eqref{eq:l1_regression} can be reformulated as 
\begin{equation*}
    \min_{\theta\in\Theta} \max_{q \in \R^N} \frac{1}{\alpha N} \sum_{i=1}^N \left[ \alpha q_i |y_i-\theta^\top x_i|-\psi^\star_\mu(q_i) -\frac{\mu\alpha^2}{2}\|y_i-\theta^\top x_i\|^2\right],
\end{equation*}
where $\psi_{\mu}(\cdot) = \psi(\cdot)+\frac{\mu}{2}\|\cdot\|^2$ and $\mu >\frac{1}{4}$ so as to enforce the strong convexity of $\psi_{\mu}$. To the best of our knowledge, there is no provably efficient algorithm for solving \eqref{eq:l1_regression}.

Besides the nonsmooth composite optimization problem \eqref{eq:one_non_s}, a variety of minimax problems can also be tackled by smoothed PLDA. Let us illustrate this by considering the general $\phi$-divergence DRO problem studied in \citep{levy2020large}.

\vspace{1ex}
{\rm (iii)} \textbf{General $\phi$-Divergence DRO \citep{levy2020large}}
\vspace{1ex}

Building on the problem setup in (i), we consider a slight generalization of $\phi$-divergence DRO, which is given by
\begin{equation}
\label{eq:phi-dro}
\min_{\theta\in\Theta} \max_{q \in \Delta_N:\sum_{i=1}^N \phi(N q_i) \leq n\rho} \frac{1}{N}\sum_{i=1}^N q_i \ell(f_\theta(x_i,y_i))-\psi(N q_i).
\end{equation}
Here, the functions $\phi, \psi: \mathbb{R} \rightarrow \mathbb{R}$ are convex and satisfy $\phi(1)=\psi(1)=0$, $\rho\ge0$ is the constraint radius, and $\Delta_N$ is the $N$-dimensional standard simplex. When $\phi=0$ and $\psi= \iota_{[0,1/\alpha)}$\footnote{For any set $\mathcal{S} \subseteq \mathbb{R}^{\ell}$, we use $\iota_{\mathcal{S}}: \mathbb{R}^{\ell} \rightarrow \{0, +\infty\}$ to denote the indicator function associated with $\mathcal{S}$.} with $\alpha\in[0,1)$, problem \eqref{eq:phi-dro} reduces to Conditional Value-at-Risk minimization~\citep{rockafellar2000optimization}. When $\psi =0$ and $\phi(t) = t\log t-t+1$, problem \eqref{eq:phi-dro} reduces to Kullback-Leibler divergence DRO \citep{hu2013kullback}. While the algorithmic approach in \citep{levy2020large} only applies to the case where $\ell$ and $f_\theta$ are both smooth, our proposed smoothed PLDA is able to handle a wide range of loss functions, including nonsmooth losses such as the absolute loss.

\section{Preliminaries}
\label{sec:pre}
Let us introduce the basic problem setup and key concepts that will serve as the basis for our subsequent analysis. 
\begin{assumption}[Problem setup] \label{ass:basic}
The following assumptions on the objective function $F$ of problem \eqref{eq:problem} hold throughout the paper. 
\begin{enumerate}[label=\normalfont(\alph*)]
\item   $(\textrm{\bf Primal function})$ For all $y\in \mathcal{Y}$, the function $F(\cdot,y): \mathbb{R}^n \rightarrow \mathbb{R}$ takes the form $F(\cdot, y):=h_y\circ c_y(\cdot)$, where $c_y:\R^n\rightarrow \R^m$ is 
continuously differentiable
with $L_c$-Lipschitz continuous Jacobian mapping, i.e.,
\[
\|\nabla c_y(x)-\nabla c_y(x')\| \leq L_c\|x-x'\| \quad \text { for all } x, x' \in \mathcal{X};
\]
and $h_y:\R^m\rightarrow \R$  is  convex and $L_h$-Lipschitz continuous, i.e.,
\[
|h_y(z)-h_y(z')| \leq L_h\|z-z'\| \quad \text { for all } z, z' \in \R^{m}.
\]
\item  $(\textrm{\bf Dual function})$ For all $x\in \mathcal{X}$, the function $F(x, \cdot):\R^d \rightarrow \R$ is continuously differentiable  with $\nabla_y F(\cdot,\cdot)$ being $L$-Lipschitz continuous on $\mathcal{X}\times \mathcal{Y}$, i.e., 
\[
\|\nabla_y F(x,y)-\nabla_y F(x',y')\| \leq L\|(x,y)-(x',y')\| \,
\text{ for all }  (x,y), (x',y') \in \mathcal{X}\times\mathcal{Y}.
\]
Without loss of generality, we may take $L=L_h L_c$.
\end{enumerate}
\end{assumption}
% \vspace{-0.21in}
\begin{assumption}[K\L{} property with exponent $\theta$ for dual problem]
\label{ass:kl}
For all  $x\in\mathcal{X}$, the problem $\max_{y\in \mathcal{Y}} F(x,y)$ has a nonempty solution set and a finite optimal value. Moreover, there exist $\mu>0$ and $\theta \in [0,1)$ such that  
\[
\dist(0, -\nabla_y F(x,y)+\partial\iota_{\mathcal{Y}}(y)) \ge \mu  \left(\max\limits_{y'\in \mathcal{Y}} F(x,y')-F(x,y)\right)^{\theta}
\]
for all $x \in \mathcal{X}$ and  $y \in \mathcal{Y}\setminus \mathcal{Y}^\star(x)$, where $\mathcal{Y}^\star(x):= \argmax_{y' \in \mathcal{Y}} F(x,y')$ 
% satisfying $F(x,y)<\max_{y'\in\mathcal{Y}} F(x,y')$
.
\end{assumption}
\begin{remark}
The P\L{} property has been widely used as a standard assumption in nonconvex-nonconcave minimax optimization~\citep{nouiehed2019solving,yang2022faster}. However, its usage is restricted to smooth unconstrained settings. To overcome this limitation, we employ the K\L{} property, which has been demonstrated in \citep{attouch2010proximal} to be a nonsmooth extension of the P\L{} property. The K\L{} exponent, a crucial quantity in the convergence rate analysis of first-order methods for nonconvex optimization~\citep{frankel2015splitting,li2018calculus}, plays a vital role in establishing the explicit convergence rate of  smoothed PLDA.
\end{remark}

Let us now examine the stationarity measures discussed in this paper. We define the value function $f:\mathbb{R}^n \rightarrow \mathbb{R}$, the potential function $F_r:\R^n\times \R^d\times \R^n\rightarrow\R$, and the dual potential value function $d_r:\R^d\times \R^n\rightarrow \R$ by
% \[
% F_r(x,y,z) := F(x,y)+\frac{r}{2}\|x-z\|^2\quad \text{and}\quad d_r(y,z):=\min_{x\in\mathcal{X}} F_r(x,y,z),
% \]
\begin{align*}
f(x) := \max_{y \in \mathcal{Y}} F(x,y),\quad F_r(x,y,z) := F(x,y) + \frac{r}{2} \| x-z \|^2,\quad d_r(y,z) := \min_{x \in \mathcal{X}} F_r(x,y,z),
\end{align*}
respectively, where we assume $r> L$ in the rest of this paper.

As the function $F(\cdot,y)$ is weakly convex for each $y \in \mathcal{Y}$~\citep[Lemma 4.2]{drusvyatskiy2019efficiency},  the value function $f$ is also weakly convex. Thus, motivated by the development in \citep{davis2019stochastic}, we may use the measure in Definition \ref{defi:primal-dual}(a), which we call \emph{optimization stationarity}, as a (primal) stationarity measure for \eqref{eq:problem}. On the other hand, it is shown in~\citep[Lemma 4.3]{drusvyatskiy2019efficiency} that $\nabla_z d_r(y,x) = r(x-\prox_{\frac{1}{r}F(\cdot,y) + \iota_{\mathcal{X}}}(x))$, where $\prox$ denotes the proximal mapping (see Definition \ref{defi:prox} in Appendix \ref{sec:lemmas}).
Thus, we may use the measure in Definition \ref{defi:primal-dual}(b), which we call \emph{game stationarity}, as a (primal-dual) stationarity measure for \eqref{eq:problem}.

\begin{definition}[Stationarity measures]
\label{defi:primal-dual}
Let $\epsilon \ge 0$ be given.
\begin{enumerate}[label=\normalfont(\alph*)]
\item The point $x\in\mathcal{X}$ is an $\epsilon$-optimization-stationary point {\rm($\epsilon$-OS)} of problem \eqref{eq:problem} if 
\[\|\prox_{\frac{1}{r}f+\iota_{\mathcal{X}}}(x)-x\|\leq \epsilon.\]
\item The pair $(x, y)\in\mathcal{X}\times \mathcal{Y}$ is an $\epsilon$-game-stationary point {\rm ($\epsilon$-GS)} of problem \eqref{eq:problem} if 
% there exists a pair $(u, v)$ such that
\[
\|\nabla_{z} d_r(y,x)\|\leq\epsilon \quad \text{and} \quad {\rm dist}(0,-\nabla_{y} F(x, y)+\partial\iota_{\mathcal{Y}}(y))\leq\epsilon. \footnote{For any $x \in \mathbb{R}^{\ell}$ and $\mathcal{S} \subseteq \mathbb{R}^{\ell}$, we use $\dist(x, \mathcal{S}):=\inf\limits_{z\in\mathcal{S}}\|x-z\|$ to denote the distance from $x$ to $\mathcal{S}$.} 
% \quad \text { and } \quad\|u\|,\|v\| \leq \epsilon.
\]
\end{enumerate}
% where $1_{\mathcal{Y}}$ is the indicator function of the set $\mathcal{Y}$. 
% A pair $(x, y)$ is a stationary solution if $\mu = 0, \nu = 0$. 
\end{definition}
Since weakly convex functions and smooth functions are subdifferentially regular, we can utilize the Fr\'{e}chet subdifferential in the above definitions. A brief overview of  various subdifferential constructions in nonsmooth nonconvex optimization can be found in \citep{li2020understanding}. In Section \ref{sec:sta}, we will explore the quantitative relationship between the two stationarity measures in Definition \ref{defi:primal-dual}.

\section{Proposed Algorithm --- Smoothed PLDA}
\label{sec:main}

GDA is a commonly used method for solving smooth nonconvex-concave  problems. However, its vanilla implementation can lead to oscillations, and a conventional approach to mitigate this issue is to use diminishing step sizes. Unfortunately, one cannot achieve the optimal iteration complexity with this approach. In view of this, the work \citep{zhang2020single} proposes a Nesterov-type smoothing technique that stabilizes the primal sequence and achieves a better balance between primal and dual updates. However, the technique crucially relies on the smoothness of the objective function.

% Due to the nonsmoothness of $h_y$, applying the smoothing technique to structured nonsmooth problems poses significant difficulties. 
The nonsmoothness in our problem setting poses a key challenge to algorithm design. To overcome this challenge, we fully leverage the composite structure of the primal function and adopt the proximal linear scheme for the primal update~\citep{drusvyatskiy2019efficiency}. Specifically, given $r > L$ and $\lambda > 0$, consider the update
\begin{equation}
    \label{eq:primal_update}
    x^{k+1} \in \mathop{\argmin}_{x\in\mathcal{X}} \left\{ F_{x^k,\lambda}(x,y^k)
  +\frac{r}{2}\|x-z^{k}\|^2
    \right\},
\end{equation}  
% \Jiajin{Decide whether we change $F$ to $\ell$ or not.}
where 
\[
F_{x^k,\lambda}(x,y^k):=h_{y^k}\left(c_{y^k}(x^{k})+\nabla c_{y^k}(x^{k})^\top(x-x^{k})\right)+\frac{1}{2\lambda}\|x-x^k\|^2.
\]
Here, $\{z^k\}$ is an auxiliary sequence. In particular, when $h(t) = t$, the primal update \eqref{eq:primal_update} reduces to the standard gradient descent step as described in~\citep{zhang2020single}. The  dual update and other smoothing steps are the same as those in the smoothed GDA developed in  \citep{zhang2020single}.

Our smoothed PLDA algorithm is formally presented in Algorithm \ref{alg}. 

\begin{algorithm}[H]\label{alg}
    \caption{\textbf{Smoothed PLDA}}
    \SetKwInOut{Input}{Input}
    \Input{Initial point $(x^0,y^0,z^0)$ and parameters $r>L$, $\lambda> 0$, $\alpha>0$, $\beta\in(0,1)$}
    \For{$k=0,1,2,\ldots$}{
    $x^{k+1}:= \mathop{\argmin}\limits_{x \in \mathcal{X}} \left\{F_{x^k,\lambda}(x,y^k) + \frac{r}{2}\|x-z^k\|^2\right\}$\\
    $y^{k+1}:=\proj_{\mathcal{Y}}\left(y^k+\alpha \nabla_{y} F(x^{k+1}, y^k)\right)$ \\
    $z^{k+1}:=z^k+\beta(x^{k+1}-z^k)$}
\end{algorithm}

\begin{remark}[Primal update]
\label{rmk:primal}
Finding a closed-form solution for the primal update \eqref{eq:primal_update} is challenging due to the nonsmooth composite structure. However, problem \eqref{eq:primal_update} is strongly convex. If it in addition possesses certain error bound property, then an $\epsilon$-optimal solution can be found by suitable first-order methods in at most $\mathcal{O}(\log(\epsilon^{-1}))$ iterations; see~\citep{zhou2017unified} and~\citep[Chapter 8]{cui22modern} for details and further references. In general, the lower bound on the iteration complexity of first-order methods for solving nonsmooth strongly convex problems is $\Omega(\epsilon^{-1})$~\citep[Theorem 3.13]{bubeck2015convex}. 
We emphasize that the optimal approach to solving \eqref{eq:primal_update} depends on the problem's specific structure. For simplicity and to focus on our main contributions, we assume the exactness of the primal update and concentrate on the iteration complexity of the outer loop in subsequent analyses.
\end{remark}
% However, we can observe that problem \eqref{eq:primal_update} is strongly convex, thus it is still possible to apply some first-order iterative algorithms that are able to achieve a linear convergence rate of $\mathcal{O}(\log(\frac{1}{\epsilon}))$. For instance, if $h_y(\cdot) = \|\cdot\|_{p}$ with $p\in\{1,\infty\}$, problem \eqref{eq:primal_update} can be rephrased  in the following form:
% \begin{equation}
% \min_{x\in \mathcal{X}} \sum_{i=1}^N \|A_ix + b_i\|_p + \frac{\tau}{2}\|x-\bar{x}\|^2.
% \end{equation}
% Here, $A_i$, $b_i$, and $\bar{x}$ are known and $\tau>0$ is a constant. We then derive the dual problem as follows:
% \begin{equation}
% \label{eq:dual_sub}
% \max_{\|w_i\|q \leq 1 } \sum_{i=1}^N w_i^\top(A_i\bar{x}+b_i) -\frac{1}{2\tau} \left\|\sum_{i=1}^N A_i^\top w_i\right\|^2.
% \end{equation}
% % \Jiajin{notation issue: $N$ and $T$?}
% Here, $\frac{1}{p}+\frac{1}{q}=1$.  
% Interestingly, the dual problem possesses the Luo-Tseng error bound condition, as shown in \citep{zhou2017unified}. Consequently, we can apply projected gradient descent and enjoy a linear convergence rate. Although we cannot obtain the exact solution to the primal problem, this inner minimization exerts minimal influence on the overall iteration complexity. 

\section{Convergence Analysis of Smoothed PLDA}
\label{sec:analysis}
In this section, we present the  main theoretical contributions of this paper.
Our goal is to investigate the convergence rate of smoothed PLDA (see Theorem \ref{thm:general}) under different settings, including nonconvex-K\L{} and nonconvex-concave.

Before proceeding, let us fix the notation as in Table \ref{tab:notation}. 

% \Lingzhi{TBC}

\begin{table*}[ht]
  \caption{Notation}
 \centering
\begin{tikzpicture}
 \node[drop shadow,fill=white,inner sep=0pt] 
 {
 \resizebox{0.98\textwidth}{!}{ 
\begin{tabular}{|c|c|c|}
 \hline
 \textbf{Notation}  & \textbf{Definition} & \textbf{Notes} \\ \hline
 $F_r(x,y,z)$   & $F(x,y)+\frac{r}{2}\|x-z\|^2$ & potential function               \\ \hline
$d_r(y,z)$           & $\min\limits_{x \in \mathcal{X}}F_r(x,y,z)$ &    dual potential value function      \\ \hline
$p_r(z)$           & $\max\limits_{y \in \mathcal{Y}}\min\limits_{x \in \mathcal{X}}F_r(x,y,z)$ &    proximal function        \\ \hline
%  \rowcolor[HTML]{ECF4FF} 
$x_r(y,z)$  & $\mathop{\argmin}\limits_{x \in \mathcal{X}}F_r(x,y,z)$  & ---\\ \hline
$x^\star_r(z)$   & $\mathop{\argmin}\limits_{x \in \mathcal{X}}\max\limits_{y \in \mathcal{Y}}F_r(x,y,z)$ & ---               \\ \hline
$Y(z)$     & $\mathop{\argmax}\limits_{y\in \mathcal{Y}}d_r(y,z)$ &   $y(z) \in Y(z)$  \\ \hline
%  \rowcolor[HTML]{ECF4FF} 
$y_{+}(z)$    & $\proj_{\mathcal{Y}}\left(y+\alpha \nabla_{y} F(x_r(y, z), y)\right)$  & one-step projected gradient ascent on the dual function\\ \hline
\end{tabular}}
};
\end{tikzpicture}
\label{tab:notation}
\end{table*}
\subsection{Analysis Framework}
\label{sec:framework}
We start by presenting the key ideas in the analysis.
 \begin{tcolorbox}
\textbf{Step 1:} Construct a \textit{Lyapunov function} that 
possesses the sufficient decrease property.
\end{tcolorbox}
Building on \citep{zhang2020single}, we define the Lyapunov function $\Phi_r:\R^n\times \R^d\times \R^n\rightarrow\R$ as follows:
\[
\Phi_r(x,y,z):= \underbrace{F_r(x,y,z)-d_r(y,z)}_{\textnormal{Primal Descent}} + \underbrace{p_r(z)-d_r(y,z)}_{\textnormal{Dual Ascent}} + \underbrace{p_r(z)}_{\textnormal{Proximal Descent}}. 
\]
The sole distinction between the above Lyapunov function and the one in \citep{zhang2020single} is that we swap the order of min and max in the definition of $p_r$, transitioning from $\min_{x \in \mathcal{X}}\max_{y \in \mathcal{Y}}F_r(x,y,\cdot)$ to $\max_{y \in \mathcal{Y}}\min_{x \in \mathcal{X}}F_r(x,y,\cdot)$. This alteration enables a more comprehensive analysis of cases where the dual function is nonconcave. Each term in the potential function $\Phi_r$ is tightly connected to the algorithmic updates. Specifically, the updates for the primal, dual, and auxiliary variables can be viewed as a descent step on the primal function $F_r$, an approximate ascent step on the dual potential value function $d_r$, and an approximate descent step on the proximal function $p_r$, respectively.

Next, our task is to quantify the change in each term of the Lyapunov function after one round of updates. To do so, the key step is the following:
\begin{tcolorbox}
\textbf{Step 2:} Establish a primal error bound (i.e., Proposition \ref{prop:lip}) to quantify the primal descent.
\end{tcolorbox}
Intuitively, the primal error bound provides an estimate of the distance between the current point $x^{k+1}$ and the optimal solution \[x_r(y^k,z^k) = \mathop{\argmin}\limits_{x\in \mathcal{X}} F_r(x,y^k,z^k)\] in terms of the iterates gap $\|x^{k+1}-x^k\|$ that results from the proximal linear scheme.

% \vspace{1ex}
Equipped with the primal error bound, we can obtain the following basic descent estimate of the Lyapunov function:

\begin{proposition}[Basic descent estimate of $\Phi_r$]
\label{prop:decrease}
    Let 
    % $r\ge3L$, $\lambda\le \frac{1}{L}$, $\alpha \leq \min\left\{\frac{1}{10L}, \frac{1}{4L\zeta^2} \right\}$, and $\beta  \leq \min \left\{\frac{1}{28}, \frac{(r-L)^2}{14\alpha r(2r-L)^2}\right\}$,     
    \[r\ge3L,\ \lambda\le \frac{1}{L},\ \alpha \leq \min\left\{\frac{1}{10L},\ \frac{1}{4L\zeta^2} \right\},\ \beta  \leq \min \left\{\frac{1}{28},\ \frac{(r-L)^2}{14\alpha r(2r-L)^2}\right\},\] 
    where $\zeta>0$ is the constant defined in Proposition \ref{prop:lip}.
    Then, for any $k\ge0$, we have
\begin{equation*}
\begin{aligned}
\Phi_r^k-\Phi_r^{k+1}\ge\ &  \frac{3}{8\lambda}\|x^{k}-x^{k+1}\|^{2}+\frac{1}{8\alpha}\|y^{k}-y_{+}^{k}(z^{k+1})\|^2 +\frac{4r}{7\beta}\|z^{k}-z^{k+1}\|^{2} \\
& - 14r\beta\|x_r(y(z^{k+1}), z^{k+1})-x_r(y_{+}^{k}(z^{k+1}), z^{k+1})\|^{2},
\end{aligned}
\end{equation*}
where $\Phi_r^k:= \Phi_r(x^k,y^k,z^k)$. 
% \footnote{The proof of Proposition \ref{prop:decrease}, which is based on the primal descent, dual ascent, and proximal descent properties, is given in Appendix \ref{sec:suff_decrease}. }
\end{proposition}
The proof of Proposition \ref{prop:decrease}, which is based on the primal descent, dual ascent, and proximal descent properties, is given in Appendix \ref{sec:suff_decrease}.

Now, we proceed to bound the negative term 
\[
\|x_r(y(z^{k+1}), z^{k+1})-x_r(y_{+}^{k}(z^{k+1}), z^{k+1})\|
\]
in the basic descent estimate using some positive terms, so as to establish the sufficient decrease property.
To achieve this goal, we explicitly quantify the primal-dual relationship by a dual error bound,  which will allow us to bound the primal update $\|x_r(y(z^{k+1}), z^{k+1})-x_r(y_{+}^{k}(z^{k+1}), z^{k+1})\|$ by the dual update $\|y^{k}-y_{+}^{k}(z^{k+1})\|$.
% hinders the achievement of the sufficient decrease  To tackle this challenge, we aim to bound $\|x_r(y(z^{k+1}), z^{k+1})-x_r(y_{+}^{k}(z^{k+1}), z^{k+1})\|$ using some positive terms.

% This term is associated with $\|y^{k}-y_{+}^{k}(z^{k+1})\|$. 
% Specifically, if $\|y^{k}-y_{+}^{k}(z^{k+1})\| =0$, then $y^{k}$ represents the optimal solution of $\max_{y \in \mathcal{Y}} d_r(y,z^{k+1})$, and consequently, $x_r(y(z^{k+1}), z^{k+1})= x_r(y^k,z^{k+1})= x_r(y_{+}^{k}(z^{k+1}), z^{k+1}) $. 
\begin{tcolorbox}
\textbf{Step 3:} Establish a dual error bound (i.e., Proposition \ref{prop:dual_eb_KL}) to show that the primal and dual updates are balanced in the sense that
\begin{equation}
\label{eq:step_dual}
\|x_r(y(z^{k+1}), z^{k+1})-x_r(y_{+}^{k}(z^{k+1}), z^{k+1})\| \leq \omega \|y^{k}-y_{+}^{k}(z^{k+1})\|^\upsilon
\end{equation}
for some constants $\omega>0$ and $\upsilon>0$. 
\end{tcolorbox}
We will demonstrate later that the growth power $\upsilon$ in \eqref{eq:step_dual} is directly determined by the K\L{} exponent  of the dual problem.

The dual error bound allows us to establish the global convergence rate of smoothed PLDA. To illustrate the main idea, we consider two distinct regimes:
\begin{enumerate}[label=\normalfont(\alph*)]
\item When $\upsilon\in [1,\infty)$, we can show that
\begin{align*}
\|x_r(y(z^{k+1}), z^{k+1})-x_r(y_{+}^{k}(z^{k+1}), z^{k+1})\|
& \leq \, \omega \|y^{k}-y_{+}^{k}(z^{k+1})\|^{\upsilon}\\
&\leq\, \omega\cdot {\rm diam}(\mathcal{Y})^{\upsilon-1} \cdot \|y^{k}-y_{+}^{k}(z^{k+1})\|.\notag
\end{align*}
Owing to the above \textbf{homogeneous} error bound, there exist constants $a,b,c>0$  such that
\begin{align*}
&\Phi_r^k-\Phi_r^{k+1} 
\ge\ a \|x^{k}-x^{k+1}\|^{2}+b\|y^{k}-y_{+}^{k}(z^{k+1})\|^2 +c\|z^{k}-z^{k+1}\|^{2}.
\end{align*}
Subsequently, it becomes straightforward to achieve an $\mathcal{O}(\epsilon^{-2})$ iteration complexity.
\item When $\upsilon \in (0,1)$, we encounter an \textbf{inhomogeneous} error bound  instead. If the negative term $\|x_r(y(z^{k+1}), z^{k+1})-x_r(y_{+}^{k}(z^{k+1}), z^{k+1})\|$ is large, then we can still establish the sufficient decrease property by choosing a sufficiently small $\beta$. The final iteration complexity will explicitly depend on $\upsilon$.
\end{enumerate}
The details of Step 3 can be found in  Subsection \ref{subsec:iter}.

% \vspace{-2mm}
\subsection{Primal Error Bound}
\label{sec:primal-erb}
In this subsection, we establish the primal error bound property of smoothed PLDA. Recall that for the smoothed GDA in \citep{zhang2020single}, its primal error bound property essentially follows from the standard Luo-Tseng error bound  for structured strongly convex problems~\citep{luo1993error,zhou2017unified,zhang2020proximal}. Specifically, it is shown in \citep[Lemma B.2]{zhang2020single} that
\[ 
 \|x^{k+1}-x_r(y^{k}, z^{k})\|\leq\zeta\|x^k-\underbrace{\proj_{\mathcal{X}}(x^k-\lambda \nabla_x F_r(x^k,y^k,z^k)}_{x^{k+1}})\|,
\]
where $\lambda>0$ is the step size for primal descent and $\zeta>0$ is a constant.
However, in our problem setting, the primal function does not satisfy the gradient Lipschitz continuity condition. One of our primary contributions is to demonstrate that even so, a Lipschitz-type primal error bound still holds. The primal error bound is crucial to establishing the sufficient decrease property of the Lyapunov function we mentioned earlier.

\begin{proposition} [Lipschitz-type primal error bound]
\label{prop:lip}
%  Suppose that $r-L > 0$. Then 
 For any $k\ge0$, we have
    \begin{equation}\label{iter-result1}
        \|x^{k+1}-x_r(y^{k}, z^{k})\| \leq \zeta\|x^{k}-x^{k+1}\|,
        \end{equation}
where   $\zeta: =\frac{2(r-L)^{-1}+(\lambda^{-1}+L) ^{-1}}{(\lambda^{-1}+L) ^{-1}} \left(\sqrt{\frac{2L}{\lambda^{-1}+L}}+1\right)$.
% and $\eta:=\alpha L \zeta$. 
\end{proposition}
\begin{proof}
For ease of notation, we define $\hat{F}_{y,z}(\cdot): = F_r(\cdot,y,z) + \iota_{\mathcal{X}}(\cdot)$ for any $y\in\mathcal{Y}$ and $z\in\R^n$.  As $x_r(y,z)$ is the optimal solution of $\min_{x \in \mathbb{R}^n} \hat{F}_{y,z} (x)$ and $F_r(\cdot,y,z)$ is $(r-L)$-strongly convex (see the remark after Fact \ref{fact-2} in Appendix \ref{sec:lemmas}), we have
\begin{equation}
    \label{eq:lip_1} 
    \hat{F}_{y^k,z^k}(x)-\hat{F}_{y^k,z^k}(x_r(y^k,z^k)) \ge \frac{r-L}{2} \|x-x_r(y^k,z^k)\|^2 \quad \text{for all}\ x \in \mathcal{X}.
\end{equation}
In addition, by the convexity of $\hat{F}_{y^k,z^k}$, we obtain 
\begin{equation}
    \label{eq:lip_2}
    \begin{aligned}
    \hat{F}_{y^k,z^k}(x)-\hat{F}_{y^k,z^k}(x_r(y^k,z^k)) & \leq \dist(0, \partial \hat{F}_{y^k,z^k}(x))\cdot \|x-x_r(y^k,z^k)\| \\ 
    \end{aligned}
\end{equation}
$\text{for all}\ x\in\mathcal{X}$.
Combining \eqref{eq:lip_1} and \eqref{eq:lip_2} yields
\[
\|x-x_r(y^k,z^k)\| \leq \frac{2}{r-L} \dist(0, \partial \hat{F}_{y^k,z^k}(x)) \quad \text{for all}\ x\in\mathcal{X}. 
\]
By utilizing the equivalence between proximal and subdifferential error bounds \citep[Theorem 3.4]{drusvyatskiy2018error}, we  conclude that
\begin{equation}
\label{eq:;lip_3}
    \|x-x_r(y^k,z^k)\|  \leq \frac{2(r-L)^{-1}+t}{t} \left\|x - \prox_{t  \hat{F}_{y^k,z^k}}(x)\right\| \text{ for all } t>0, \, x\in\mathcal{X}.
\end{equation}
Now, let us examine the relationship between  $\|x^{k+1}-x^k\|$ and $\|x^{k+1} - \prox_{t \hat{F}_{y^k,z^k}}(x^{k+1})\|$. Let $\varphi_k:\R^n\rightarrow \R \cup \{+\infty\}$ be the function defined by
\[
\varphi_k(x):=\hat{F}_{y^k,z^k}(x)+\frac{\lambda^{-1}+L}{2}\|x-x^k\|^{2}-\frac{\lambda^{-1} + L}{2}\|x-x^{k+1}\|^{2}.
\]
% By combining Fact \ref{fact-2} and the strong convexity of $F_{x^k,\lambda}(\cdot,y^k)$
It is clear that for any $x\in\mathcal{X}$,
\begin{align*}\label{safeguard-key1}
            & \hat{F}_{y^k,z^k}(x)
             = F(x,y^k) + \frac{r}{2}\|x-z^k\|^2 \notag\\
             \ge\ & F_{x^k,\lambda}(x,y^k) -\frac{\lambda^{-1}+L}{2}\|x-x^{k}\|^2 + \frac{r}{2}\|x-z^k\|^2\notag \\ 
              \ge\ &  F_{x^k,\lambda}(x^{k+1},y^k) +\frac{\lambda^{-1}+r}{2}\|x-x^{k+1}\|^2 -\frac{\lambda^{-1}+L}{2}\|x-x^{k}\|^2 +\frac{r}{2}\|x^{k+1}-z^k\|^2\notag\\
             \ge\ & F(x^{k+1},y^k) + \frac{\lambda^{-1}-L}{2}\|x^{k}-x^{k+1}\|^2 +\frac{\lambda^{-1}+r}{2}\|x-x^{k+1}\|^2\ \\
             & \ -\frac{\lambda^{-1}+L}{2}\|x-x^{k}\|^2 + \frac{r}{2}\|x^{k+1}-z^k\|^2 \notag \\ 
             =\ & \hat{F}_{y^k,z^k}(x^{k+1})+\frac{\lambda^{-1}-L}{2}\|x^{k+1}-x^k\|^2 +\frac{\lambda^{-1}+r}{2}\|x-x^{k+1}\|^2 -\frac{\lambda^{-1}+L}{2}\|x-x^{k}\|^2,
\end{align*}
where the first and third inequalities follow from Fact \ref{fact-2}, and the second one is from the $(\lambda^{-1} + r)$-strong convexity of $F_{x^k,\lambda}(\cdot,y^k) +\frac{r}{2}\|\cdot-z^k\|^2$ on $\mathcal{X}$ and the definition of $x^{k+1}$.
Since $r > L$, we see that for any $x\in\mathcal{X}$,
\begin{align*}
\varphi_k(x)
&\ge \hat{F}_{y^k,z^k}(x^{k+1})+\frac{\lambda^{-1}-L}{2}\|x^{k+1}-x^{k}\|^2+\frac{r-L}{2}\|x-x^{k+1}\|^2\notag\\
&\ge\hat{F}_{y^k,z^k}(x^{k+1})+\frac{\lambda^{-1}-L}{2}\|x^{k+1}-x^{k}\|^2.
\end{align*}
It then follows from the definition of $\varphi_k$ that 
\[\varphi_k(x^{k+1})-\inf_{x\in\R^n} \varphi_k(x)\leq L \|x^{k+1}-x^k\|^2.\]
This, together with the definition of the proximal mapping, implies that for any $\rho>0$,
\begin{align}
&\varphi_k\left(\operatorname{prox}_{\frac{1}{\rho} \varphi_k}(x^{k+1})\right)+\frac{\rho}{2}\|\operatorname{prox}_{\frac{1}{\rho} \varphi_k}(x^{k+1})-x^{k+1}\|^{2}\leq  \varphi_k(x^{k+1})\notag\\
\leq\ & \inf_{x\in\R^n} \varphi_k(x)+L \|x^{k+1}-x^k\|^2
\leq \varphi_k\left(\operatorname{prox}_{\frac{1}{\rho} \varphi_k}(x^{k+1})\right)+L \|x^{k+1}-x^k\|^2,\notag
\end{align}
which in turn implies that
\begin{equation}\label{safeguard-key111}
\|\operatorname{prox}_{\frac{1}{\rho} \varphi_k}(x^{k+1})-x^{k+1}\|\leq \sqrt{\frac{2L}{\rho}} \|x^{k+1}-x^k\|.    
\end{equation}
Recall that our goal is to find a relationship between  $\|x^{k+1}-x^k\|$ and $\|x^{k+1} - \prox_{t \hat{F}_{y^k,z^k}}(x^{k+1})\|$. To do this, it suffices to relate $\prox_{t \hat{F}_{y^k,z^k}}(x^{k+1})$ to $\operatorname{prox}_{\frac{1}{\rho} \varphi_k}(x^{k+1})$.  By definition, we have
\[\hat{F}_{y^k,z^k}(x)+\frac{\lambda^{-1}+L}{2}\|x-x^{k}\|^{2}=\varphi_k(x)+\frac{\lambda^{-1}+L}{2}\|x-x^{k+1}\|^{2},\]
which implies that $\operatorname{prox}_{\frac{1}{\lambda^{-1}+L} \varphi_k}(x^{k+1})=\operatorname{prox}_{\frac{1 }{\lambda^{-1}+L}\hat{F}_{y^k,z^k}}(x^k)$. 
By setting $t = (\lambda^{-1} + L)^{-1}$ in \eqref{eq:;lip_3} and $\rho = \lambda^{-1} + L$ in \eqref{safeguard-key111}, we conclude that
\begin{align*}
    &\|x^{k+1}-x_r(y^{k}, z^{k})\| \notag\\
    \leq\ & \frac{2(r-L)^{-1}+(\lambda^{-1}+L) ^{-1}}{(\lambda^{-1}+L) ^{-1}}\left\|x^{k+1}- \prox_{\frac{1}{\lambda^{-1}+L} \hat{F}_{y^k,z^k}}(x^{k+1})\right\| \\
    \leq\ &   \frac{2(r-L)^{-1}+(\lambda^{-1}+L) ^{-1}}{(\lambda^{-1}+L) ^{-1}}\left(\left\|x^{k+1}- \prox_{\frac{1}{\lambda^{-1}+L} \hat{F}_{y^k,z^k}}(x^{k})\right\|+ \|x^{k+1}-x^k\|\right) \\
     =\ &   \frac{2(r-L)^{-1}+(\lambda^{-1}+L) ^{-1}}{(\lambda^{-1}+L) ^{-1}}\left(\left\|x^{k+1}- \prox_{\frac{1}{\lambda^{-1}+L} \varphi_k}(x^{k+1})\right\|+ \|x^{k+1}-x^k\|\right) \\
    \leq\ & \frac{2(r-L)^{-1}+(\lambda^{-1}+L) ^{-1}}{(\lambda^{-1}+L) ^{-1}} \left(\sqrt{\frac{2L}{\lambda^{-1}+L}}+1\right) \|x^{k+1}-x^k\|,
\end{align*}
where the second inequality is due to the nonexpansiveness of the proximal mapping. The proof is complete.
\end{proof}

\begin{remark}
We can also establish the primal error bound \eqref{iter-result1} using \citep[Theorem 5.3 and Theorem 5.10]{drusvyatskiy2018error}. However, since these two theorems are proved using Ekeland's variational principle, the resulting constant $\zeta$ will be larger. Our approach  to establishing \eqref{iter-result1} is more elementary and yields a smaller constant $\zeta$. It is worth noting that the  constant $\zeta$ obtained in Proposition \ref{prop:lip} plays a crucial role in controlling the step sizes for both the primal and dual updates.
\end{remark}

\subsection{Dual Error Bound}
\label{sec:dual-error}
After establishing the primal error bound and verifying the basic descent estimate in Proposition \ref{prop:decrease} (which completes Step 2), we proceed to use the K\L{} exponent  of the dual problem to derive a new dual error bound. Such an error bound provides a crucial relationship between the primal and dual updates and reveals an interesting phase transition phenomenon. Moreover, it furnishes an effective and theoretically justified approach for balancing the primal and dual updates and allows us to establish the global convergence rate of smoothed PLDA. We should point out that our use of the K{\L} exponent represents a marked departure from typical approaches, which only focus on primal nonconvex optimization~\citep{attouch2013convergence,li2018calculus}.
% The dual error bound property plays a critical role in analyzing the explicit convergence rate or iteration complexity of the proposed smoothed PLDA algorithm. 

\begin{proposition}[Dual error bound with K\L{} exponent] 
\label{prop:dual_eb_KL}
% Let $\delta>0$. 
Suppose that Assumption \ref{ass:kl} holds. Then, for any $y\in\mathcal{Y}$ and $z \in \mathbb{R}^n$, we have
% \[
% \|x_r^\star(z)-x_r(y_+(z),z)\| \leq \omega \|y-y_+(z)\|^{\frac{1}{2\theta}},
% \]
% where $\omega:=\frac{\sqrt{2}}{\sqrt{r-L}}\left(\frac{1+\alpha L(1+\sigma_2)}{\alpha\mu}\right)^{\frac{1}{2\theta}}$.
\begin{enumerate}[label=\normalfont(\alph*)]
 \item \rm{(K\L\ exponent $\theta=0$):} $\|x_r(y(z),z)-x_r(y_+(z),z)\| \leq \omega_1 \|y-y_+(z)\|$, 
  \item \rm{(K\L\ exponent $\theta\in(0,1)$):} $\|x_r(y(z),z)-x_r(y_+(z),z)\| \leq \omega_2\|y-y_+(z)\|^{\frac{1}{2\theta}}$, 
  \end{enumerate}
    where $\omega_1:= \frac{\sqrt{2L\cdot {\rm diam}(\mathcal{Y})}\cdot (1+\alpha L)}{\alpha\mu\sqrt{r-L}}$ and $\omega_2:=\frac{\sqrt{2}}{\sqrt{r-L}}\left(\frac{1+\alpha L(1+\sigma_2)}{\alpha\mu}\right)^{\frac{1}{2\theta}}$ (recall that $y_+(z) = \proj_{\mathcal{Y}} ( y + \alpha \nabla_y F(x_r(y,z), y)$).  
\end{proposition}
\begin{proof}
Let $\psi:\R^n\times \R^n\rightarrow\R$ be the function defined by
\[
\psi(x,z) =  F_r(x,y(z),z). 
\]
Consider arbitrary $x \in \mathcal{X}$, $y \in \mathcal{Y}$, and $z \in \mathbb{R}^n$. Note that the function $\psi(\cdot,z)$ is $(r-L)$-strongly convex. Since
\begin{align*}
\mathop{\argmin}_{x'\in\mathcal{X}}\psi(x',z) & =\mathop{\argmin}_{x'\in\mathcal{X}} F_r(x',y(z),z) =  x_r(y(z),z),
\end{align*}
we see that 
% for any $x\in \mathcal{X}$, 
\begin{equation}
\label{eq:dual1}
\psi(x,z) - \psi(x_r(y(z),z),z) \ge \frac{r-L}{2}\|x-x_r(y(z),z)\|^2. 
\end{equation}
In addition, we have
\begin{align} 
\label{eq:dual2}
&\psi(x,z) - \psi(x_r(y(z),z),z)\notag\\   
\leq \  &\psi(x,z)-F_r(x_r(y_+(z),z),y_+(z),z)\notag \\
\leq \ & \max_{y^\prime\in \mathcal{Y}} F(x,y^\prime) + \frac{r}{2}\|x-z\|^2-F_r(x_r(y_+(z),z),y_+(z),z)\notag\\
=\,& \max_{y^\prime\in \mathcal{Y}}F(x,y^\prime)-F(x_r(y_+(z),z),y_+(z)) + \frac{r}{2}\|x-z\|^2-\frac{r}{2}\|x_r(y_+(z),z)-z\|^2,
\end{align}
where the first inequality follows from
\begin{align*}
F_r(x_r(y_+(z),z),y_+(z),z) 
=\ & \min_{x'\in\mathcal{X}}\left\{ F(x',y_+(z)) +\frac{r}{2}\|x'-z\|^2\right\} \nonumber \\
\leq\ & \max_{y' \in \mathcal{Y}}\min_{x'\in\mathcal{X}}\left\{F(x',y') +\frac{r}{2}\|x'-z\|^2\right\} \nonumber \\
=\ & \min_{x'\in\mathcal{X}}\left\{F(x',y(z)) +\frac{r}{2}\|x'-z\|^2\right\} \nonumber \\
=\ & \min_{x'\in\mathcal{X}} \psi(x',z)
= \psi(x_r(y(z),z),z). % \label{eq:Fr-one}
\end{align*}
% \Jiajin{We may have to double-check the above derivation. --- confirmed. With the uniform KL property, we can confirm  $x(y(z),z)$ is unique over $Y(z)$.}
As \eqref{eq:dual1} and \eqref{eq:dual2} hold for any $x\in\mathcal{X}$, we obtain the intermediate relation
\begin{equation}\label{OS-GS-better-key3-main} 
% \begin{aligned}
     \frac{r-L}{2}\|x_r(y(z),z)-x_r(y_+(z),z)\|^2\leq  \max\limits_{y^\prime\in \mathcal{Y}} F(x_r(y_+(z),z),y^\prime) - F(x_r(y_+(z),z),y_+(z)) 
% \end{aligned}
\end{equation}
by taking $x= x_r(y_+(z),z)$.

Now, we utilize the K{\L} exponent $\theta$ of the dual problem to bound the right-hand side of 
\eqref{OS-GS-better-key3-main}  in terms of $\|y-y_+(z)\|$. Consider the following two cases:
 \begin{enumerate}[label=\normalfont(\alph*)]
 \item Suppose that $\theta=0$. If $F(x_r(y_+(z),z),y_+(z))=\max_{y'\in\mathcal{Y}} F(x_r(y_+(z),z),y')$, then the desired inequality follows trivially from~\eqref{OS-GS-better-key3-main}. Otherwise, we have $y_+(z) \in \mathcal{Y} \setminus \mathcal{Y}^\star(x_r(y_+(z),z))$. By Assumption~\ref{ass:kl}, we have
% we deduce using~\eqref{eq:Fr-one} that
% \begin{align*}
% F_r(x_r(y_+(z),z),y_+(z),z)
% &=\max_{y'\in\mathcal{Y}} F_r(x_r(y_+(z),z),y',z)\\
% &\ge F_r(x_r(y_+(z),z),y_+(z),z)\\
% &\ge F_r(x_r(y(z),z),y(z),z).
% \end{align*}
%By definition of $y(z)$, it follows that $y_+(z)=y(z)$, and consequently the conclusion trivially holds. 
\[
\dist(0, -\nabla_y F(x_r(y_+(z),z),y_+(z))+\partial\iota_{\mathcal{Y}}(y_+(z)))\ge \mu.
\]
Since  $F(x,\cdot)$ is continuously differentiable on the compact set $\mathcal{Y}$, $\|\nabla_y F(x,\cdot)\|$ is bounded on $\mathcal{Y}$. Without loss of generality, we can assume that  $\|\nabla_y F(x,\cdot)\|\leq L$ on $\mathcal{Y}$. According to  the mean value theorem, we have 
\[
\max\limits_{y^\prime\in \mathcal{Y}} F(x_r(y_+(z),z),y^\prime) - F(x_r(y_+(z),z),y_+(z))
\leq L \cdot {\rm diam}(\mathcal{Y}).
\]
Using the fact that $y_+(z) = \proj_{\mathcal{Y}} (y + \alpha \nabla_y F(x_r(y,z), y))$, we bound
\begin{align*}
 &\max\limits_{y^\prime\in \mathcal{Y}} F(x_r(y_+(z),z),y^\prime) - F(x_r(y_+(z),z),y_+(z))\notag\\
\leq\ & \frac{L \cdot {\rm diam}(\mathcal{Y})}{\mu^2}\cdot\dist^2(0, -\nabla_y F(x_r(y_+(z),z),y_+(z))+\partial\iota_{\mathcal{Y}}(y_+(z)))\notag\\ 
% \stackrel{\text{\ding{172}}}{\leq} \ &   \frac{L \cdot {\rm diam}(\mathcal{Y})}{\mu^2} b^2 \|\proj_{\mathcal{Y}}(y+\alpha\nabla_y F(x_r(y_+(z),z),y))-y\|^2
% \notag \\
% \leq \ &  \frac{ L b^2 \cdot {\rm diam}(\mathcal{Y})}{\mu^2}(
% \|\proj_{\mathcal{Y}}(y+\alpha\nabla_y F(x_r(y_+(z),z),y))-y_+(z)\|\ +\|y_+(z)-y\|)^2 \\ 
% \leq \ & \frac{L b^2 \cdot {\rm diam}(\mathcal{Y})}{\mu^2}(\alpha L \|x_r(y_+(z),z)-x_r(y,z)\| + \|y-y_+(z)\|)^2\\
\leq \ & \frac{L b^2 \cdot {\rm diam}(\mathcal{Y})}{\mu^2}\|y_+(z)-y\|^2,
\end{align*}
where $b:=\frac{1}{\alpha} + L > 0$ and the last inequality is due to the fact that the projected gradient ascent method satisfies the so-called \textit{relative error condition}
\citep[Section 5]{attouch2013convergence}. 
It follows that
\[
\|x_r(y(z),z)-x_r(y_+(z),z)\| \leq\frac{\sqrt{2L\cdot {\rm diam}(\mathcal{Y})}\cdot (1+\alpha L)}{\alpha\mu\sqrt{r-L}}\|y-y_+(z)\|. 
\]
\item If $\theta \in (0,1)$, then we have
\begin{align*}
% \label{KL-key-01}
& \mu\left(\max\limits_{y^\prime\in \mathcal{Y}} F(x_r(y_+(z),z),y^\prime) - F(x_r(y_+(z),z),y_+(z))\right)^{\theta}\notag\\
\leq \ & \dist(0, -\nabla_y F(x_r(y_+(z),z),y_+(z))+\partial\iota_{\mathcal{Y}}(y_+(z))) \notag\\ 
\leq \ & \dist(0, -\nabla_y F(x_r(y,z),y_+(z)) +\partial\iota_{\mathcal{Y}}(y_+(z)))   \notag\\
\ & + \|\nabla_y F(x_r(y,z),y_+(z))-\nabla_y F(x_r(y_+(z),z),y_+(z))\|\notag\\ 
{\leq} \ & \dist(0, -\nabla_y F(x_r(y,z),y_+(z))+\partial\iota_{\mathcal{Y}}(y_+(z)))+ L\sigma_2\|y-y_+(z)\|\notag\\
{\leq} \ & \left(\frac{1}{\alpha}+L+L\sigma_2\right)\|y_+(z)-y\|,
\end{align*}
where the third inequality follows from the $L$-Lipschitz continuity  of $\nabla_y F(\cdot,\cdot)$ and  \eqref{lip-y};  the last inequality follows from the relative error condition of the projected gradient ascent method. This, together with \eqref{OS-GS-better-key3-main}, yields
\[
\|x_r(y(z),z)-x_r(y_+(z),z)\|\leq \frac{\sqrt{2}}{\sqrt{r-L}}\left(\frac{1+\alpha L(1+\sigma_2)}{\alpha\mu}\right)^{\frac{1}{2\theta}} 
\|y-y_+(z)\|^{\frac{1}{2\theta}}.
\]
 \end{enumerate}
The proof is complete.
\end{proof}
The dual error bound in Proposition \ref{prop:dual_eb_KL} involves the primal-dual quantity $x_r(\cdot,z)$. As the following corollary shows, one can also derive an alternative dual error bound that involves the pure primal quantity $x_r^{\star}(z)$. 
Since $x_r^{\star}(z)$ is the proximal mapping of $\tfrac{1}{r} f+ \iota_{\mathcal{X}}$ at $z \in \mathbb{R}^n$, the alternative dual error bound is closely related to the optimization-stationarity measure. As we shall see, such an error bound plays a crucial role in establishing not only the iteration complexity of smoothed PLDA for finding an OS of problem \eqref{eq:problem} (see Theorem \ref{thm:general}) but also a quantitative relationship between the optimization-stationarity and game-stationarity measures (see Section \ref{sec:sta}).

% Moreover, in order to establish the convergence result while considering optimization-stationarity in Theorem \ref{thm:general}, we require an alternative dual error bound. Additionally, the following corollary is of paramount importance in deriving the quantitative relationship between optimization-stationarity and game-stationarity, as discussed in Section \ref{sec:sta}.
\begin{corollary}[Alternative dual error bound with K\L{} exponent]
% Let $\delta>0$. 
Suppose that  Assumption \ref{ass:kl} holds. Then, for any $y \in \mathcal{Y}$ and $z \in \mathbb{R}^n$,  we have
% \[
% \|x_r^\star(z)-x_r(y_+(z),z)\| \leq \omega \|y-y_+(z)\|^{\frac{1}{2\theta}},
% \]
% where $\omega:=\frac{\sqrt{2}}{\sqrt{r-L}}\left(\frac{1+\alpha L(1+\sigma_2)}{\alpha\mu}\right)^{\frac{1}{2\theta}}$.
\begin{enumerate}[label=\normalfont(\alph*)]
 \item \rm{(K\L\ exponent $\theta=0$):} $\|x_r^\star(z)-x_r(y_+(z),z)\| \leq \omega_1 \|y-y_+(z)\|$,
  \item \rm{(K\L\ exponent $\theta\in(0,1)$):} $\|x_r^\star(z)-x_r(y_+(z),z)\| \leq \omega_2\|y-y_+(z)\|^{\frac{1}{2\theta}}$. 
  \end{enumerate}
 \label{prop:dual_eb_KL2}   
\end{corollary}

The ideas of the proof of Corollary \ref{prop:dual_eb_KL2} are similar to those of Proposition \ref{prop:dual_eb_KL}. The main difference lies in the definition of $\psi$, which needs to be modified as follows:
\[
\psi(x,z) = \max\limits_{y\in \mathcal{Y}} F_r(x,y,z).
\]
We refer the reader to Appendix \ref{app:coro} for details.

\begin{remark}
{\rm (i)} Our results extend those in \citep{yang2022faster} to a more general setting, where the primal function is nonsmooth and the dual problem possesses the K\L{} property with exponent taking any value in $[0,1)$. In particular, this covers the case where there are  constraints on the dual variable. More importantly, all the existing techniques cannot handle the nonsmooth structure of the primal function. 
Our generalization provides a deeper understanding of the relationship between the primal and dual updates and serves as a useful tool for studying various stationarity measures in Section \ref{sec:sta}. 
{\rm (ii)} In the concave case where Assumption \ref{ass:kl} is not satisfied, one can establish a similar dual error bound with $\theta=1$ and $\omega$ being related to ${\rm diam} (\mathcal{Y})$; see Lemma \ref{lemma-dual-bd} in Appendix \ref{sec:dua-appen}. 
{\rm (iii)} The dual error bound developed in this paper can potentially be applied to study the convergence rates of other algorithms for minimax optimization, making it a valuable contribution in its own right. 
\end{remark}

\subsection{Iteration Complexity of Smoothed PLDA}
\label{subsec:iter}
To establish the main theorem, we first develop the following lemma, which establishes a connection between various iterates gaps and the game-stationarity measure.
\begin{lemma} \label{lemma-episolcol}
        Let $\epsilon\ge 0$ be given. Suppose that
        \[
        \max\left\{\|x^{k+1}-x^{k}\|,\|y^{k}_+(z^{k+1}) - y^k\|,\|x^{k+1}-z^k\|\right\} \leq  \epsilon.
        \]
        Then,  $\left(x^{k+1}, y^{k+1}\right)$ is a $\rho \epsilon$-GS of problem \eqref{eq:problem}, where
\[\rho = \max\left\{(\eta+1+\sigma_1\alpha\beta L)\left(\frac{1}{\alpha}+L\right), r(\zeta+\sigma_2(\eta+1+\sigma_1\alpha \beta L)+\sigma_1)\right\}\] 
 with $\sigma_{1}:=\frac{r}{r-L}$ and $\sigma_{2}:=\frac{2r-L}{r-L}$.
        \end{lemma}
\begin{proof}
By Definition \ref{defi:primal-dual}, we only have to quantify $\|\nabla_{z} d_r(y^{k+1},x^{k+1})\|$ and ${\rm dist}(0,-\nabla_{y} F(x^{k+1}, y^{k+1})+\partial\iota_{\mathcal{Y}}(y^{k+1}))$.
To begin, observe that
\begin{equation*}
\begin{aligned}
& \| \nabla_{z} d_r(y^{k+1},x^{k+1})\| 
\\
=\ & r \|x^{k+1}-x_r (y^{k+1},x^{k+1})\|\\  
\leq\ & r(\|x^{k+1}-x_r(y^{k},z^{k})\| +\|x_r(y^{k},z^{k})-x_r(y^{k+1},z^{k})\|)  +r \|x_r(y^{k+1},z^{k})-x_r(y^{k+1},x^{k+1})\|  \\
\leq\  &  r\left(\zeta\|x^{k+1}-x^k\| +\sigma_2\|y^{k}-y^{k+1}\| + \sigma_1 \|x^{k+1}-z^{k}\| \right) \\
% \leq & r\left(\|x^{k+1}-x_r(y^{k},z^{k})\| +\|x_r(y^{k},z^{k})-x_r(y^{k+1},z^{k})\| + \|x_r(y^{k+1},z^{k})-x_r(y^{k+1},x^{k+1})\| \right) \\
\leq\  &  r\left(\zeta\|x^{k+1}-x^k\| + \sigma_1 \|x^{k+1}-z^{k}\| \right) +r\sigma_2 (\eta\| x^{k+1}-x^k\|+\sigma_1\alpha L\|z^{k+1}-z^k\|+\| y_+^k(z^{k+1})-y^k\|)\\
\leq\ & r(\zeta+\sigma_2(\eta+1+\sigma_1\alpha \beta L)+\sigma_1)\epsilon,
\end{aligned}
\end{equation*}
where the second inequality is due to Proposition \ref{prop:lip} and Lemma \ref{lemma-sollip}, the third is due to Lemma \ref{lemma-dualxy}, and the fourth follows from the update $z^{k+1} = z^k + \beta (x^{k+1} - x^k)$ and the assumption of the lemma. Next, recall that
\begin{equation*}
y^{k+1} = \proj_{\mathcal{Y}}\left(y^k + \alpha \nabla_y F(x^{k+1}, y^k)\right)=\argmin_{y \in \mathcal{Y}}\left\{\|y-y^k - \alpha \nabla_y F(x^{k+1}, y^k)\|^2\right\}.
\end{equation*}
The necessary optimality condition yields
\begin{align*}
0
&\in y^{k+1}-y^k - \alpha \nabla_y F(x^{k+1}, y^k)+\partial \iota_{\mathcal{Y}}(y^{k+1})\notag\\
&=y^{k+1}-y^k - \alpha \nabla_y F(x^{k+1}, y^{k+1})+\partial \iota_{\mathcal{Y}}(y^{k+1}) +\alpha (\nabla_y F(x^{k+1}, y^{k+1})- \nabla_y F(x^{k+1}, y^{k})).
\end{align*}
Since $\partial \iota_{\mathcal{Y}} (y^{k+1})$ is the normal cone to $\mathcal{Y}$ at $y^{k+1}$, we have
\begin{align*}
v
&:=-\frac{1}{\alpha}(y^{k+1}-y^{k})+\nabla_y F(x^{k+1}, y^{k})- \nabla_y F(x^{k+1}, y^{k+1})\in -\nabla_{y} F(x^{k+1}, y^{k+1})+\partial \iota_{\mathcal{Y}}(y^{k+1}).
\end{align*}
It follows that
\begin{align*}
\|v\| \ & \leq \left(\frac{1}{\alpha}+L\right)\|y^{k+1}-y^{k}\|\\
\ &  \leq  \left(\frac{1}{\alpha}+L\right)\left(\eta\|x^{k}-x^{k+1}\|+\sigma_1\alpha L\|z^{k+1}-z^k\|+\|y_+^k(z^{k+1})-y^k\|\right) \\
\ & \leq (\eta+1+\sigma_1\alpha\beta L)\left(\frac{1}{\alpha}+L\right)\epsilon,
\end{align*}
where the first inequality is due to the $L$-Lipschitz continuity of $\nabla_y F(\cdot,\cdot)$ and the second is due to Lemma \ref{lemma-dualxy}. Putting everything together yields the desired result.
\end{proof}

With the help of Propositions \ref{prop:decrease}, \ref{prop:lip}, \ref{prop:dual_eb_KL}, and Lemma \ref{lemma-episolcol}, we are now ready to develop the main theorem, which gives the iteration complexity of smoothed PLDA under various settings.
\begin{theorem}[Main theorem]
\label{thm:general}
 % Suppose that 
 % \[r\ge3L,\ \lambda\le \frac{1}{L},\ \alpha \leq \min\left\{\frac{1}{10L},\ \frac{1}{4L\zeta^2} \right\},\ \beta  \leq \min \left\{\frac{1}{28},\ \frac{(r-L)^2}{14\alpha r(2r-L)^2}\right\}.\] 
 Under the setting of Proposition \ref{prop:decrease}, for any integer $K>0$, there exists an index $k \in\{1,2, \ldots, K\}$ such that the following hold:
 \begin{enumerate}[label=\normalfont(\alph*)]
   \item \rm{(Concave):} Suppose that $F(x,\cdot)$ is concave for all $x \in \mathcal{X}$. If $\max_{y\in \mathcal{Y}} F(\cdot,y)$ is bounded below on $\mathcal{X}$ and $\beta\leq \mathcal{O}(K^{-\frac{1}{2}})$,  then $\left(x^{k+1}, y^{k+1}\right)$ is an $\mathcal{O}(K^{-\frac{1}{4}})$-GS and $z^{k+1}$ is an $\mathcal{O}(K^{-\frac{1}{4}})$-OS of problem \eqref{eq:problem}.
  \item \rm{(K\L\ exponent $\theta\in(\frac{1}{2},1)$):} Suppose that Assumption \ref{ass:kl} holds with $\theta \in (\tfrac{1}{2},1)$. If $\mathcal{X}$ is compact and  $ \beta\leq \mathcal{O}(K^{-\frac{2\theta-1}{2\theta}})$,
   then $\left(x^{k+1}, y^{k+1}\right)$ is an $\mathcal{O}( K^{-\frac{1}{4\theta}})$-GS and $z^{k+1}$ is an $\mathcal{O}( K^{-\frac{1}{4\theta}})$-OS  of problem \eqref{eq:problem}.
  \item \rm{(K\L\ exponent $\theta\in[0,\frac{1}{2}]$):} Suppose that Assumption \ref{ass:kl} holds with $\theta \in [0, \tfrac{1}{2}]$.
 If $\mathcal{X}$ is compact and  $\beta  \leq \frac{ {\rm diam}(\mathcal{Y})^{\frac{2\theta-1}{\theta}}}{224 \alpha r\omega_2^2}$ when $\theta \in (0,\tfrac{1}{2}]$ {\rm (}resp., $\beta  \leq \tfrac{ 1}{224 \alpha r\omega_1^2}$ when $\theta=0${\rm )}, then $\left(x^{k+1}, y^{k+1}\right)$ is an $\mathcal{O}(K^{-\frac{1}{2}})$-GS and $z^{k+1}$ is an $\mathcal{O}(K^{-\frac{1}{2}})$-OS of problem \eqref{eq:problem}.
  \end{enumerate}
  
\end{theorem}

\begin{remark}\label{rmk-main}
{\rm (i)} Note that concavity alone is not sufficient to guarantee the K\L{} property. An example can be found in \citep[Theorem 36]{bolte2010characterizations}. Therefore, in Theorem \ref{thm:general}, we need to treat the concave case separately.
{\rm (ii)} When the dual function is concave, 
the results in cases (b) and (c) of Theorem \ref{thm:general} remain valid under the weaker assumption that the K{\L} property holds locally around any GS of problem \eqref{eq:problem}, cf. Assumption \ref{ass:kl}. The proof follows that in \citep{zhang2020single} and is omitted here.
% the assumption of the K\L{} property (i.e., Assumption \ref{ass:kl}) can be reduced to the local version around any GS $(x^\star, y^\star)$. The proof argument is similar to \citep{zhang2020single}
{\rm (iii)}
The compactness assumption on $\mathcal{X}$ implies that 
$p_r$
is bounded below.  
For the case where the dual function is concave, we can relax this assumption to that of the lower boundedness of $\max_{y\in \mathcal{Y}} F(\cdot,y)$ on $\mathcal{X}$. The latter allows for the possibility that $\mathcal{X}$ is unbounded and is in line with the assumptions made in the literature on smooth nonconvex-concave minimax optimization; see, e.g., \citep{zhang2020single} and \citep{yang2022faster}.
\end{remark}

\begin{remark} \label{rmk:gs-os}
For smooth nonconvex-concave problems, the work \citep{lin2020gradient} introduces a two-timescale GDA algorithm that computes an $\epsilon$-OS  with an iteration complexity of $\mathcal{O}(\epsilon^{-6})$. Such an $\epsilon$-OS  can then be converted into an $\epsilon$-GS  with an additional cost of $\mathcal{O}(\epsilon^{-2})$ (depending on the specific algorithm employed). 
% However, when using Theorem \ref{thm-stationary}, which is algorithm-independent, our proposed smoothed PLDA
% obtains an $\epsilon$-OS with a suboptimal iteration complexity of $\mathcal{O}(\epsilon^{-8})$.
% Nevertheless, it is possible to improve the complexity by leveraging specific properties of the algorithm used, as demonstrated in Theorem \ref{thm:general} when we study our proposed smoothed PLDA. Indeed, 
By contrast, Theorem \ref{thm:general} shows that smoothed PLDA can find both an $\epsilon$-GS and an $\epsilon$-OS of a structured nonsmooth nonconvex-concave problem with the same iteration complexity of $\mathcal{O}(\epsilon^{-4})$. Furthermore, when the dual problem satisfies the K\L{} property with exponent $\theta\in[0,\frac{1}{2}]$, 
% the conversion from an $\epsilon$-GS  to an $\epsilon$-OS  is algorithmically optimal due to Theorem \ref{thm-stationary} 
smoothed PLDA has the optimal iteration complexity of $\mathcal{O}(\epsilon^{-2})$ for finding both an $\epsilon$-GS and an $\epsilon$-OS; see Footnote 1.
 \end{remark}

\begin{proof}

As mentioned in Section \ref{sec:framework}, the global convergence rate of smoothed PLDA depends on the specific regime of the growth power $v$ in \eqref{eq:step_dual}. Let us begin by considering cases (a) and (b), i.e., either the dual function is concave or the dual problem possesses the K\L{} property with exponent $\theta\in(\frac{1}{2},1)$.

In view of the basic descent estimate in Proposition \ref{prop:decrease}, we consider the following two scenarios separately:\\
(1) There exists a $ k \in\{0,1, \ldots, K-1\}$ such that
\begin{align*}
\ &\frac{1}{2}\max \left\{\frac{3}{8\lambda}\|x^{k}-x^{k+1}\|^{2}, \frac{1}{8\alpha}\|y^{k}-y_{+}^{k}(z^{k+1})\|^2 , \frac{4r}{7\beta}\|z^{k}-z^{k+1}\|^{2}\right\}\\
\leq \ & 14r\beta\|x_r(y(z^{k+1}),z^{k+1})-x_r(y_{+}^{k}(z^{k+1}), z^{k+1})\|^{2}.
\end{align*}
(2) For any $ k \in\{0,1, \ldots, K-1\}$, we have
\begin{align*}
\ &\frac{1}{2}\max \left\{\frac{3}{8\lambda}\|x^{k}-x^{k+1}\|^{2}, \frac{1}{8\alpha}\|y^{k}-y_{+}^{k}(z^{k+1})\|^2 , \frac{4r}{7\beta}\|z^{k}-z^{k+1}\|^{2}\right\}\\
\ge\ & 14r\beta\|x_r(y(z^{k+1}),z^{k+1})-x_r(y_{+}^{k}(z^{k+1}), z^{k+1})\|^{2}.
\end{align*}

\textbf{Scenario (1)}. For the case where the K{\L} exponent $\theta \in (\tfrac{1}{2},1)$, we deduce from Proposition \ref{prop:dual_eb_KL}(b) that 
\begin{align*}
    \|y^{k}-y_{+}^{k}(z^{k+1})\|^2  & \leq 224r\alpha\beta\|x_r(y(z^{k+1}),z^{k+1})-x_r(y_{+}^{k}(z^{k+1}), z^{k+1})\|^{2}  \\ &\leq  224r\alpha\beta \omega_2^2 \|y^{k}-y_{+}^{k}(z^{k+1})\|^{\frac{1}{\theta}},  
\end{align*}
which gives $\|y^{k}-y_{+}^{k}(z^{k+1})\|\leq \rho_1\beta^{\frac{\theta}{2\theta-1}}$ with $\rho_1 := (224r\alpha\omega_2^2)^{\frac{\theta}{2\theta-1}}$. Additionally, using the update $z^{k+1} = z^k + \beta( x^{k+1} - z^k )$ and Proposition \ref{prop:dual_eb_KL}(b), we have
\begin{align*}
\|x^{k+1}-z^{k}\|^{2}  \ & =\frac{1}{\beta^2}\|z^{k+1}-z^{k}\|^{2}\\
\ &  \leq 49\|x_r(y(z^{k+1}),z^{k+1})-x_r(y_{+}^{k}(z^{k+1}), z^{k+1})\|^{2} \\
\ & \leq  49 \omega_2^2 \|y^{k}-y_{+}^{k}(z^{k+1})\|^{\frac{1}{\theta}}\leq \rho_2^2\beta^{\frac{1}{2\theta-1}},
\end{align*}
where $\rho_2 := \sqrt{49\omega_2^2\rho_1^{\frac{1}{\theta}}}$. 
Lastly, we have
\begin{align*}
    \|x^{k+1}-x^k\|^2 \ & \leq {\frac{224}{3}r\beta\lambda}\|x_r(y(z^{k+1}),z^{k+1})-x_r(y_{+}^{k}(z^{k+1}), z^{k+1})\|^{2} \\
   \  &\leq {\frac{224}{3}r\omega^2_2\beta\lambda}\|y^{k}-y_{+}^{k}(z^{k+1})\|^{\frac{1}{\theta}} \leq \rho_3^2 \beta^{\frac{2\theta}{2\theta-1}},  
\end{align*}
where $\rho_3 := \sqrt{{\frac{224}{3}r\omega_2^2\lambda}\rho_1^{\frac{1}{\theta}}}$. 
Combining the above inequalities, we have
\begin{align*}
& \max \left\{\|x^{k}-x^{k+1}\|,\|y^{k}-y_{+}^{k}(z^{k+1})\|,\|z^{k}-x^{k+1}\|\right\} \leq  \max\left\{ \rho_1\beta^{\frac{\theta}{2\theta-1}}, \rho_2\beta^{\frac{1}{4\theta-2}}, \rho_3 \beta^{\frac{\theta}{2\theta-1}}\right\}. 
\end{align*}
It then follows from Lemma \ref{lemma-episolcol}
 that $\left(x^{k+1}, y^{k+1}\right)$ is a $\rho\cdot\max\{ \rho_1\beta^{{\theta}/{(2\theta-1)}},$ $ \rho_2\beta^{1/{(4\theta-2)}}, \rho_3\beta^{{\theta}/{(2\theta-1)}}\}$-GS of problem \eqref{eq:problem}. 
 
 For the concave case, we can replace the primal-dual quantity $x_r( y(z^{k+1}), z^{k+1} )$ by the primal quantity $x_r^\star( z^{k+1} )$ on the right-hand side of the inequality for Scenario (1) and apply Lemma \ref{lemma-dual-bd} in Appendix \ref{sec:dual} (so that $\theta$ and $\omega_2$ are replaced by $1$ and $\kappa = \tfrac{1 + \alpha L \sigma_2 + \alpha L}{\alpha (r-L)} \cdot \textrm{diam}(\mathcal{Y})$, respectively) to conclude that $\left(x^{k+1}, y^{k+1}\right)$ is a $\max\{ \rho_1\beta, \rho_2\beta^{1/2}, \rho_3 \beta\}$-GS of problem \eqref{eq:problem}.

Furthermore, observe that for the case where $\theta \in (\tfrac{1}{2}, 1)$, the optimization-stationarity measure of $z^{k+1}$ satisfies
\begin{equation}
\label{eq:os-gs}
\begin{aligned}
    & \|z^{k+1}-x_r^\star(z^{k+1})\| \\
    \leq\ &  \|z^{k+1}-x^{k+1}\|+\| x^{k+1} - x_r(y^k,z^k) \| + \| x_r(y^k,z^k) - x_r(y^k,z^{k+1}) \|\\
    \ & + \|x_r(y^k,z^{k+1})-x_r(y_{+}^{k}(z^{k+1}), z^{k+1})\|+ \|x_r(y_{+}^{k}(z^{k+1}), z^{k+1})-x_r^\star(z^{k+1})\|\\
     \leq\ &  (1+\sigma_1)\|z^{k}-x^{k+1}\| + \zeta\|x^{k}-x^{k+1}\| + \sigma_2 \|y^k-y_{+}^{k}(z^{k+1})\|+ \omega_2 \|y^k-y_{+}^{k}(z^{k+1})\|^{\frac{1}{2\theta}} \\
    \leq\ & \max\left\{\left(\rho_2(1+\sigma_1)+ \omega_2 \rho_1^\frac{1}{2\theta}\right)\beta^{\frac{1}{4\theta-2}},(\zeta\rho_3+\sigma_2\rho_1) \beta^{\frac{\theta}{2\theta-1}} \right\}, 
\end{aligned}
\end{equation}
where the second inequality follows from the update $z^{k+1} = z^k + \beta (x^{k+1} - z^k)$, Proposition \ref{prop:lip}, Corollary \ref{prop:dual_eb_KL2}, Lemma \ref{lemma-sollip}, and the fact that $0 < \beta \le \tfrac{1}{28}$.
It is evident that the dependence on $\beta$ is identical to that of the game-stationarity measure of $(x^{k+1}, y^{k+1})$.
For the concave case, the optimization-stationarity measure of $z^{k+1}$ can be bounded similarly.

\textbf{Scenario (2)}. Proposition \ref{prop:decrease} implies that for any $k \in\{0,1, \ldots, K-1\}$, we have 
\begin{align}
\label{eq:suff-decrease}
\Phi_r^k-\Phi_r^{k+1} 
\ge \frac{3}{16\lambda}\|x^{k}-x^{k+1}\|^{2}+\frac{1}{16\alpha}\|y^{k}-y_{+}^{k}(z^{k+1})\|^2 +\frac{2r}{7\beta}\|z^{k}-z^{k+1}\|^{2}.
\end{align}
The assumptions in case (a)  imply that 
\[ \min_{x\in\mathcal{X}}\max_{y\in\mathcal{Y}} F_r(x,y,z) = \max_{y\in\mathcal{Y}} \min_{x\in\mathcal{X}} F_r(x,y,z) \]
for all $z \in \mathbb{R}^n$ (cf.~the proof of~\citep[Lemma B.7]{zhang2020single}) and lead to the lower boundedness of $p_r$, while the assumptions in case (b) directly imply that $p_r$ is lower bounded.
In both cases, there exists a $\underline{\Phi} > -\infty$ such that for all $x \in \mathcal{X}$, $y \in \mathcal{Y}$, and $z \in \mathbb{R}^n$,
\begin{align*}
& \Phi_r(x, y, z)  = p_r(z)+(F_r(x,y,z)-d_r(y, z))+(p_r(z)-d_r(y, z))\geq p_r(z)\geq \underline{\Phi} >-\infty.
\end{align*}
It follows that
    \begin{align*}
    &\Phi_r^0-\underline{\Phi} \geq \sum_{k=0}^{K-1}\Phi_r(x^k,y^k,z^k)-\Phi_r(x^{k+1},y^{k+1},z^{k+1}) \\
    % \geq\ & \sum_{k=0}^{K-1} \frac{1}{32\lambda}\|x^{k}-x^{k+1}\|^{2}+\frac{1}{16\alpha}\|y^{k}-y_{+}^{k}(z^{k})\|^2 +\frac{2r}{7\beta}\|z^{k}-z^{k+1}\|^{2} \\
    \geq \ & \sum_{k=0}^{K-1} \min\left\{\frac{3}{16\lambda},\frac{1}{16\alpha},\frac{2\beta r}{7}\right\} \left(\|x^{k}-x^{k+1}\|^{2}+\|y^{k}-y_{+}^{k}(z^{k+1})\|^2 + \|z^{k}-x^{k+1}\|^{2} \right),
    \end{align*}
    where the last inequality is due to \eqref{eq:suff-decrease} and the update $z^{k+1}=z^{k}+\beta(x^{k+1}-z^k)$.
    In particular, there exists a $k \in\{0,1, \ldots, K-1\}$ such that
    \[
    \max \left\{\|x^{k}-x^{k+1}\|^{2},\|y^{k}-y_{+}^{k}(z^{k+1})\|^{2}, \|x^{k+1}-z^{k}\|^{2}\right\} \leq \frac{\Phi_r^0-\underline{\Phi}}{\min\left\{\frac{3}{16\lambda},\frac{1}{16\alpha},\frac{2\beta  r}{7}\right\}  K}.
    \]
% \Jiajin{continue here!}
Based on Lemma \ref{lemma-episolcol} and the fact that $0 < \beta \le \tfrac{1}{28}$,  we conclude that $\left(x^{k+1}, y^{k+1}\right)$ is an $\mathcal{O}(\sqrt{1/{K\beta}})$-GS of problem \eqref{eq:problem}. Moreover, by using the same argument as in \eqref{eq:os-gs}, we can show that $z^{k+1}$ is an $\mathcal{O}(\sqrt{1/{K\beta}})$-OS of problem \eqref{eq:problem}. Upon setting $\beta = C K^{-\frac{1}{2}}$ for the concave case and  $\beta = C K^{-\frac{2\theta-1}{2\theta}}$ for the case where the K{\L} exponent $\theta \in (\tfrac{1}{2}, 1)$, with $C$ being a suitably chosen constant, we see that the 
% where $C$ is a constant.
% Similarly, for the general concave case, we can set $\beta = C K^{-\frac{1}{2}}$. 
bounds obtained for the optimization-stationarity and game-stationarity measures are of the same order for both Scenarios (1) and (2). This establishes cases (a) and (b).

Next, let us consider case (c). First, suppose that the K{\L} exponent $\theta \in (0, \tfrac{1}{2}]$. 
Once again, we can use Proposition \ref{prop:decrease} to obtain 
\begin{align*}
\Phi_r^k-\Phi_r^{k+1} 
\ge\ &  \frac{3}{8\lambda}\|x^{k}-x^{k+1}\|^{2}+\frac{1}{8\alpha}\|y^{k}-y_{+}^{k}(z^{k+1})\|^2 +\frac{4r}{7\beta}\|z^{k}-z^{k+1}\|^{2}\\
\ & -14r\beta\|x_r(y(z^{k+1}),z^{k+1})-x_r(y_{+}^{k}(z^{k+1}), z^{k+1})\|^{2}.
\end{align*}
Using Proposition \ref{prop:dual_eb_KL}(b) and the boundedness of $\mathcal{Y}$, we have
% there exists  for each $\theta \in(0,\frac{1}{2})$ that 
\begin{equation}\label{dual-eb-bd}
\begin{aligned}
  \|x_r(y(z^{k+1}),z^{k+1})-x_r(y_{+}^{k}(z^{k+1}), z^{k+1})\|
&\leq \omega_2 \|y^{k}-y_{+}^{k}(z^{k+1})\|^{\frac{1}{2\theta}}\\
&=  \omega_2 \|y^{k}-y_{+}^{k}(z^{k+1})\|^{\frac{1}{2\theta}-1} \|y^{k}-y_{+}^{k}(z^{k+1})\|\\
&\leq   \omega_2\cdot {\rm diam}(\mathcal{Y})^{\frac{1}{2\theta}-1}\cdot \|y^{k}-y_{+}^{k}(z^{k+1})\|.
\end{aligned}
\end{equation}
It follows that
\begin{align*}
\Phi_r^k-\Phi_r^{k+1} 
\ge \ &  \frac{3}{8\lambda}\|x^{k}-x^{k+1}\|^{2}+\left(\frac{1}{8\alpha}-14r\beta\omega_2^2\cdot{\rm diam}(\mathcal{Y})^{\frac{1-2\theta}{\theta}}\right) \|y^{k}-y_{+}^{k}(z^{k+1})\|^2  +\frac{4r}{7\beta}\|z^{k}-z^{k+1}\|^{2}.
\end{align*}
Since 
$\beta  \leq \frac{ {\rm diam}(\mathcal{Y})^{\frac{2\theta-1}{\theta}}}{224 \alpha r\omega_2^2}$, we have 
\begin{align*}
&\Phi_r^k-\Phi_r^{k+1} \ge  \frac{3}{8\lambda}\|x^{k}-x^{k+1}\|^{2}+\frac{1}{16\alpha}\|y^{k}-y_{+}^{k}(z^{k+1})\|^2 +\frac{4r}{7\beta}\|z^{k}-z^{k+1}\|^{2}.
\end{align*}
The desired result can then be obtained by following a similar argument as that in the paragraph proceeding \eqref{eq:suff-decrease}.

When the K{\L} exponent $\theta=0$, we can apply Proposition \ref{prop:dual_eb_KL}(a) to derive the inequality \eqref{dual-eb-bd} 
% by letting $\theta=\frac{1}{2}$ and replacing $\omega_2$ with $\omega_1$. 
with $\theta$ and $\omega_2$ being replaced by $\tfrac{1}{2}$ and $\omega_1$, respectively.
The desired result then follows by adapting the argument in the preceding paragraph.
\end{proof}

\subsection{Phase Transition Phenomenon} \label{subsec:phase}
Theorem \ref{thm:general} shows that smoothed PLDA achieves the optimal iteration complexity of $\mathcal{O}(\epsilon^{-2})$ when the K\L{} exponent of the dual problem lies in $[0,\tfrac{1}{2}]$. What is more, it reveals a surprising phase transition phenomenon at the boundary $\theta = \tfrac{1}{2}$, where the iteration complexity of the method changes. 
Such a phenomenon can be explained using the dual error bound we developed in Proposition \ref{prop:dual_eb_KL}, which characterizes the inherent tradeoff between the primal and dual updates. Indeed, such an error bound aims to control the negative term in the basic descent estimate in Proposition \ref{prop:decrease}. The ``degree'' of this control depends on the K\L{} exponent of the dual problem, which decides the final iteration complexity of our method. Let us consider the two regimes separately. 
\begin{enumerate}
\item When $\theta \in (0,\tfrac{1}{2}]$, the primal update dominates the optimization process as the dual problem is already nice enough. Conceptually, since the primal update still requires the solution of a strongly convex problem, we cannot surpass the optimal iteration complexity of $\mathcal{O}(\epsilon^{-2})$. 
Quantitatively,  the dual update yields a faster decrease in the Lyapunov function value than the primal one, as we have
\[\|x_r(y(z),z)-x_r(y_+(z),z)\| \leq \omega_2\|y-y_+(z)\|^{\frac{1}{2\theta}}.\]
Thus, the main bottleneck in the basic descent estimate is the three positive  quantities, and we are only able to achieve an  iteration complexity of $\mathcal{O}(\epsilon^{-2})$ by using the following sufficient  decrease inequality: 
\begin{align*}
&\Phi_r^k-\Phi_r^{k+1} 
= \Omega \left( \|x^{k}-x^{k+1}\|^{2}+\|y^{k}-y_{+}^{k}(z^{k+1})\|^2 +\|z^{k}-z^{k+1}\|^{2}\right).
\end{align*}
\item When $\theta \in (\tfrac{1}{2},1)$, the dual update dominates the optimization process. This is because the growth condition of the dual problem is worse compared to that of the strongly convex function in the primal update. Quantitatively, the dual update yields a slower decrease in the Lyapunov function value than the primal one. When we incorporate the dual error bound into the basic descent estimate, it becomes evident that the main challenge lies in managing the quantity $\|y^{k}-y_{+}^{k}(z^{k+1})\|^\frac{1}{\theta}$. In order to effectively balance the three positive quantities in the basic descent estimate, it is necessary to optimally set $\beta = \mathcal{O}(\epsilon^{4\theta})$ in our proof of Theorem \ref{thm:general}. 
\end{enumerate}

The underlying insight of the aforementioned phase transition phenomenon is that the overall iteration complexity of smoothed PLDA is determined by the slower of the primal and dual updates. As we have previously discussed, the dual error bound in Proposition \ref{prop:dual_eb_KL} offers an effective and theoretically justified approach for balancing the primal and dual updates. It is thus natural to ask whether the primal-dual relationship within our dual error bound is optimal or not.

While a complete answer to this question remains elusive, we now show that an approach different from the one we used in Section \ref{sec:dual-error} will lead to a dual error bound that gives suboptimal iteration complexity results. By Lemma \ref{lemma-sollip} in Appendix \ref{sec:lemmas}, we know that 
% Next, our objective is to validate our approach for developing the proposed dual error bound.
% As Lemma \ref{lemma-sollip} demonstrates that $x_r(\cdot,z)$ is Lipschitz for any $z\in \R^n$, it implies that
\begin{equation}\label{pure-dual}
\|x_r(y(z),z)-x_r(y_+(z),z)\| \leq \sigma_2 \|y(z)-y_+(z)\|
\end{equation}
for any $z\in \R^n$ (recall that $y(z) \in \argmax_{y \in \mathcal{Y}} d_r(y,z)$ and $y_+(z) = \proj_{\mathcal{Y}} ( y + \alpha \nabla_y F(x_r(y,z), y) )$). 
This suggests that one may obtain a dual error bound by relating $\|y(z)-y_+(z)\|$ to $\|y-y_+(z)\|$ for any $y \in \mathcal{Y}$.
To implement this approach, we develop the following new K{\L} calculus rule for the max operator, which is motivated by the definition of $y(z)$ and could be of independent interest.
%  This alternative approach leads to a similar dual error bound, which we refer to as the pure error bound to distinguish it from the resulting error bound in Proposition \ref{prop:dual_eb_KL}.

% Surprisingly, the pure dual error suffers from a suboptimal result. To illustrate this, we develop a new K\L{} calculus rule for the max operator, which could be of independent interest.

\begin{proposition}[K\L{} exponent  of max operator] 
\label{prop:max_kl}
Suppose that Assumption \ref{ass:kl} holds. 
Then, for any $z\in \R^n$,  the function $-d_r(\cdot, z)+\iota_{\mathcal{Y}}(\cdot)$ possesses  the K\L{} property with exponent $\theta$. Specifically,
for any  $y\in\mathcal{Y}$ and  $z\in \R^n$, we have
\begin{equation*}
% \label{eq:kl-max}
\dist(0, -\nabla_y d_r(y,z)+\partial\iota_{\mathcal{Y}}(y)) \ge \mu  \left(\max\limits_{y'\in \mathcal{Y}} d_r(y',z)-d_r(y,z)\right)^{\theta}.
\end{equation*}
\end{proposition}
\begin{proof}
Let $y \in \mathcal{Y}$ and $z \in \mathbb{R}^n$ be arbitrary. We bound
\begin{align*}
\ & \dist(0, -\nabla_y d_r(y,z)+\partial\iota_{\mathcal{Y}}(y))\\ {=} \ & \dist(0, -\nabla_y F_r(x_r(y,z),y,z) +\partial\iota_{\mathcal{Y}}(y)) \\
{\ge} \ & \mu  \left(\max_{y'\in\mathcal{Y}} F(x_r(y,z),y') -F(x_r(y,z),y)\right)^{\theta} \\
{=} \ & \mu  \left(\max_{y'\in\mathcal{Y}} F(x_r(y,z),y') +\frac{r}{2}\|x_r(y,z)-z\|^2-d_r(y,z)\right)^{\theta} \\
\ge \ &\mu \left(\max\limits_{y'\in \mathcal{Y}} \min\limits_{x\in \mathcal{X}} \left\{F(x,y') +\frac{r}{2}\|x-z\|^2\right\} -d_r(y,z) \right)^{\theta}\\
 = \ & \mu \left(\max\limits_{y'\in \mathcal{Y}} d_r(y',z)- d_r(y,z) \right)^{\theta},
\end{align*}
where the first equality holds due to the strong concavity of $-F_r(\cdot,y,z)$ and \citep[Theorem 10.31]{rockafellar2009variational}, the first inequality follows directly from Assumption \ref{ass:kl} by taking $x = x_r(y,z)$, and  the second equality follows from the fact that 
\[
d_r(y,z) = \min_{x\in \mathcal{X}} \left\{ F(x,y) +\frac{r}{2}\|x-z\|^2\right\} = F(x_r(y,z),y) +\frac{r}{2}\|x_r(y,z)-z\|^2.
\]
% The last inequality holds as 
% \begin{align*}
% \max\limits_{y'\in \mathcal{Y}} d_r(y',z) - d_r(y,z) & = \max\limits_{y'\in \mathcal{Y}} \min\limits_{x\in \mathcal{X}} \left\{F(x,y') +\frac{r}{2}\|x-z\|^2\right\} -d_r(y,z) \\
% & \leq \max\limits_{y'\in \mathcal{Y}} F(x_r(y,z),y') +\frac{r}{2}\|x_r(y,z)-z\|^2-d_r(y,z).
% \end{align*}
The proof is then complete.
\end{proof}

% Lemma \ref{lemma-dualdiff} establishes that $d_r$ has a Lipschitz gradient and is therefore weakly convex. This allows us to use \citep[Theorem 3.7]{drusvyatskiy2021nonsmooth} to deduce that the slope, Fr\'{e}chet, and limiting subdifferentials coincide. 

To proceed, we make use of  \citep[Theorem 3.7]{drusvyatskiy2021nonsmooth}, which connects the K\L\ property and the slope error bound.
Specifically, let $y \in \mathcal{Y}$ and $z \in \mathbb{R}^n$ be arbitrary. By Lemma \ref{lemma-dualdiff} in Appendix \ref{sec:lemmas}, the function $d_r(\cdot,z)$ is differentiable. Thus, the Fr\'{e}chet and limiting subdifferentials of the function $-d_r(\cdot,z) + \iota_{\mathcal{Y}}(\cdot)$ coincide (see, e.g., \citep{li2020understanding}), and the slope of this function equals $\dist( 0, -\nabla_y d_r(\cdot,z) + \partial \iota_{\mathcal{Y}}(\cdot) )$ (see, e.g.,  \citep{drusvyatskiy2021nonsmooth}). Consequently, we can apply Proposition \ref{prop:max_kl}, \citep[Theorem 3.7]{drusvyatskiy2021nonsmooth}, and the relative error condition of the projected gradient ascent method to get
\begin{equation}\label{pure-dual-bound}
\begin{aligned}
\dist(y_+(z),Y(z)) & \ = \mathcal{O} \left (\dist^{\frac{1-\theta}{\theta}} (0, -\nabla_y d_r(y_+(z),z)+\partial\iota_\mathcal{Y}(y_+(z)) \right)\\
& \ = \mathcal{O} \left( \|y-y_+(z)\|^{\frac{1-\theta}{\theta}}\right).
\end{aligned}
\end{equation}
% Moreover, since $d_r$ is gradient  Lipschitz continuous, it is weakly convex and  thus the slope, Fr\'{e}chet, and limiting subdifferentials coincide. Thus, we can apply the K\L{} calculus rule of the max operator presented in Proposition \ref{prop:max_kl}, and obtain the following inequality (ignoring the coefficient):
% \begin{equation}\label{pure-dual-bound}
% \begin{aligned}
% \dist(y_+(z),Y(z)) & \ \leq \dist^{\frac{1-\theta}{\theta}} (0, -\nabla_y d_r(y_+(z),z)+\partial\iota_\mathcal{Y}(y_+(z)) \\
% & \ \leq \|y-y_+(z)\|^{\frac{1-\theta}{\theta}},
% \end{aligned}
% \end{equation}
Combining \eqref{pure-dual-bound} with \eqref{pure-dual} yields the bound
\[ \| x_r( y(z),z ) - x_r( y_+(z),z ) \| = \mathcal{O}\left( \| y - y_+(z) \|^{ \tfrac{1-\theta}{\theta} } \right), \]
which we shall refer to as the pure dual error bound.

Let us now compare the exponents of the dual error bound in Proposition \ref{prop:dual_eb_KL} and the pure dual error bound; see Figure \ref{fig:power}.
When $\theta \in (\tfrac{1}{2},1)$, the exponent of the former is smaller; when $\theta \in (0,\frac{1}{2}]$, the opposite is true. However, as we have noted earlier, the exponent of the dual error bound does not affect the final iteration complexity when $\theta \in (0,\tfrac{1}{2}]$,  as the primal update dominates the optimization process. Thus, we see that for the purpose of studying the iteration complexity of smoothed PLDA, the pure dual error bound is inferior to the dual error bound in Proposition \ref{prop:dual_eb_KL}.

\begin{figure}[ht]
    \centering
    \includegraphics[width=0.5\textwidth]{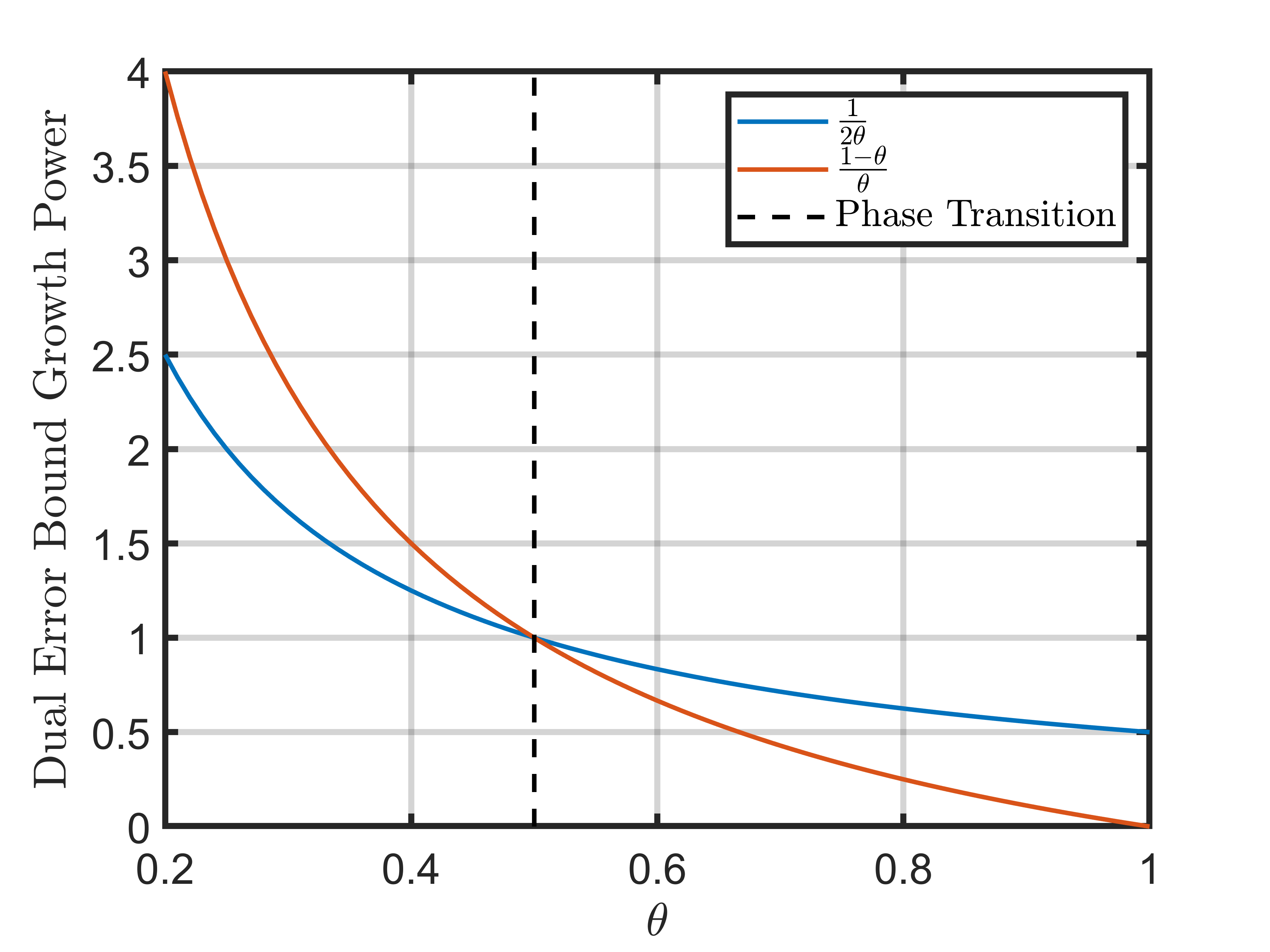}
    \caption{Comparison between the dual error bound (Proposition \ref{prop:dual_eb_KL}) and the pure dual error bound. }
    \label{fig:power}
\end{figure}

\section{Verification of K\L{} Property}
\label{sec:verify}
As we have seen from the last section, the  K\L{} property (Assumption \ref{ass:kl}) is crucial to establishing the iteration complexity of our proposed smoothed PLDA. In this section, we show that  Assumption \ref{ass:kl} holds when the dual function is linear, i.e., problem \eqref{eq:problem} takes the form
\begin{equation}
\label{eq:linear}
\min_{x \in \mathcal{X}} \max_{y \in \mathcal{Y}}
\left\{F(x,y)=y^\top G(x)\right\},
\end{equation}
where $G: \mathbb{R}^n \rightarrow \mathbb{R}^d$  is an arbitrary mapping
% is a  continuously differentiable map with a Lipschitz continuous Jacobian 
and $\mathcal{Y}$ (assumed to be nonempty) is a polytope defined by
\[
\mathcal{Y}=\{y\in\R^d: a_{i}^\top y \leq b_{i},\ i\in [l]\},
\] 
with $l \in \mathbb{N}$ and $a_i \in \mathbb{R}^d$, $b_i \in \mathbb{R}$ for $i \in [l]$.

\begin{theorem}[K\L{} exponent of linear dual problem] 
\label{cor:linear}Consider problem \eqref{eq:linear}. Suppose there exists a $\delta>0$ such that for any $x \in \mathcal{X}$ and $y^{\star}(x) \in \argmax_{y \in \mathcal{Y}} y^\top G(x)$, we have $\lambda_i(x) \ge \delta$ for all $i\in \mathcal{I}(x)$, where $\mathcal{I}(x)=\{i:a_{i}^\top y^\star(x) = b_{i}\}\subseteq[l]$ is the active index set and $\lambda_i(x)$ is the dual variable associated with the constraint $a_{i}^\top y - b_{i}\leq 0$ for $i \in [l]$.
Then,
Assumption \ref{ass:kl} holds with $\theta = 0$.
\end{theorem}

\begin{proof}
Our goal is to show that there exists a $\mu>0$ satisfying
\begin{equation*}
% \label{eq:weak-sharp}
\dist(0, -G(x)+\partial\iota_{\mathcal{Y}}(y))\ge\mu \quad \text{for all}\  x \in \mathcal{X}, \ y \in \mathcal{Y} \setminus \mathcal{Y}^\star(x),
\end{equation*}
where $\mathcal{Y}^{\star}(x) = \argmax_{y' \in \mathcal{Y}} F(x,y') = \left\{ y \in \mathcal{Y} : \max_{y' \in \mathcal{Y}} F(x,y') \le F(x,y) \right\}$. In view of the equivalence between the weak sharp minimum property and the K{\L} property with exponent $\theta=0$ for convex functions \citep[Theorem 5]{bolte2017error}, it suffices to show that
\begin{align*}
% \label{y-linear-key2}
\mu\cdot\dist\left(y, \mathcal{Y}^{\star}(x)\right)\leq \max\limits_{y'\in \mathcal{Y}} F(x,y') - F(x,y)\quad \text{for all}\ x\in \mathcal{X},\ y\in\mathcal{Y}.
\end{align*}
Towards that end, let $x \in \mathcal{X}$ be arbitrary and consider the following linear programming problem:
 \begin{equation}
    \begin{array}{cll}
     \max\limits_{y'\in\R^d} & F(x,y') \\ 
        \st & a_i^\top y'-b_i \leq 0, \quad i \in [l].
    \end{array} 
    \label{y-linear-P}
    \end{equation}
% Throughout the rest of our proof, we maintain the notation $F(x,\cdot)=\langle G(x),\cdot\rangle$ for any $x\in \mathcal{X}$. If we choose 
% \begin{align*}
% \bar{y} &  \in \left[-F(x,\cdot) \leq \min_{y'\in\mathcal{Y}} -F(x,y')\right] \cap \mathcal{Y},
% \end{align*}
For any $\bar{y} \in \mathcal{Y}^{\star}(x)$, using the  KKT condition of \eqref{y-linear-P}, we can find $\lambda_{i}(x) \geq 0$ for  $i \in [l]$ such that
\begin{equation}\label{y-linear-kkt}
\left\{
\begin{array}{lr}
-G(x)+\sum_{i=1}^{l} \lambda_{i}(x) a_{i}=0,\\
\lambda_{i}(x)\cdot(a_{i}^\top \bar{y}-b_{i})=0, \quad i\in[l].
\end{array}
\right.
\end{equation}
% Define the function $\phi_x(y):=-F(x,y)+\sum_{i=1}^{l} \lambda_{i}(a_{i}^\top y-b_{i})$. From \eqref{y-linear-kkt} and the fact that $-F(x,y)$ is linear, we know that
% \[
% \phi_x(y)=-\sum_{i=1}^{l} \lambda_i b_{i}=\langle -\nabla_y F(x,\bar{y}), \bar{y}\rangle=-F(x,\bar{y}),
% \]
% and it follows that $\nabla_y \phi_x(\bar{y})=0$ which implies that $\phi_x$ attains its minimum at $\bar{y}$.
% % and $\min \phi_x=\phi_x(\bar{y})=0$. 
% Thus, from Lemma 3.3, there exists $\mu>0$ such that
% $$
% {\rm dist}(y,[\phi_x=\max_{y\in\mathcal{Y}} F(x,y)])={\rm dist}(y,[\phi_x \geq \max_{y\in\mathcal{Y}} F(x,y)]) \leq \mu\left(\max_{y\in\mathcal{Y}} F(x,y)-\phi_x(y)\right)\quad \text{for all}\ y \in \mathbb{R}^{d}.
% $$
Moreover, if we let $\mathcal{L}(x):=\left\{y'\in\R^d:(\lambda_{i}(x) a_{i})^\top y'=\lambda_{i}(x) b_{i},\ i\in[l]\right\}$,  then
$
\mathcal{Y} \cap \mathcal{L}(x)=\mathcal{Y}^{\star}(x).
$

% Note that $[-F(x,\cdot) \leq \min_{y'\in\mathcal{Y}} -F(x,y')] \cap \mathcal{Y} \subseteq  \mathcal{Y} \cap \mathcal{L}(x)$ always holds. For the reverse side, let $z \in \mathcal{Y} \cap \mathcal{L}(x)$, and then $\lambda_{i}(a_{i}^\top z-b_{i})=0$. On the other hand, since $F(x,\cdot)$ is linear, we have
% \[
% -\nabla_z F(x,z)+\sum_{i=1}^{l} \lambda_{i}(x) a_{i}=-\nabla_y F(x,\bar{y})+\sum_{i=1}^{l} \lambda_{i}(x) a_{i}=0.
% \]
% This implies that $z$ satisfies the KKT condition \eqref{y-linear-kkt}, and consequently $z$ is a global maximizer of \eqref{y-linear-P} (i.e., $z\in \min_{y'\in\mathcal{Y}} -F(x,y')] \cap \mathcal{Y}$) since \eqref{y-linear-P} is a linear program.

% With the help of \eqref{y-linear-key-imp} and noting that $\mathcal{Y}$ and $\mathcal{L}(x)$ are polyhedral sets, Lemma $3.4$ (BLR) implies that there exists an absolute $\tau_{0}>0$ such that
% \[
% {\rm dist}(y,[-F(x,\cdot) \leq \min_{y'\in\mathcal{Y}} -F(x,y')] \cap \mathcal{Y}) ={\rm dist}(y, \mathcal{Y} \cap \mathcal{L}(x)) \leq \tau_{0}\cdot({\rm dist}(y, \mathcal{Y})+{\rm dist}(y, \mathcal{L}(x))).
% \]
% It is worth noting that $\mathcal{L}(x)$ is a polyhedral set. 
Now, observe that as $x$ varies over $\mathcal{X}$, there are at most $2^l$ different sets of active constraints in problem \eqref{y-linear-P}. This, together with the complementarity condition in \eqref{y-linear-kkt}, implies that there are at most $2^l$ different polyhedra in the set $\{ \mathcal{L}(x) : x \in \mathcal{X} \}$. Thus, by  the linear regularity property of polyhedral sets \citep[Corollary 5.26]{bauschke1996projection},  there exists an absolute constant $\gamma_0>0$ such that for all $y\in\mathcal{Y}$,
\begin{equation}\label{y-linear-key-trans}
\begin{aligned}
\dist(y,\mathcal{Y}^{\star}(x))=\dist(y, \mathcal{Y} \cap \mathcal{L}(x))   \leq \gamma_0\cdot \dist(y, \mathcal{L}(x)).
\end{aligned}
\end{equation}
Similarly, by Hoffman's error bound \citep{hoffman1952approximate},  there exists an absolute constant $\gamma_1>0$ such that for all $y\in\mathcal{Y}$,
\begin{equation}\label{y-linear-key-hoff}
% \begin{aligned}
{\rm dist}(y, \mathcal{L}(x)) \leq \gamma_1\cdot \sum_{\lambda_i(x)>0}\left|a_{i}^\top y-b_{i}\right|  = \gamma_1\cdot \sum_{\lambda_i(x)>0}\frac{1}{\lambda_i(x)}\left|\lambda_{i}(x)\cdot(a_{i}^\top y-b_{i})\right|
.
% \end{aligned}
\end{equation}
% By the assumption 
% that for $\mathcal{I}(x,\bar{y})=\mathcal{I}(x)\subseteq[l]$ such that $G(x)=\sum_{i\in \mathcal{I}(x,\bar{y})} \lambda_{i}(x) a_{i}$, one has $\lambda_i(x)\ge\delta$ for $i\in \mathcal{I}(x,\bar{y})$. Consequently, if $\lambda_i(x)\neq 0$, we know that  
% $\lambda_i(x)\ge \delta$ for all $x\in\mathcal{X}$.
Since $\lambda_i(x)\ge \delta$ if $\lambda_i(x)\neq0$ by assumption, we get from \eqref{y-linear-key-trans} and \eqref{y-linear-key-hoff} that
\begin{equation}\label{eq:kl0-key1}
{\rm dist}\left(y,\mathcal{Y}^{\star}(x)\right)\leq \frac{\gamma}{\delta}\cdot \sum_{i=1}^{l}|\lambda_{i}(x)\cdot(a_{i}^\top y-b_{i})|,
\end{equation}
% Noting that $\phi_x(y) \leq \max_{y\in\mathcal{Y}} F(x,y)-F(x,y)$ for all $y \in \mathcal{Y}$, this together with (3.17) implies that
% $$
% \begin{aligned}
% {\rm dist}(y,[F(x,\cdot) \geq \max_{y\in\mathcal{Y}} F(x,y)] \cap \mathcal{Y}) \leq & \mu \tau_{0}\left(\max_{y\in\mathcal{Y}} F(x,y)-F(x,y)\right) \\
% &+\alpha \tau_{0} \sum_{i=1}^{m}|\lambda_{i}(a_{i}^\top y-b_{i})|\quad \text{for all} \ y \in \mathcal{Y}.
% \end{aligned}
% $$
where $\gamma:=\gamma_0\gamma_1$.
In addition, we have
\begin{equation}\label{eq:kl0-key2}
\begin{aligned}
& \sum_{i=1}^{l}|\lambda_{i}(x)\cdot(a_{i}^\top y-b_{i})|
=-\sum_{i=1}^{l} \lambda_{i}(x)\cdot(a_{i}^\top y-b_{i})\\
= &\ -y^\top G(x) +\sum_{i=1}^{l} \lambda_{i}(x)\cdot a_{i}^\top \bar{y} = -F(x,y)+ \max_{y' \in \mathcal{Y}} F(x,y'),
\end{aligned}
\end{equation}
where the second and last equalities follow from \eqref{y-linear-kkt}. Putting \eqref{eq:kl0-key1} and \eqref{eq:kl0-key2} together yields 
\begin{align*}
\ &\mu\cdot{\rm dist}\left (y,\mathcal{Y}^{\star}(x)\right)  \leq \max_{y'\in\mathcal{Y}} F(x,y')-F(x,y)
\end{align*}
 for all $y \in \mathcal{Y}$, 
where $\mu=\frac{\delta}{\gamma}$. The proof is then complete.
\end{proof}

\begin{corollary}[Max-structured problem]\label{coro:linear-2}
Consider problem \eqref{eq:linear}, where $\mathcal{Y} = \{ y \in \mathbb{R}^d: \sum_{i=1}^d y_i = 1, \ y \ge 0 \}$ is the standard simplex. Suppose there exists a $\delta>0$ such that for all $x\in\mathcal{X}$ and 
$j\in \mathcal{I}(x)$, 
% {\color{blue} $j\in [d]$ satisfying $y_j=0$ that}
\begin{equation}\label{regular-condition-linear}
\max_{i\in [d]}\, G_i(x)\ge G_j(x)+\delta.
\end{equation}
Then, Assumption \ref{ass:kl} holds with $\theta = 0$.
\end{corollary}

\begin{proof}
Let $x \in \mathcal{X}$ and $y^{\star}(x) \in \argmax_{y \in \mathcal{Y}} y^\top G(x)$ be arbitrary. Since $\mathcal{Y}$ is the standard simplex, the KKT condition \eqref{y-linear-kkt} implies that
\begin{align*}
G(x)
% &=\sum_{i\in \mathcal{I}(x)} \lambda_{i}(x) a_{i}
=(u(x)-\lambda_1(x),\ldots,u(x)-\lambda_{d}(x)),
\end{align*}
where $\lambda_i(x) \ge 0$ (resp., $u(x) \in \mathbb{R}$) is the dual variable associated with the constraint $y_i\ge 0$ for $i \in [d]$ (resp., $\sum_{i=1}^d y_i= 1$). 
If $y_i^{\star}(x) > 0$ for some $i \in [d]$, then $\lambda_i(x)=0$, which implies that $\max_{i\in [d]} G_i(x)=u(x)$. On the other hand, condition \eqref{regular-condition-linear} implies that for any $j \in [d]$ with $y_j^{\star}(x) = 0$, we have
\[\lambda_j(x)=u(x)-(u(x)-\lambda_j(x))=\max_{i\in [d]}\, G_i(x)- G_j(x)\ge\delta.\]
Therefore, by applying Theorem \ref{cor:linear}, we conclude that problem \eqref{eq:linear} satisfies Assumption \ref{ass:kl} with $\theta = 0$.
\end{proof}

\begin{remark} \label{rmk:linear}
{\rm (i)} The results in Theorem \ref{cor:linear} and Corollary \ref{coro:linear-2}, which require minimal assumption on the mapping $G$, concern the K{\L} exponent of problem~\eqref{eq:linear} and do not depend on the algorithm used to solve the problem. However, if we are interested in solving problem~\eqref{eq:linear} using our smoothed PLDA algorithm, then the mapping $G$ needs to possess additional structure, so that the objective function $F$ satisfies Assumption~\ref{ass:basic}. In particular, if $\mathcal{Y}$ is an arbitrary polytope, then the mapping $G$ should be continuously differentiable and have a Lipschitz continuous Jacobian. If $\mathcal{Y}$ is the standard simplex (which is the case in the max-structured problem), then the mapping $G$ can take the more general form $G(x)=(h_1(c_1(x)),\ldots, h_d(c_d(x)))$, where $h_i:\R^m\rightarrow \R$ is convex Lipschitz and $c_i:\R^n\rightarrow \R^m$ is continuously differentiable with a Lipschitz continuous Jacobian for $i \in [d]$.
{\rm (ii)} If the assumptions in Theorem \ref{cor:linear} and Corollary \ref{coro:linear-2} hold only locally around every GS of problem \eqref{eq:linear}, then it is still possible to show that Assumption~\ref{ass:kl} holds with $\theta=0$ locally around every GS of problem \eqref{eq:linear}. Such a local property is already sufficient to yield a convergence rate guarantee similar to that in Theorem \ref{thm:general} for smoothed PLDA when applied to problem \eqref{eq:linear}; cf.~Remark \ref{rmk-main}{\rm (ii)}.
% The local version of the assumptions around any GS of problem \eqref{eq:linear} in Theorem \ref{cor:linear} and Corollary \ref{coro:linear-2} ensures that the K{\L} property holds locally around any GS, cf. Assumption \ref{ass:kl}. This condition is sufficient to guarantee the convergence results in Theorem \ref{thm:general} for problem \eqref{eq:linear}, as discussed in Remark \ref{rmk-main}{\rm (ii)}.
{\rm (iii)} The work \citep{zhang2020single} establishes a dual error bound for the max-structured problem under the assumptions that the primal function satisfies the gradient Lipschitz continuity condition and certain strict complementarity condition (see \citep[Assumption 3.5]{zhang2020single}) holds. It can be shown that the latter implies condition \eqref{regular-condition-linear}  in Corollary \ref{coro:linear-2}. However, it should be noted that the proof technique used in \citep{zhang2020single} cannot be extended to the nonsmooth,  possibly even discontinuous, setting considered here. 
\end{remark}

\section{Quantitative Relationships among Different Stationarity Concepts}
\label{sec:sta}

A fundamental question in the study of minimax optimization is how to define the concept of stationarity. One approach is to consider various natural optimality conditions of the minimax problem and extract from them the corresponding stationarity concepts. In the nonconvex-nonconcave setting, the different stationarity concepts obtained via this approach may not coincide, and their relationships are still not well understood. In this section, we aim to elucidate the relationships among several well-known stationarity concepts for problem \eqref{eq:problem}. Interestingly, the dual error bound we developed in Corollary \ref{prop:dual_eb_KL2} plays an important role in obtaining our results.

To begin, let us consider the following exact stationarity concepts for problem  \eqref{eq:problem}:
\begin{definition}[see, e.g.,
\citep{jin2020local,razaviyayn2020nonconvex}]
\label{def:sta} 
The point $(x^\star, y^\star)\in \mathcal{X}\times \mathcal{Y}$ is called a
\begin{enumerate}[label={\rm (\alph*)}]
\item {\it minimax point} {\rm (MM)} of problem \eqref{eq:problem} if \[F(x^{\star}, y) \leq F(x^{\star}, y^{\star}) \leq \max _{y^{\prime} \in \mathcal{Y}} F(x, y^{\prime}) \quad \text{for any } (x,y)\in\mathcal{X}\times\mathcal{Y};\]
% $x^\star \in \underset{x \in \mathcal{X}}{\argmin}\ f(x) = \underset{x \in \mathcal{X}}{\operatorname{argmin}}\left(\max\limits_{y \in \mathcal{Y}} F(x, y)\right)$;
\item {\it game-stationary} point {\rm (GS)} of problem \eqref{eq:problem} if 
\[
0\in \partial_{x} F(x^\star, y^\star)+\partial\iota_{\mathcal{X}}(x^\star)\quad \text{and}\quad  0\in -\nabla_{y} F(x^\star, y^\star)+\partial\iota_{\mathcal{Y}}(y^\star).
\]
\end{enumerate}
Furthermore, the point $x^\star\in \mathcal{X}$ is called an
\begin{enumerate}[label={\rm (c)}]
\item {\it optimization-stationary} point {\rm (OS)} of problem \eqref{eq:problem} if 
\[
0\in
\partial (f+\iota_{\mathcal{X}})(x^\star)=
\partial\left(\max_{y \in \mathcal{Y}} F(\cdot, y)+\iota_{\mathcal{X}}\right) (x^\star).
\]
% and $F(x^{\star}, y) \leq F(x^{\star}, y^{\star})$ for any $y\in \mathcal{Y}$.
\end{enumerate}
\end{definition}
The extreme value theorem guarantees the existence of an MM when $\mathcal{X}$ is compact, even if $F(\cdot,\cdot)$ is nonconvex-nonconcave. In addition,  the weak convexity of $F(\cdot,y)$ and $-F(x,\cdot)$ for any $x \in \mathcal{X}$ and $y \in \mathcal{Y}$ ensures the existence of a GS \citep[Proposition 4.2]{pang2016unified}. 
% However, it is worth noting that finding a \emph{local} MM (which may not even be an MM) of a constrained nonconvex-nonconcave optimization problem with a smooth objective is already hard \citep{daskalakis2021complexity}. 
However, it is worth noting that finding an MM of a nonconvex-nonconcave optimization problem is hard, as it includes the task of finding a global maximum of a nonconcave function as special case.
Also, since
 most applications in machine learning, such as GANs and adversarial training, involve sequential games, the optimization-stationarity concept may be more attractive.

While the above definition of OS is standard, it only characterizes the primal variable $x$. As such, the concept of OS is not directly comparable to that of MM or GS. To circumvent this difficulty, we introduce the concept of an extended OS (eOS). Specifically, the point $(x^\star, y^\star)\in \mathcal{X}\times \mathcal{Y}$ is called an eOS of problem \eqref{eq:problem} if $x^\star$ is an OS and  $F(x^{\star}, y) \leq F(x^{\star}, y^{\star})$ for any $y\in \mathcal{Y}$. From Definition \ref{def:sta}, we can immediately conclude that every MM is an eOS. To gain some intuition on the possible
 relationships among the three types of stationarity points MM, GS, and eOS, it is instructive to consider the following example.
\begin{figure}[ht]
\centering

\tikzset{every picture/.style={line width=0.75pt}} %set default line width to 0.75pt        

\begin{tikzpicture}[x=0.75pt,y=0.75pt,yscale=-0.75,xscale=0.75]
%uncomment if require: \path (0,235); %set diagram left start at 0, and has height of 235

%Shape: Circle [id:dp9899845750205545] 
\draw   (67,101.44) .. controls (67,64.68) and (96.8,34.89) .. (133.56,34.89) .. controls (170.32,34.89) and (200.11,64.68) .. (200.11,101.44) .. controls (200.11,138.2) and (170.32,168) .. (133.56,168) .. controls (96.8,168) and (67,138.2) .. (67,101.44) -- cycle ;
%Shape: Circle [id:dp342967870577243] 
\draw   (150.11,101.44) .. controls (150.11,87.64) and (161.31,76.44) .. (175.11,76.44) .. controls (188.92,76.44) and (200.11,87.64) .. (200.11,101.44) .. controls (200.11,115.25) and (188.92,126.44) .. (175.11,126.44) .. controls (161.31,126.44) and (150.11,115.25) .. (150.11,101.44) -- cycle ;
%Shape: Circle [id:dp5010444066239879] 
\draw   (112.5,101.44) .. controls (112.5,77.25) and (132.11,57.64) .. (156.31,57.64) .. controls (180.5,57.64) and (200.11,77.25) .. (200.11,101.44) .. controls (200.11,125.64) and (180.5,145.25) .. (156.31,145.25) .. controls (132.11,145.25) and (112.5,125.64) .. (112.5,101.44) -- cycle ;
%Shape: Circle [id:dp9818357856582633] 
\draw   (273,103.44) .. controls (273,66.68) and (302.8,36.89) .. (339.56,36.89) .. controls (376.32,36.89) and (406.11,66.68) .. (406.11,103.44) .. controls (406.11,140.2) and (376.32,170) .. (339.56,170) .. controls (302.8,170) and (273,140.2) .. (273,103.44) -- cycle ;
%Shape: Circle [id:dp6184161239411539] 
\draw   (280.5,106.07) .. controls (280.5,86.08) and (296.7,69.89) .. (316.68,69.89) .. controls (336.67,69.89) and (352.86,86.08) .. (352.86,106.07) .. controls (352.86,126.05) and (336.67,142.25) .. (316.68,142.25) .. controls (296.7,142.25) and (280.5,126.05) .. (280.5,106.07) -- cycle ;
%Shape: Circle [id:dp6798860747162185] 
\draw   (328.5,104.07) .. controls (328.5,84.08) and (344.7,67.89) .. (364.68,67.89) .. controls (384.67,67.89) and (400.86,84.08) .. (400.86,104.07) .. controls (400.86,124.05) and (384.67,140.25) .. (364.68,140.25) .. controls (344.7,140.25) and (328.5,124.05) .. (328.5,104.07) -- cycle ;
%Shape: Circle [id:dp34784353545953195] 
\draw   (472,101.44) .. controls (472,64.68) and (501.8,34.89) .. (538.56,34.89) .. controls (575.32,34.89) and (605.11,64.68) .. (605.11,101.44) .. controls (605.11,138.2) and (575.32,168) .. (538.56,168) .. controls (501.8,168) and (472,138.2) .. (472,101.44) -- cycle ;
%Shape: Circle [id:dp9707000523285083] 
\draw   (555.11,101.44) .. controls (555.11,87.64) and (566.31,76.44) .. (580.11,76.44) .. controls (593.92,76.44) and (605.11,87.64) .. (605.11,101.44) .. controls (605.11,115.25) and (593.92,126.44) .. (580.11,126.44) .. controls (566.31,126.44) and (555.11,115.25) .. (555.11,101.44) -- cycle ;
%Shape: Circle [id:dp43592558005738113] 
\draw   (517.5,101.44) .. controls (517.5,77.25) and (537.11,57.64) .. (561.31,57.64) .. controls (585.5,57.64) and (605.11,77.25) .. (605.11,101.44) .. controls (605.11,125.64) and (585.5,145.25) .. (561.31,145.25) .. controls (537.11,145.25) and (517.5,125.64) .. (517.5,101.44) -- cycle ;

% Text Node
% \draw (721,21) node    {$0$};
% % Text Node
% \draw (701,71) node    {$0$};
% Text Node
\draw (158,93) node [anchor=north west][inner sep=0.75pt]   [align=left] {MM};
% Text Node
\draw (119,93) node [anchor=north west][inner sep=0.75pt]   [align=left] {GS};
% Text Node
\draw (70,93) node [anchor=north west][inner sep=0.75pt]   [align=left] {eOS};
% Text Node
\draw (122,182) node [anchor=north west][inner sep=0.75pt]    {$( a)$};
% Text Node
\draw (358,96) node [anchor=north west][inner sep=0.75pt]   [align=left] {MM};
% Text Node
\draw (296,96) node [anchor=north west][inner sep=0.75pt]   [align=left] {GS};
% Text Node
\draw (325,144) node [anchor=north west][inner sep=0.75pt]   [align=left] {eOS};
% Text Node
\draw (328,184) node [anchor=north west][inner sep=0.75pt]    {$( b)$};
% Text Node
\draw (528,181) node [anchor=north west][inner sep=0.75pt]    {$( c)$};
% Text Node
\draw (567,93) node [anchor=north west][inner sep=0.75pt]   [align=left] {GS};
% Text Node
\draw (517,93) node [anchor=north west][inner sep=0.75pt]   [align=left] {MM};
% Text Node
\draw (473,93) node [anchor=north west][inner sep=0.75pt]   [align=left] {eOS};

\end{tikzpicture}

    \caption{Possible relationships among the three types of stationarity points MM, GS, and eOS: (a)  $F(x,y)=x^3-2xy-y^2$ with $\mathcal{X}\times\mathcal{Y}=[-1,1]\times [-1,1]$; (b) $F(x,y)=\sin(x)y$ with $\mathcal{X}\times\mathcal{Y}=[-\frac{\pi}{2},\frac{\pi}{2}] \times [-1,1]$;  (c) $F(x,y)=xy$ with $\mathcal{X}\times\mathcal{Y}=\R \times [-1,1]$.}
    \label{fig:staionary_concepts}
\end{figure}
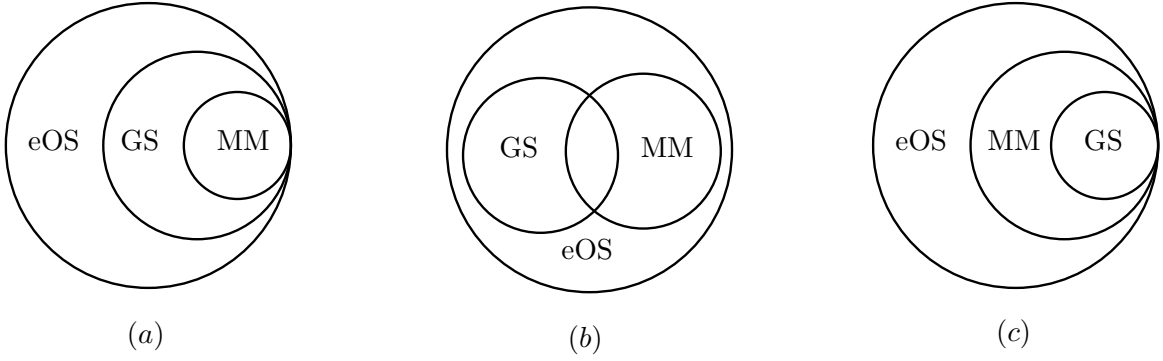
\begin{example}
The following instances of problem \eqref{eq:problem} give rise to the relationships among MM, GS, and eOS shown in Figure \ref{fig:staionary_concepts}. 
\begin{enumerate}[label=\normalfont(\alph*)]
  \item    If $F(x,y)=x^3-2xy-y^2$ and  $\mathcal{X}\times\mathcal{Y}=[-1,1]\times [-1,1]$, then
      \[
      \mathrm{MM} = \left\{(-1,1), (0,0)\right\}, \ 
      \mathrm{GS} =  \mathrm{eOS} = \left\{(-1,1), \left(-\frac{2}{3},\frac{2}{3}\right), (0,0)\right\}.  
    %   \text{OS} = \left\{(-1,1), \left(-\frac{2}{3},\frac{2}{3}\right), (0,0)\right\};
      \]
\item If $F(x,y)=\sin(x)y$ and $\mathcal{X}\times\mathcal{Y}=[-\frac{\pi}{2},\frac{\pi}{2}] \times [-1,1]$, then
      \begin{align}
      &\mathrm{MM} = \left\{(0,[-1,1])\right\},\  
      \mathrm{GS} = \left\{\left(-\frac{\pi}{2},-1\right), (0,0), \left(\frac{\pi}{2},1\right)\right\}, \notag\\
      &\mathrm{eOS} = \left\{\left(-\frac{\pi}{2},-1\right), (0,[-1,1]), \left(\frac{\pi}{2},1\right)\right\}. \notag
      \end{align}
\item If $F(x,y)=xy$ and $\mathcal{X}\times\mathcal{Y}=\R \times [-1,1]$, then
      \[
      \mathrm{GS} = \left\{(0,0)\right\}, \ 
      \mathrm{MM} =  \mathrm{eOS} = \left\{(0,[-1,1])\right\}.  
    %   \text{OS} = \left\{(0,[-1,1])\right\};
      \]
  \end{enumerate}
\end{example}

As we can see from Figure \ref{fig:staionary_concepts}, the set of GS can fail to capture that of MM. Now, in view of the hardness of computing an MM, let us focus on expounding the relationship between the concepts of game stationarity and optimization stationarity. In fact, we shall address the more general problem of relating the approximate versions of these concepts as introduced in Definition \ref{defi:primal-dual}.

To begin, let $x \in \mathcal{X}$ be arbitrary. Observe that since $f + \iota_{\mathcal{X}}$ is $L$-weakly convex on $\mathcal{X}$ (see Fact \ref{fact-1} in Appendix \ref{sec:lemmas}) and $r>L$ by assumption, we have
\[\operatorname{dist}\left(0, \partial (f+\iota_{\mathcal{X}})\left(\prox_{\frac{1}{r}f+\iota_{\mathcal{X}}}(x)\right)\right) \leq r\left\|\prox_{\frac{1}{r}f+\iota_{\mathcal{X}}}(x)-x\right\|;\]
see, e.g., \citep[Section 2.2]{davis2019stochastic}. It follows from Definition \ref{defi:primal-dual} that for problem \eqref{eq:problem}, a $0$-OS is an OS. Similarly, since $F(\cdot,y) + \iota_{\mathcal{X}}(\cdot)$ is $L$-weakly convex on $\mathcal{X}$ for any $y \in \mathcal{Y}$, we have
% If $\epsilon = 0$ in Definition \ref{def:sta}, then $\prox_{\frac{1}{r}f+\iota_{\mathcal{X}}}(x)=x$ and thus $0\in\partial (f+\iota_{\mathcal{X}})(x)$.
% For $\epsilon$-GS, due to \citep[Theorem 10.31]{rockafellar2009variational}, we have
% \[\nabla_{z} d_r(y,x)=r\left(x-\prox_{\frac{1}{r}F(\cdot, y)+\iota_{\mathcal{X}}}(x)\right)\in\partial (F(\cdot,y)+\iota_{\mathcal{X}})\left(\prox_{\frac{1}{r}F(\cdot, y)+\iota_{\mathcal{X}}}(x)\right).\]
\begin{align*}
\operatorname{dist}\left(0 , \partial (F(\cdot,y)+\iota_{\mathcal{X}})\left(\prox_{\frac{1}{r}F(\cdot, y)+\iota_{\mathcal{X}}}(x)\right)\right) 
&\leq r \left\|\prox_{\frac{1}{r}F(\cdot, y)+\iota_{\mathcal{X}}}(x)-x\right\| =\left\|\nabla_{z} d_r(y,x)\right\|,
\end{align*}
where the last equality is due to Lemma \ref{lemma-dualdiff} in Appendix \ref{sec:lemmas}. It follows from Definition \ref{defi:primal-dual} that for problem \eqref{eq:problem}, a $0$-GS is a GS. Based on the above, we may regard
 $\epsilon$-OS and $\epsilon$-GS as smoothed surrogates of OS and GS, respectively.

\begin{remark}
When $F(x,\cdot)$ is concave for any $x\in\mathcal{X}$,
each OS $x^{\star}$ has a corresponding GS whose $x$-coordinate is $x^{\star}$. To be more precise, if $x^{\star}$ is an OS, then we have $x_r^{\star}(x^{\star}) = \prox_{\frac{1}{r}f+\iota_{\mathcal{X}}}(x^\star)=x^\star$. Consequently, for any $y \in \mathcal{Y}$, we can bound the game-stationarity measure of $(x^{\star}, y)$ as
% we know that
% $x_r^\star(x^\star)=x_r(y(x^\star),x^\star)$,
% where $y(x^\star)\in {\color{blue}\argmax F(x^\star,y)+\frac{r}{2}\|x^\star-x^\star\|^2=\argmax_y F(x^\star,y)}$.
%  Then for any $y$, it follows from the  Lipschitz continuity of the gradient w.r.t. $y$ that
%     $$\|x_r(y,x^\star)-x_r^\star(x^\star)\|\leq c\cdot \|y- y(x^\star)\|.$$
    % for some $y(x^\star)\in {\color{blue}\argmax F(x^\star,y)+\frac{r}{2}\|x^\star-x^\star\|^2=\argmax_y F(x^\star,y)}$.
    % it follows that
    \begin{align*}
  & \ \frac{1}{r}\|\nabla_{z} d_r(y,x^\star)\|  
  = \|x_r(y,x^\star)-x^\star\| \notag\\
   \leq & \ \|x_r(y,x^\star)-x_r(y_+(x^\star),x^\star)\|+\|x_r(y_+(x^\star),x^\star)-x_r^\star(x^\star)\|  +   \|\prox_{\frac{1}{r}f+\iota_{\mathcal{X}}}(x^\star)-x^\star\|\\
   \leq & \  \sigma_2 \|y-y_+(x^\star)\|+\kappa^{\frac{1}{2}} \|y-y_+(x^\star)\|^{\frac{1}{2}} \\
   = & \ \mathcal{O} \left({\rm dist}( 0, -\nabla_y d_r(y,x^{\star}) + \partial \iota_{\mathcal{Y}}(y))+\dist^{\frac{1}{2}}(0,-\nabla_y d_r(y,x^\star)+\partial\iota_{\mathcal{Y}}(y))  \right),
   % &\leq c\cdot \|y- y(x^\star)\|+\|\prox_{\frac{1}{r}f+\iota_{\mathcal{X}}}(x^\star)-x^\star\|
\end{align*}    
where the second inequality follows from Lemma \ref{lemma-sollip}, Lemma \ref{lemma-dual-bd}, and the fact that $\prox_{\frac{1}{r}f+\iota_{\mathcal{X}}}(x^\star)=x^\star$, and the last equality is due to \citep[Lemma 4.1]{li2018calculus} and Lemma \ref{lemma-dualdiff} in Appendix \ref{sec:lemmas}.   
% that OS and GS have the same $x$ coordinate. 
% Also, from $\prox_{\frac{1}{r}f+\iota_{\mathcal{X}}}(x^\star)=x^\star$ one has that
% $$
% % \max_y\min_x F(x,y)+\frac{r}{2}\|x-x^\star\|^2=
% \min_x\max_y F(x,y)+\frac{r}{2}\|x-x^\star\|^2=\max_y F(x^\star,y).$$
% Then
% Since $d_r(\cdot,x^\star)$ is a differentiable function, we know from Weierstrass' theorem that there exists $y(x^\star)\in\mathcal{Y}$ such that 
% \[
% 0\in -\nabla_y d_r(y(x^\star),x^\star)+\partial \iota_{\mathcal{Y}}(y(x^\star))=-\nabla_{y} F(x_r(y(x^\star),x^\star),y(x^\star))+\partial \iota_{\mathcal{Y}}(y(x^\star))
% \]
% Then,
% % Since minmax changable we know that $F(x^\star,y(x^\star))=\max d_r(y,x^\star)=\min_x\max_y F(x,y)+\frac{r}{2}\|x-x^\star\|^2 $
% due to the concavity of $F(x,\cdot)$ for any $x\in\mathcal{X}$, we have 
Now, observe that
\begin{align*}
\max_{y\in \mathcal{Y}} d_r(y,x^\star) & \ =\max_{y\in \mathcal{Y}}\min_{x\in\mathcal{X}} \left\{F(x,y)+\frac{r}{2}\|x-x^\star\|^2\right\} \\ &\ =\min_{x\in\mathcal{X}}\max_{y\in \mathcal{Y}} \left\{F(x,y)+\frac{r}{2}\|x-x^\star\|^2\right\}=\max_{y\in \mathcal{Y}} F(x^\star,y),
\end{align*}
where the second equality is due to the convexity of $F_r(\cdot,y,x^{\star})$ for any $y \in \mathcal{Y}$ (recall that $r>L$ by assumption) and concavity of $F(x,\cdot)$ for any $x \in \mathcal{X}$, and the last equality is from $\prox_{\frac{1}{r}f+\iota_{\mathcal{X}}}(x^\star)=x^\star$.
Since an optimal solution $y^{\star} (x^{\star})$ of the above problem satisfies
% $0\in -\nabla_{y} F(x^\star, y(x^\star))+\partial \iota_{\mathcal{Y}}(y(x^\star))$ (since $y(x^\star)\in\argmax_y F(x^\star,y)$) implies that $(x^\star, y(x^\star))$ is a GS.
% Then
% $y(x^\star)\in\argmax_y \min_x F(x,y)+\frac{r}{2}\|x-x^\star\|^2=\argmax_y F(x^\star,y)$, 
% which satisfies 
\[
0\in -\nabla_y d_r(y^\star(x^\star),x^\star)+\partial\iota_{\mathcal{Y}}(y^\star(x^\star))= -\nabla_{y} F(x^\star, y^\star(x^\star))+\partial \iota_{\mathcal{Y}}(y^\star(x^\star))
\]
and thus also $\|\nabla_{z} d_r(y^\star(x^\star),x^\star)\|=0$, we conclude that $(x^\star, y^\star(x^\star))$ is a GS.
\end{remark}

Now, let us state the main result in this section, which establishes a quantitative, algorithm-independent relationship between the $\epsilon$-game-stationarity and $\epsilon$-optimization-stationarity concepts for any $\epsilon \in [0,1)$. The key tool used is the alternative dual error bound (Corollary \ref{prop:dual_eb_KL2}) presented in Section \ref{sec:dual-error}.
\begin{theorem}
\label{thm-stationary}
Suppose that $(x,y)\in\mathcal{X}\times\mathcal{Y}$ is 
an $\epsilon$-GS of problem \eqref{eq:problem} for some $\epsilon \in [0,1)$. Then, the point $x$ is 
% an $\mathcal{O}(\epsilon^{\min\{1,\frac{1}{2\theta}\}})$-OS point. In the case where the dual function is general concave, $x$ is an $\mathcal{O}(\epsilon^{\frac{1}{2}})$-OS point.
% Assumption \ref{ass:kl} holds
\begin{enumerate}[label=\normalfont(\alph*)]
  \item \rm{(K\L\ exponent $\theta\in[0,1)$):} an $\mathcal{O}(\epsilon^{\min\{1,\frac{1}{2\theta}\}})$-OS of problem \eqref{eq:problem};
\item \rm{(Concave):} an $\mathcal{O}(\epsilon^{\frac{1}{2}})$-OS of problem \eqref{eq:problem}.
  \end{enumerate}
\end{theorem}

\begin{proof}
For case (a), we first compute
\begin{align}\label{GS-OS-key1}
  \|\prox_{\frac{1}{r}f+\iota_{\mathcal{X}}}(x)-x\|
   &= \|x_r^\star(x)-x\|  \notag\\
   &\leq \|x_r^\star(x)-x_r(y_+(x),x)\|+\|x_r(y_+(x),x)-x_r(y,x)\| +\|x_r(y,x)-x\| \notag\\
   &\leq \omega \|y-y_+(x)\|^{\gamma} + \sigma_2 \|y-y_+(x)\|+\frac{1}{r}\|\nabla_{z} d_r(y,x)\|, 
\end{align}
where 
% $\omega$ is equal to $\omega_2$ (resp. $\omega_1$) when $\theta\in(0,1)$ (resp. $\theta=0$) and 
the last inequality follows from Corollary \ref{prop:dual_eb_KL2}, Lemma \ref{lemma-sollip}, and Lemma \ref{lemma-dualdiff} with $\omega$ being equal to $\omega_2$ (resp., $\omega_1$) and $\gamma$ being equal to $\tfrac{1}{2\theta}$ (resp., $1$) when $\theta \in (0,1)$ (resp., $\theta=0$).

Next, we estimate $\|y-y_+(x)\|$ in terms of  the game-stationarity measure of $(x,y)$. 
Let $y_{\#}(x) := \proj_{\mathcal{Y}}(y+\alpha \nabla_y F(x,y))$. By \citep[Lemma 4.1]{li2018calculus}, we have
\[
\|y-y_{\#}(x)\|\leq {\rm dist}(0,-\nabla_{y}F(x, y)+\partial\iota_{\mathcal{Y}}(y)).
\]
Moreover, as $\mathcal{Y}$ is a convex set, we have
\begin{align*}
   &\|y_{\#}(x)-y_+(x)\|\notag\\ 
    =\ &   \|\proj_{\mathcal{Y}}\left(y+\alpha \nabla_y F(x,y)\right)-\proj_{\mathcal{Y}}\left(y+\alpha \nabla_{y}F( x_r(y,x), y )\right)\| \\
    \leq\ &  \alpha \|\nabla_y F(x,y)-\nabla_{y} F(x_r(y, x), y)\|\\
    \leq\ & \alpha L\|x_r(y, x)-x\|,
\end{align*} 
where the first inequality holds due to the nonexpansiveness of $\proj_{\mathcal{Y}}$ and the second inequality is due to the $L$-Lipschitz continuity of $\nabla_{y} F(\cdot, y)$ on $\mathcal{X}$. Putting the above together yields
\begin{align}
   \|y-y_+(x)\|  
   & \leq \|y-y_{\#}(x)\| + \|y_{\#}(x)-y_+(x)\|\notag\\
&\leq {\rm dist}(0,-\nabla_{y} F(x, y)+\partial\iota_{\mathcal{Y}}(y))+\alpha L\|x_r(y, x)-x\|. \label{GS-OS-key2}
\end{align} 
Since $(x,y)\in\mathcal{X}\times\mathcal{Y}$ is an $\epsilon$-GS, we have
\[
\|\nabla_{z} d_r(y,x)\|\leq \epsilon\quad \text{and}\quad {\rm dist}(0,-\nabla_{y} F(x, y)+\partial\iota_{\mathcal{Y}}(y))\leq\epsilon.
\]
It follows that the optimization-stationarity measure of $x$ can be bounded as
% and we recall that $\nabla_{x} d_r(y,x)=r(x-x_r(y,x))$, then we can conclude from \eqref{GS-OS-key1} and \eqref{GS-OS-key2} that:
\begin{align}
& \|\prox_{\frac{1}{r}f+\iota_{\mathcal{X}}}(x)-x\|\notag\\
 \leq\ & \omega \|y-y_+(x)\|^{\gamma} + \sigma_2 \|y-y_+(x)\|+\frac{1}{r}\|\nabla_{z} d_r(y,x)\|\notag\\
  \leq\ & \omega \left({\rm dist}(0,-\nabla_{y} F(x, y)+\partial\iota_{\mathcal{Y}}(y))+\frac{\alpha L}{r}\|\nabla_{z} d_r(y,x)\|\right)^{\gamma}\notag\\
  & +\sigma_2\left({\rm dist}(0,-\nabla_{y} F(x, y)+\partial\iota_{\mathcal{Y}}(y))+\frac{\alpha L}{r}\|\nabla_{z} d_r(y,x)\|\right)+\frac{1}{r}\|\nabla_{z} d_r(y,x)\|\notag\\
  \leq\ & \omega\left(1+\frac{\alpha L}{r}\right)^{\gamma}\epsilon^{\gamma}+\sigma_2\left(1+\frac{\alpha L}{r}\right)\epsilon+\frac{\epsilon}{r}= \mathcal{O}(\epsilon^{\min\{1,\frac{1}{2\theta}\}}),\notag
\end{align}
where the first inequality follows from \eqref{GS-OS-key1} and the second from \eqref{GS-OS-key2}.
% $${\color{red}\|\nabla_{z} d_r(y,x)\|\leq \epsilon\quad \text{and}\quad {\rm dist}(0,\nabla_{y} F(x, y)+\partial\iota_{\mathcal{Y}}(y))\leq\epsilon\quad \Longrightarrow\quad  \|\prox_{\frac{1}{r}f+\iota_{\mathcal{X}}}(x)-x\|\leq ?\epsilon,}$$
This implies that $x$ is an $\mathcal{O}(\epsilon^{\min\{1,\frac{1}{2\theta}\}})$-OS of problem \eqref{eq:problem}. 

Now, for case (b), we can apply Lemma \ref{lemma-dual-bd} to derive the inequality \eqref{GS-OS-key1} with 
% $\theta$ with $1$ and $\omega$ with $\sqrt{\kappa}$. 
$\omega$ and $\gamma$ being replaced by $\sqrt{\kappa}$ and $\tfrac{1}{2}$, respectively.
The remainder of the proof follows the same argument as that for case (a). 
% The proof is then complete.
\end{proof}
\begin{remark}
Theorem \ref{thm-stationary} expands on the findings of Propositions 4.11 and 4.12 in \citep{lin2020gradient} by covering a broader range of scenarios. Specifically, it applies to settings where {\rm (i)} $\mathcal{X}$ is not necessarily the entire space $\mathbb{R}^n$; {\rm (ii)} the primal function is nonsmooth or lacks gradient Lipschitz continuity; {\rm (iii)} the dual problem may not be concave and only satisfies the K\L{} property with exponent $\theta$. Furthermore, our results apply to the problems studied in \citep{zhang2020single,yang2022faster}. 
\end{remark}

\section{Closing Remarks}
\label{sec:conclu}
In this paper, we proposed smoothed PLDA for solving a class of nonsmooth composite nonconvex-nonconcave problems. When the dual function is concave,  we showed that our algorithm can find both an $\epsilon$-GS and an $\epsilon$-OS of the problem in $\mathcal{O}(\epsilon^{-4})$ iterations, which is a first step towards matching the complexity results for smooth minimax problems. Moreover, when the dual problem possesses the K\L{} property with exponent $\theta\in[0,1)$, we showed that our algorithm has an  iteration complexity of $\mathcal{O}(\epsilon^{-2\max\{2\theta,1\}})$. As it turns out, this complexity is determined by the slower of the primal and dual updates in smoothed PLDA and reveals an interesting phase transition phenomenon: When $\theta \in [0,\frac{1}{2}]$, the primal update, which involves solving a strongly convex problem, dominates. As such, we cannot break the optimal iteration complexity of $\mathcal{O}(\epsilon^{-2})$. On the other hand, when $\theta \in (\frac{1}{2},1)$, the dual update dominates and explicitly depends on $\theta$, resulting in an iteration complexity of $\mathcal{O}(\epsilon^{-4\theta})$. The insights gained from our analysis suggest a new algorithm design principle---namely, primal-dual balancing---which holds promise for the design of more efficient algorithms in the context of minimax optimization.
% \vspace{2mm }

Our work suggests several directions for further study. We mention some of them here.
\begin{enumerate}[label={\rm (\alph*)}]
\item Extend our algorithm and its analysis to the stochastic setting to benefit modern machine learning tasks.
\item Investigate the lower complexity bounds of first-order methods for nonconvex-K\L{} problems and their dependence on the K{\L} exponent  $\theta$. We conjecture that our proposed smoothed PLDA already has the optimal complexity, at least in terms of the dependence on $\theta$.
\item Identify additional structured nonconvex-nonconcave problems that satisfy the  K\L{} property and characterize their K\L{} exponents.
\end{enumerate}

\bibliography{ref}
\bibliographystyle{abbrvnat}

\newpage 
\begin{appendices}
\section {Useful Technical Lemmas} \label{sec:lemmas}
To begin, we introduce the concept of weakly convex functions, which plays an important role in our subsequent analysis.
\begin{definition}[Weak convexity]
\label{defi:weak-convex}
 A function $\ell:\R^n \rightarrow \R$ is said to be $\rho$-weakly convex on a set $\mathcal{X}\subseteq \R^n$ for some constant $\rho \ge 0$ if for any $x, y \in \mathcal{X}$ and $\tau \in [0,1]$, we have
$$
\ell(\tau x+(1-\tau) y) \leq \tau \ell(x)+(1-\tau) \ell(y)+\frac{\rho \tau(1-\tau)}{2}\|x-y\|^{2}.
$$
% When $\ell$ is locally Lipschitz, 
The above definition is equivalent to the convexity of the function $\ell(\cdot) + \tfrac{\rho}{2} \| \cdot \|^2$ on $\mathcal{X}$.
\end{definition}
\begin{definition}[Proximal mapping]
% Consider a $\rho$-weakly convex function 
Let $\ell:\R^n\rightarrow \R \cup \{+\infty\}$ 
% and fix a parameter $\mu < \rho^{-1}$
be a proper lower-semicontinuous function.  The proximal mapping of $\ell$ with parameter $\mu>0$ at the point $x \in \mathbb{R}^n$ is defined by
\[
\prox_{\mu \ell}(x) = \mathop{\argmin}_{y\in \R^n} \left\{ \ell(y) + \frac{1}{2\mu} \|y-x\|^2\right\}. 
\]
\label{defi:prox}
\end{definition}
By our assumptions on problem \eqref{eq:problem} (recall that $L=L_h L_c$ and $r> L$) and \citep[Lemma 3.2 and Lemma 4.2]{drusvyatskiy2019efficiency}, we have the following useful results:
\begin{fact}\label{fact-1}
The functions $F(\cdot, y)$ for any $y\in\mathcal{Y}$ and $f$ are $L$-weakly convex on $\mathcal{X}$.
\end{fact}
\begin{fact}\label{fact-2}
Let $y\in\R^d$ and $r>0$ be given. For all $x, \bar{x} \in\mathcal{X}$, we have
\[
-\frac{r^{-1}+L}{2}\|x-\bar{x}\|^{2} \leq F(x,y)-F_{\bar{x},r}(x, y) \leq \frac{-r^{-1} + L}{2}\|x-\bar{x}\|^{2} .
\]
\end{fact}
Using Definition \ref{defi:weak-convex} and Fact \ref{fact-1}, we can deduce that $F_r(\cdot,y,z)$ is $(r-L)$-strongly convex for any $y \in \mathcal{Y}$ and $z \in \mathbb{R}^n$.

% \begin{lemma}
% \label{lm:sup_weak}
% If $f_i$ is $\rho_i$-weakly convex for all $i\in[N]$ and $\mathcal{A}$ is a bounded set with ${\rm diam}(\mathcal{A}) \leq B$, then the function $\sup\limits_{y \in \mathcal{A}}\sum_{i=1}^N y_i f_i$ is $\rho$-weakly  convex with  $\rho =B\max\limits_{i\in[N]} \rho_i$.
% \end{lemma}

Lemma \ref{lemma-sollip} and Lemma \ref{lemma-dualdiff} are also essential to  our analysis. The proofs of these lemmas are similar to those of Lemma B.2 and  Lemma B.3 in \citep{zhang2020single}, respectively. For completeness, we present the proofs here.

\begin{lemma}\label{lemma-sollip}
For any $y, y^{\prime} \in \mathcal{Y}$ and $z, z^{\prime} \in \R^n$, we have 
    \begin{align}
        &\|x_r(y, z)-x_r(y, z^{\prime})\| \leq \sigma_{1}\|z-z^{\prime}\|, \label{lip-z}\\
        &\|x_r^\star(z)-x_r^\star(z^{\prime})\| \leq \sigma_{1}\|z-z^{\prime}\|, \label{lip-starz}\\
        &\|x_r(y, z)-x_r(y^{\prime}, z)\| \leq \sigma_{2}\|y-y^{\prime}\|, \label{lip-y}
        \end{align}
     where $\sigma_{1}:=\frac{r}{r-L}$ and $\sigma_{2}:=\frac{2r-L}{r-L}$.
    \end{lemma}

\begin{proof}
From the definition of $F_r$, we know that for any $x\in\mathcal{X}$, $y\in\mathcal{Y}$, and $z,z'\in\R^n$, 
\begin{equation}\label{func-cal}
F_r(x,y,z^{\prime})-F_r(x,y,z)=\frac{r}{2}(\|x-z'\|^2-\|x-z\|^2)=\frac{r}{2}(\|z^{\prime}\|^{2}-2(z^{\prime}-z)^\top x-\|z\|^{2}).
\end{equation}
Fact \ref{fact-1} implies that $F(\cdot, y)$ is $L$-weakly convex for any $y \in \mathcal{Y}$. Therefore, for any $x\in\mathcal{X}$, $y\in\mathcal{Y}$, and $z\in\R^n$, we have
\begin{equation}\label{psi-strongcvx}
F_r(x,y, z)- F_r(x_r(y,z),y,z)
\geq \frac{r-L}{2}\|x-x_r(y,z)\|^{2}.
\end{equation}
Combining \eqref{func-cal} and \eqref{psi-strongcvx} yields
\begin{align}\label{lip-key1}
    &F_r(x_r(y,z),y,z^{\prime})-F_r(x_r(y,z^{\prime}),y,z^{\prime}) \notag\\
    =\ &F_r(x_r(y,z),y,z^{\prime})-F_r(x_r(y,z),y,z)+F_r(x_r(y,z),y,z)-F_r(x_r(y,z^{\prime}),y,z)\, \notag\\
    &-(F_r(x_r(y,z^{\prime}),y, z^{\prime})-F_r(x_r(y,z^{\prime}),y, z)) \notag\\
    \leq\ &\frac{r}{2}(\|z^{\prime}\|^{2}-2(z^{\prime}-z)^\top x_r(y,z)-\|z\|^{2})-\frac{r-L}{2}\|x_r(y,z)-x_r(y,z^{\prime})\|^{2}\,  \notag\\
    &-\frac{r}{2}(\|z^{\prime}\|^{2}-2(z^{\prime}-z)^\top x_r(y,z^{\prime})-\|z\|^{2}) \notag\\
    \leq\ &r(z^{\prime}-z)^\top(x_r(y,z^{\prime})-x_r(y,z))-\frac{r-L}{2}\|x_r(y,z)-x_r(y,z^{\prime})\|^{2}.
    \end{align}
 On the other hand, \eqref{psi-strongcvx} implies that
    \[
    F_r(x_r(y,z),y, z^{\prime})-F_r(x_r(y,z^{\prime}),y,z^{\prime}) \geq \frac{r-L}{2}\|x_r(y,z)-x_r(y,z^{\prime})\|^{2}.
    \]
Combining this inequality with \eqref{lip-key1}, we get
    \[
    (r-L)\|x_r(y,z)-x_r(y,z^{\prime})\|^{2}\leq r(z^{\prime}-z)^\top (x_r(y,z^{\prime})-x_r(y,z)).
    \]
Using the Cauchy-Schwarz inequality, we obtain
    \[
    \|x_r(y,z)-x_r(y,z^{\prime})\| \leq \frac{r}{r-L}\|z-z^{\prime}\|.
    \]
Therefore, we conclude that \eqref{lip-z} holds with $\sigma_1=\frac{r}{r-L}$. Since $f=\max_{y \in \mathcal{Y}} F(\cdot,y)$ is also $L$-weakly convex by  Fact \ref{fact-1}, we can prove that \eqref{lip-starz} holds by using a similar argument as above.

We  now proceed to prove \eqref{lip-y}. Using \eqref{psi-strongcvx}, we have
\begin{align}
&F_r(x_r(y^{\prime}, z),y, z)-F_r(x_r(y, z), y,z) \geq \frac{r-L}{2}\|x_r(y, z)-x_r(y^{\prime}, z)\|^{2},\label{psi-strongcvxineq1} \\
&F_r(x_r(y, z), y^{\prime},z)-F_r(x_r(y^{\prime}, z), y^{\prime},z) \geq \frac{r-L}{2}\|x_r(y, z)-x_r(y^{\prime}, z)\|^{2}\label{psi-strongcvxineq2}.
\end{align}
Moreover, by the $L$-Lipschitz continuity of $\nabla_{y} F_r(x, \cdot, z)$ for any $x \in \mathcal{X}$ and $z \in \mathbb{R}^n$, we have
\begin{equation}
\begin{aligned}\label{psi-yconcave}
 F_r(x_r(y, z), y^{\prime},z)-F_r(x_r(y, z), y,z)\leq &\langle\nabla_{y} F_r(x_r(y, z),y,z), y^{\prime}-y\rangle +\frac{L}{2}\|y-y^{\prime}\|^{2}
\end{aligned}
\end{equation}
and 
\begin{equation}
\begin{aligned}\label{psi-ylip}
 F_r(x_r(y^{\prime}, z), y, z)-F_r(x_r(y^{\prime}, z), y^{\prime},z) \leq & \langle\nabla_{y} F_r(x_r(y^{\prime}, z), y^{\prime}, z), y-y^{\prime}\rangle + \frac{L}{2}\|y-y^{\prime}\|^{2}.
\end{aligned}
\end{equation}
Incorporating \eqref{psi-strongcvxineq1}--\eqref{psi-ylip}, we obtain
\begin{equation*}
% \label{lip-key2}
\begin{aligned}
& (r-L)\|x_r(y, z)-x_r(y^{\prime}, z)\|^{2}  \\ \leq \ &\langle\nabla_{y} F_r(x_r(y, z), y, z)-\nabla_{y} F_r(x_r(y^{\prime}, z), y^{\prime}, z), y^{\prime}-y\rangle +L\|y-y^{\prime}\|^{2}. 
\end{aligned}
\end{equation*}
Since $\nabla_y F_r(\cdot,\cdot,z) = \nabla_y F(\cdot,\cdot)$ is $L$-Lipschitz for any $z \in \mathbb{R}^n$, we have
\begin{align*}
&\|\nabla_{y} F_r(x_r(y, z), y, z)-\nabla_{y} F_r(x_r(y^{\prime}, z), y^{\prime}, z)\| \notag\\
\leq\ &\|\nabla_{y} F_r(x_r(y, z), y, z)-\nabla_{y}F_r(x_r(y, z), y^{\prime}, z)\|  +\|\nabla_{y}F_r(x_r(y, z), y^{\prime}, z)-\nabla_{y} F_r(x_r(y^{\prime}, z), y^{\prime}, z)\|\notag\\
\leq\ &L(\|y-y^{\prime}\|+\|x_r(y^{\prime}, z)-x_r(y, z)\|).\notag
\end{align*}
It follows that
\begin{align*}
(r-L)\|x_r(y, z)-x_r(y^{\prime}, z)\|^{2} 
\leq L\|x_r(y^{\prime}, z)-x_r(y, z)\|\cdot\|y-y^{\prime}\|+2L\|y-y^{\prime}\|^{2}.
\end{align*}
Let $\zeta:=\frac{\|x_r(y, z)-x_r(y^{\prime}, z)\|}{\|y-y^{\prime}\|}$. 
Then, the above inequality gives
\begin{align*}
\zeta^{2}
& \ \leq \frac{L}{r-L} \zeta+\frac{2L}{(r-L)}\leq \frac{1}{2}\zeta^2+\frac{L^2}{2(r-L)^2} +\frac{2L}{(r-L)}\\
& \ \leq \frac{1}{2} \zeta^{2}+\frac{L^2+4L(r-L)}{2(r-L)^2}  \leq \frac{1}{2} \zeta^{2}+\frac{(L+2(r-L))^2}{2(r-L)^2},
\end{align*}
where the second inequality holds because $ab\leq \frac{1}{2}(a^2+b^2)$ for any $a,b\in\R$. Thus, we get
\[
\|x_r(y, z)-x_r(y^{\prime}, z)\| \leq  \frac{2r-L}{r-L}\|y-y^{\prime}\|,
\]
which shows that \eqref{lip-y} holds with $\sigma_{2}=\frac{2r-L}{r-L}$. The proof is complete.
\end{proof}
\begin{lemma}\label{lemma-dualdiff} 
    The dual potential value function $d_r$ is differentiable on $\mathcal{Y}\times \R^n$ with
    \begin{align*}
    &\nabla_{y} d_r(y,z)=\nabla_{y} F(x_r(y,z),y),\\
    &\nabla_{z} d_r(y,z)=\nabla_{z} F_r(x_r(y, z),y, z)=r(z-x_r(y,z))
    \end{align*}
    for any $y \in \mathcal{Y}$ and $z \in \mathbb{R}^n$. Moreover, for any $y \in \mathcal{Y}$ and $z \in \mathbb{R}^n$,  the gradients $\nabla_y d_r(\cdot,z)$ and $\nabla_z d_r(y,\cdot)$ are Lipschitz continuous, i.e.,
    \begin{align}
    &\|\nabla_{y} d_r(y^{\prime},z)-\nabla_{y} d_r(y^{\prime \prime},z)\| \leq L_{d_r}\|y^{\prime}-y^{\prime \prime}\| \quad \text{for all}\ y^{\prime}, y^{\prime \prime} \in \mathcal{Y},\notag\\
    &\|\nabla_{z} d_r(y,z')-\nabla_{z} d_r(y,z'')\| \leq L_{d_r}'\|z^{\prime}-z^{\prime \prime}\| \quad \text{for all}\ z^{\prime}, z^{\prime \prime} \in \R^n,\notag
    \end{align}
    where $L_{d_r}:=(\sigma_{2}+1)L$ and $L_{d_r}':=(\sigma_1 + 1)r$.
    \end{lemma}
\begin{proof}
Since $F_r(\cdot,y,z)$ is strongly convex, $F_r(x,\cdot,z)$ is weakly concave, and $F_r(x,y,\cdot)$ is strongly convex for any $x \in \mathcal{X}$, $y \in \mathcal{Y}$, and $z \in \mathbb{R}^n$, the function
    \[
    d_r(\cdot,\cdot)=\min_{x\in\mathcal{X}} F_r(x,\cdot,\cdot)
    \] 
    is differentiable on $\mathcal{Y}\times \R^n$ due to \citep[Theorem 10.31]{rockafellar2009variational}. In particular, for any $y\in\mathcal{Y}$ and $z\in\R^n$, one has
    \[
    \nabla_{y} d_r(y, z)=\nabla_{y} F_r(x_r(y, z),y, z)=\nabla_{y} F(x_r(y, z), y).
    \]
Using  \citep[Lemma 4.3]{drusvyatskiy2019efficiency}, we have 
    \[
    \nabla_{z} d_r(y,z)=\nabla_{z} F_r(x_r(y, z),y, z)=r(z-\prox_{\frac{1}{r}F(\cdot,y)+\iota_{\mathcal{X}}}(z))=r(z-x_r(y,z)).
    \]
    It follows that for  any $y^{\prime}, y^{\prime \prime} \in \mathcal{Y}$, 
    \begin{align*}
    &\|\nabla_{y} d_r(y^{\prime}, z)-\nabla_{y} d_r(y^{\prime \prime}, z)\|\notag\\
    =\ &\|\nabla_{y} F_r(x_r(y^{\prime}, z), y^{\prime},z)-\nabla_{y} F_r(x_r(y^{\prime \prime}, z), y^{\prime \prime},z)\| \\
    \leq\ &\|\nabla_{y} F_r(x_r(y^{\prime}, z), y^{\prime},z)-\nabla_{y} F_r(x_r(y^{\prime}, z), y^{\prime \prime}, z)\|+\|\nabla_{y} F_r(x_r(y^{\prime}, z), y^{\prime \prime},z)-\nabla_{y} F_r(x_r(y^{\prime \prime}, z), y^{\prime \prime}, z)\| \\
    \leq\ & L\|y^{\prime}-y^{\prime \prime}\|+L\|x_r(y^{\prime}, z)-x_r(y^{\prime \prime}, z)\|\notag\\
    \leq\ & L\|y^{\prime}-y^{\prime \prime}\|+L \sigma_{2}\|y^{\prime}-y^{\prime \prime}\|\leq L_{d_r}\|y^{\prime}-y^{\prime \prime}\|,
    \end{align*}
    where the third inequality is due to \eqref{lip-y}. Moreover, we have
    \begin{align*}
    \|\nabla_{z} d_r(y,z')-\nabla_{z} d_r(y,z'')\|=\ &r\|z'-x_r(y,z')-(z''-x_r(y,z''))\| \\
    \leq\ & r\|z^{\prime}-z^{\prime \prime}\|+r\|x_r(y, z')-x_r(y, z'')\| \\
    \leq\ & r\|z^{\prime}-z^{\prime \prime}\|+ r\sigma_{1}\|z^{\prime}-z^{\prime \prime}\|\leq L_{d_r}'\|z^{\prime}-z^{\prime \prime}\|,
    \end{align*}
where the second inequality is due to \eqref{lip-z}. The proof is complete.
\end{proof}

\begin{lemma}\label{lemma-dualxy}
For any $k\ge0$, we have
    \begin{equation*}
        \|y^{k+1}-y_{+}^{k}(z^{k+1})\| \leq \eta \|x^{k}-x^{k+1}\|+\sigma_1\alpha L\|z^{k+1}-z^k\|,
        \end{equation*}
where $\eta:=\alpha L \zeta$.
\end{lemma}
% \Jiajin{maybe change  it to be 
% \[
% \|y^{k+1}-y_{+}^{k}(z^{k+1})\| \leq \eta \|x^{k}-x^{k+1}\|+\alpha L \|z^{k+1}-z^k\|
% \]}
\begin{proof}
We compute
\begin{equation*}
% \label{iter-result2}
\begin{aligned}
\ &\|y^{k+1}-y_{+}^{k}(z^{k+1})\|\\ =\ &\|\proj_{\mathcal{Y}}(y^{k}+\alpha \nabla_{y} F(x^{k+1},y^{k}))-\proj_{\mathcal{Y}}(y^{k}+\alpha \nabla_{y} F(x_r(y^{k}, z^{k+1}),y^{k}))\|\notag \\
\leq\ &\|y^{k}+\alpha \nabla_{y} F(x^{k+1}, y^{k})-(y^{k}+\alpha \nabla_{y} F(x_r(y^{k}, z^{k+1}),  y^{k}))\|\notag \\
\leq\ & \alpha L(\|x^{k+1}-x_r(y^{k}, z^{k})\|+\|x_r(y^{k}, z^{k})-x_r(y^{k}, z^{k+1})\|) \notag \\
\leq\ & \eta\|x^{k}-x^{k+1}\|+\sigma_1\alpha L\|z^{k+1}-z^k\|,
\end{aligned}
\end{equation*}
where the first inequality is due to the nonexpansiveness of the projection operator, the second is due to the $L$-Lipschitz continuity of $\nabla_y F(\cdot,\cdot)$, and the third follows from Proposition \ref{prop:lip} and \eqref{lip-z}. The proof is complete.
\end{proof}
\section{Sufficient Decrease Property of Lyapunov Function}
\label{sec:suff_decrease}
\begin{lemma}[Primal descent]
\label{lemma-F-des} 
For any $k\ge0$, we have
\begin{align*}
& F_r(x^{k}, y^{k},z^k)-F_r(x^{k+1}, y^{k+1},z^{k+1})\\
\ge\ &
\frac{2\lambda^{-1}+r-L}{2}\|x^k-x^{k+1}\|^2 + \langle\nabla_{y} F_r(x^{k+1}, y^{k},z^k), y^{k}-y^{k+1}\rangle-\frac{L}{2}\|y^k-y^{k+1}\|^2 \\  &  +\frac{(2-\beta)r}{2 \beta}\|z^{k}-z^{k+1}\|^{2}. 
\end{align*}
\end{lemma}

\begin{proof}
One can infer from the definition that $F_{x^k,\lambda}(\cdot,y^k)$ is $\lambda^{-1}$-strongly convex. Therefore, we have
\begin{equation*}
\begin{aligned}
F_r(x^{k}, y^{k},z^{k}) & =F(x^{k}, y^{k})+ \frac{r}{2}\|x^k-z^{k}\|^2  = F_{x^k,\lambda}(x^k,y^k) + \frac{r}{2}\|x^k-z^{k}\|^2 \\
& \ge F_{x^k,\lambda}(x^{k+1},y^k)+\frac{r}{2}\|x^{k+1}-z^{k}\|^2+\frac{\lambda^{-1}+r}{2}\|x^k-x^{k+1}\|^2.
\end{aligned}
\end{equation*}
Moreover, Fact \ref{fact-2} implies that
\begin{equation*} 
F_{x^k,\lambda}(x^{k+1},y^k)\ge F(x^{k+1},y^k) + \frac{\lambda^{-1}-L}{2}\|x^{k+1}-x^k\|^2.
\end{equation*}
It follows that
\begin{align}\label{primaldec-key1}
F_r(x^{k},y^{k}, z^k) &\ge F(x^{k+1},y^k)+\frac{r}{2}\|x^{k+1}-z^k\|^2+\frac{2\lambda^{-1}+r-L}{2}\|x^k-x^{k+1}\|^2\notag\\
&= F_r(x^{k+1},y^{k}, z^{k})+\frac{2\lambda^{-1}+r-L}{2}\|x^k-x^{k+1}\|^2.
\end{align}
Next,  as $\nabla_y F_r(x,\cdot,z)$ is $L$-Lipschitz continuous for any $x\in\mathcal{X}$ and $z\in\R^n$, we have 
    \begin{align}\label{primaldec-key2}
        F_r(x^{k+1}, y^{k},z^k)-F_r(x^{k+1}, y^{k+1},z^k) \ge \langle\nabla_{y} F_r(x^{k+1},y^{k}, z^k), y^{k}-y^{k+1}\rangle-\frac{L}{2}\|y^k-y^{k+1}\|^2.
    \end{align}
At last, based on the update $z^{k+1}=z^{k}+\beta(x^{k+1}-z^{k})$, we get
\begin{equation}\label{primaldec-key3}
F_r(x^{k+1}, y^{k+1},z^{k})-F_r(x^{k+1},y^{k+1}, z^{k+1})  = \frac{(2-\beta)r}{2 \beta}\|z^{k}-z^{k+1}\|^{2}.
\end{equation}
By summing up \eqref{primaldec-key1}, \eqref{primaldec-key2}, and \eqref{primaldec-key3}, we obtain the desired result.
\end{proof}

\begin{lemma}[Dual ascent]\label{lemma-d-aes} 
For any $k\ge0$, we have
\begin{equation*}
% \label{dual-ascent}
 \begin{aligned}
    \ & d_r(y^{k+1},z^{k+1})-d_r(y^{k},z^k)\\
    \ge\ &\langle\nabla_{y} F_r(x_r(y^{k}, z^{k}),  y^{k}, z^{k}), y^{k+1}-y^{k}\rangle-\frac{L_{d_r}}{2}\|y^{k}-y^{k+1}\|^{2}\,\\
    &+\frac{r}{2}\left\langle z^{k+1}-z^{k},z^{k+1}+z^{k}-2 x_r(y^{k+1}, z^{k+1})\right\rangle. 
    \end{aligned}
\end{equation*}
\end{lemma}
\begin{proof}
Based on Lemma \ref{lemma-dualdiff}, $\nabla_{y} d_r(\cdot,z)$ is $L_{d_r}$-Lipschitz continuous for any $z \in \mathbb{R}^n$. Thus, we have 
        \begin{align*}
        \ & d_r(y^{k+1},z^k)-d_r(y^{k},z^k)  
        \geq \langle\nabla_{y} d_r(y^{k},z^k), y^{k+1}-y^{k}\rangle-\frac{L_{d_r}}{2}\|y^{k}-y^{k+1}\|^{2} \\
        =\ & \langle \nabla_{y} F_r(x_r(y^{k}, z^{k}), y^{k}, z^{k}), y^{k+1}-y^{k}\rangle-\frac{L_{d_r}}{2}\|y^{k}-y^{k+1}\|^{2}.
        \end{align*}
In addition, one has 
        \begin{align*}
        &d_r(y^{k+1}, z^{k+1})-d_r(y^{k+1}, z^{k})\notag\\ 
        =\ & F_r(x_r(y^{k+1}, z^{k+1}), y^{k+1}, z^{k+1})-F_r(x_r(y^{k+1}, z^{k}), y^{k+1},z^{k}) \\
        \geq\ &  F_r(x_r(y^{k+1}, z^{k+1}), y^{k+1},z^{k+1})-F_r(x_r(y^{k+1}, z^{k+1}), y^{k+1},z^{k}) \\
        =\ & \frac{r}{2}\|x_r(y^{k+1}, z^{k+1})-z^{k+1}\|^{2}-\frac{r}{2}\|x_r(y^{k+1}, z^{k+1})-z^{k}\|^{2} \\
        =\ & \frac{r}{2} \left\langle z^{k+1}-z^{k},z^{k+1}+z^{k}-2 x_r(y^{k+1}, z^{k+1})\right\rangle.
        \end{align*}
Finally, by combining the above inequalities, the proof is complete.
\end{proof}
 
\begin{lemma}[Proximal descent (smoothness)]
\label{lemma-p-des} 
    For any $k\ge0$, we have
   \begin{equation*}
   % \label{proximal-descent}
    p_r(z^k)-p_r(z^{k+1}) \ge \frac{r}{2}\left\langle z^{k+1}-z^{k},2 x_r(y(z^{k+1}), z^{k})-z^{k}-z^{k+1}\right\rangle, 
    \end{equation*}
    where $y(z^{k+1})\in Y(z^{k+1})$.
    \end{lemma}
\begin{proof}
% From Sion's minimax theorem \citep{sion1958general}, we know that
%     \[
%     \min _{x \in \mathcal{X}} \max _{y \in \mathcal{Y}} F_r(x, y, z)=\max_{y \in \mathcal{Y}} \min_{x \in \mathcal{X}} F_r(x,y,z)
%     \]
%     which implies
Recall that 
% $p_r(z)$:
    % \[
    $p_r(z)=\max _{y \in \mathcal{Y}} d_r(y, z)$.
    % \]
Due to the definition of $y(z^{k+1})$, we have
    \begin{align*}
     &p_r(z^{k+1})-p_r(z^{k})\notag\\
    \leq\ & d_r(y(z^{k+1}), z^{k+1})-d_r(y(z^{k+1}), z^{k}) \\
    \leq\ & F_r(x_r(y(z^{k+1}), z^{k}), y(z^{k+1}), z^{k+1})-F_r(x_r(y(z^{k+1}), z^{k}), y(z^{k+1}),z^{k}) \\
    =\ & \frac{r}{2}\left\langle z^{k+1}-z^{k},z^{k+1}+z^{k}-2 x_r(y(z^{k+1}), z^{k})\right\rangle,
    \end{align*}
    where the second inequality follows from the fact that \[F_r(x', y, z)\geq \min_{x\in\mathcal{X}} F_r(x,y,z) = d_r(y, z)\] for any $x'\in\mathcal{X}$. The proof is complete.
    \end{proof}
\paragraph{{\bf Proof of Proposition \ref{prop:decrease}}}
% \begin{proof}
From Lemmas \ref{lemma-F-des}, \ref{lemma-d-aes}, and \ref{lemma-p-des}, we have
        \begin{align*}
        % \label{desascent-key1}
        & \Phi_r(x^k,y^k,z^k)-\Phi_r(x^{k+1},y^{k+1},z^{k+1}) \notag\\
         =\ & F_r(x^k,y^k,z^k)-F_r(x^{k+1},y^{k+1},z^{k+1}) +2(d_r(y^{k+1},z^{k+1})-d_r(y^k,z^k))+2(p_r(z^k)-p_r(z^{k+1})) \notag\\
         \ge\ & \frac{2\lambda^{-1}+r-L}{2}\|x^k-x^{k+1}\|^2 +\frac{(2-\beta)r}{2 \beta}\|z^{k}-z^{k+1}\|^{2}-\left(L_{d_r}+\frac{L}{2}\right)\|y^{k}-y^{k+1}\|^{2}\,  \notag\\
        &  + \underbrace{\langle\nabla_{y} F_r(x^{k+1}, y^{k},z^k), y^{k}-y^{k+1}\rangle+2\langle\nabla_{y} F_r(x_r(y^{k}, z^{k}),  y^{k}, z^{k}), y^{k+1}-y^{k}\rangle}_{\text{\ding{172}}}\, 
      \notag \\ 
      & + \underbrace{2 r\left\langle z^{k+1}-z^{k},x_r(y(z^{k+1}), z^{k})-x_r(y^{k+1}, z^{k+1})\right\rangle}_{\text{\ding{173}}}.
        \end{align*}
% Subsequently, we simplify the terms \text{\ding{172}} and  \text{\ding{173}}. 
First, we simplify the term \text{\ding{172}}. We know that
        \begin{align*}
        \text{\ding{172}} 
         =\ & \langle\nabla_{y} F_r(x^{k+1}, y^{k},z^k), y^{k}-y^{k+1}\rangle+2\langle\nabla_{y} F_r(x_r(y^{k}, z^{k}), y^{k}, z^{k}), y^{k+1}-y^{k}\rangle\\
        =\ & \langle\nabla_{y} F_r(x^{k+1}, y^{k},z^k), y^{k+1}-y^{k}\rangle + 2\langle\nabla_{y} F_r(x_r(y^{k}, z^{k}), y^{k}, z^{k})-\nabla_{y} F_r(x^{k+1}, y^{k},z^k), y^{k+1}-y^{k}\rangle.
        \end{align*}
        For the first term, we have
        \begin{align*}
    &\langle\nabla_{y} F_r(x^{k+1}, y^{k},z^k), y^{k+1}-y^{k}\rangle \notag\\
     =\ &  \langle\nabla_{y} F_r(x^{k+1}, y^{k},z^k)+\frac{1}{\alpha}(y^{k}-y^{k+1}), y^{k+1}-y^{k}\rangle + \frac{1}{\alpha}\|y^{k}-y^{k+1}\|^2\\
     =\ & \frac{1}{\alpha}\langle y^{k} + \alpha \nabla_{y} F_r(x^{k+1}, y^{k},z^k)-y^{k+1}, y^{k+1}-y^{k}\rangle +\frac{1}{\alpha}\|y^{k}-y^{k+1}\|^2\\
     \ge\ &  \frac{1}{\alpha}\|y^{k}-y^{k+1}\|^2, 
\end{align*}
where the last inequality follows from the property of the projection operator and the dual update $
y^{k+1} = \proj_{\mathcal{Y}}( y^{k} + \alpha \nabla_{y} F_r(x^{k+1}, y^{k},z^k))
$ (recall that $\nabla_y F_r(x,y,z) = \nabla_y F(x,y)$ for any $x \in \mathcal{X}$, $y \in \mathcal{Y}$, and $z \in \mathbb{R}^n$). For the remaining terms, we have
                \begin{align*}
        %\label{desascent-key3}
         &2\langle\nabla_{y} F_r(x_r(y^{k}, z^{k}), y^{k}, z^{k})-\nabla_{y} F_r(x^{k+1}, y^{k},z^k), y^{k+1}-y^{k}\rangle \notag\\
         \ge & -2\|\nabla_{y} F_r(x_r(y^{k}, z^{k}),y^{k},z^{k})-\nabla_{y} F_r(x^{k+1}, y^{k},z^{k})\| \cdot\|y^{k+1}-y^{k}\| \notag\\
         \ge & -2 L\|x^{k+1}-x_r(y^{k}, z^{k})\| \cdot\|y^{k}-y^{k+1}\| \notag\\
        \ge  & -L \zeta^2\|y^{k}-y^{k+1}\|^{2}-L\zeta^{-2} \|x^{k+1}-x_r(y^{k}, z^{k})\|^{2} \notag\\
         \ge & -L \zeta^2\|y^{k}-y^{k+1}\|^{2}-L\|x^{k+1}-x^{k}\|^{2},
        \end{align*}
        where the third inequality 
        % follows from $2|x||y| \leq \tau x^2 + \frac{1}{\tau}y^2$ 
        holds because $2 |a| |b| \le \tau a^2 + \tfrac{1}{\tau} b^2$
        for any $a,b \in \mathbb{R}$ and $\tau>0$, and the last inequality follows from Proposition \ref{prop:lip}. 
        % where the first inequality is because of the Cauchy-Schwarz in equality, the second inequality is because $\nabla_{y} K=\nabla_{y} f$ is $L$-Lipschitz-continuous, the third inequality is due to the AM-GM inequality and the last is because of (B.5).
% Subsequently, we focus on the left term in $\text{\ding{172}}$ that 
Putting the above together, we obtain
\begin{equation}
\label{desascent-key3}
    \text{\ding{172}}  \ge \left(\frac{1}{\alpha}-L\zeta^2\right) \|y^{k}-y^{k+1}\|^{2}-L\|x^{k+1}-x^{k}\|^{2}. 
\end{equation}
Next, we bound the term $\text{\ding{173}}$ by 
\begin{align}\label{desascent-key2}
  \text{\ding{173}} 
 =\ & 2 r \left\langle z^{k+1}-z^{k}, x_r(y(z^{k+1}), z^{k})-x_r(y^{k+1}, z^{k+1})\right\rangle \notag\\
 =\ & 2 r \left\langle z^{k+1}-z^{k}, x_r(y(z^{k+1}), z^{k})-x_r(y(z^{k+1}), z^{k+1})\right\rangle\, \notag\\
 & + 2 r\left\langle z^{k+1}-z^{k}, x_r(y(z^{k+1}), z^{k+1})-x_r(y^{k+1}, z^{k+1})\right\rangle\notag \\
\ge &  -2r \sigma_1\|z^{k+1}-z^k\|^2+2 r\left\langle z^{k+1}-z^{k}, x_r(y(z^{k+1}), z^{k+1})-x_r(y^{k+1}, z^{k+1})\right\rangle \notag\\
\ge & -2r \sigma_1\|z^{k+1}-z^k\|^2-\frac{r}{7\beta}\|z^{k+1}-z^k\|^2  -7r\beta \|x_r(y(z^{k+1}), z^{k+1})-x_r(y^{k+1}, z^{k+1})\|^2, 
\end{align}
where the first inequality is due to \eqref{lip-z} and the Cauchy-Schwarz inequality, and the second inequality again follows from the fact that $2 |a| |b| \le \tau a^2 + \tfrac{1}{\tau} b^2$ for any $a,b \in \mathbb{R}$ and $\tau > 0$. Now, the inequalities \eqref{desascent-key3} and \eqref{desascent-key2} imply that
\begin{align}\label{desas-keymid}
        & \Phi_r(x^k,y^k,z^k)-\Phi_r(x^{k+1},y^{k+1},z^{k+1}) \notag\\
        \geq\ & \frac{2\lambda^{-1}+r-3L}{2}\|x^{k}-x^{k+1}\|^{2}+\left(\frac{1}{ \alpha}-L_{d_r}-\frac{L}{2}-L\zeta^2\right)\|y^{k}-y^{k+1}\|^{2}\, \notag\\
        &+\left(\frac{(2-\beta)r}{2 \beta}-2r\sigma_1-\frac{r}{7\beta}\right)\|z^{k}-z^{k+1}\|^{2}-7r\beta \|x_r(y(z^{k+1}), z^{k+1})-x_r(y^{k+1}, z^{k+1})\|^2. 
\end{align}
By Lemma \ref{lemma-dualxy} and the fact that $\| u+v \|^2 \le 2(\|u\|^2+\|v\|^2)$ for any $u,v \in \mathbb{R}^d$, we have
\begin{align}\label{desas-final-key1}
\|y^{k+1}-y^{k}\|^{2} 
&=\|y^{k+1}-y_{+}^{k}(z^{k+1})+y_{+}^{k}(z^{k+1})-y^{k}\|^{2}\notag\\
&\geq\frac{1}{2}\|y^{k}-y_{+}^{k}(z^{k+1})\|^{2} -\|y^{k+1}-y_{+}^{k}(z^{k+1})\|^{2}\notag\\
% & \geq \frac{1}{2}\|y^{k}-y_{+}^{k}(z^{k+1})\|^{2} -\alpha^2L^2\|x^{k+1}-x_r(y^k,z^{k+1})\|^2\notag\\
& \geq \frac{1}{2}\|y^{k}-y_{+}^{k}(z^{k+1})\|^{2} -2\eta^2\|x^{k+1}-x^k\|^2-2\sigma_1^2\alpha^2L^2\|z^{k+1}-z^k\|^2.
\end{align}
% where we recall that 
% \[
% y_{+}^{k}(z^{k+1}) = \proj_{\mathcal{Y}}(y^k+\alpha \nabla_y F_r (x_r(y^k,z^{k+1}),y^k,z^{k+1})).
% \]
% \leq\ & \alpha L\|x^{k+1}-x_r(y^{k}, z^{k})\| \notag \\
% \leq\ & \eta\|x^{k}-x^{k+1}\|,
Similarly, by Lemma \ref{lemma-sollip} and Lemma \ref{lemma-dualxy}, we have
\begin{align}\label{desas-final-key2}
&\|x_r(y(z^{k+1}), z^{k+1})-x_r(y^{k+1}, z^{k+1})\|^{2}\notag\\
\leq \ & 2\|x_r(y(z^{k+1}), z^{k+1})-x_r(y_+^{k}(z^{k+1}),z^{k+1})\|^2   + 2\|x_r(y_+^{k}(z^{k+1}), z^{k+1})-x_r(y^{k+1}, z^{k+1})\|^2 \notag\\
\leq \ & 2\|x_r(y(z^{k+1}), z^{k+1})-x_r(y_+^{k}(z^{k+1}),z^{k+1})\|^2+2\sigma_2^2 \|y^{k+1}-y_{+}^{k}(z^{k+1})\|^{2}\notag\\
\leq \ & 2\|x_r(y(z^{k+1}), z^{k+1})-x_r(y_+^{k}(z^{k+1}),z^{k+1})\|^2 + 4\sigma_2^2\eta^2 \|x^{k+1}-x^k\|^2+4\sigma_2^2\sigma_1^2\alpha^2L^2\|z^{k+1}-z^k\|^2.
\end{align}

% \begin{align}\label{desas-final-key2}
% &\|x_r(y(z^{k+1}), z^{k+1})-x_r(y^{k+1}, z^{k+1})\|^{2}\notag\\
% % =\ & \|x_r^\star(z^{k+1})-x_r^\star(z^{k})+x_r^\star(z^{k})-x_r(y_{+}^{k}(z^{k}), z^{k})+x_r(y_{+}^{k}(z^{k}), z^{k})-x_r(y^{k+1}, z^{k})+x_r(y^{k+1}, z^{k})-x_r(y^{k+1}, z^{k+1})\|^2 \notag\\
% \leq\ & 4\|x_r(y(z^{k+1}), z^{k+1})-x_r(y(z^{k}), z^{k})\|^{2}+4\|x_r(y(z^{k}), z^{k})-x_r(y_{+}^{k}(z^{k}), z^{k})\|^{2}\ +\notag \\
% &4\|x_r(y_{+}^{k}(z^{k}), z^{k})-x_r(y^{k+1}, z^{k})\|^{2}+4\|x_r(y^{k+1}, z^{k})-x_r(y^{k+1}, z^{k+1})\|^{2} \notag\\
% \leq\ & 8 \sigma_{1}^{2}\|z^{k}-z^{k+1}\|^{2}+4\|x_r(y(z^{k}), z^{k})-x_r(y_{+}^{k}(z^{k}), z^{k})\|^{2} +4 \sigma_{2}^{2} \eta^2\|x^{k}-x^{k+1}\|^2.
% \end{align}
% Besides, we recall that 
% \begin{align}\label{desas-final-key0}
% x^{k+1}-z^{k} =\frac{1}{\beta}(z^{k+1}-z^k).
% \end{align} 
Substituting \eqref{desas-final-key1} and \eqref{desas-final-key2} into \eqref{desas-keymid} yields
\begin{align*}
&\Phi_r(x^k,y^k,z^k)-\Phi_r(x^{k+1},y^{k+1},z^{k+1})\notag\\
 \geq\ &\left( \frac{2\lambda^{-1}+r-3L}{2} -28 r\beta \sigma_2^2\eta^2 \right)\|x^{k}-x^{k+1}\|^{2}\ +\left(\frac{1}{ \alpha}-L_{d_r}-\frac{L}{2}-L\zeta^2\right) \cdot \\
& \left(\frac{1}{2}\|y^{k}-y_{+}^{k}(z^{k+1})\|^{2} -2\eta^2\|x^{k+1}-x^k\|^2-2\sigma_1^2\alpha^2L^2\|z^{k+1}-z^k\|^2 \right)\ \notag\\
&+ \left(\frac{(2-\beta)r}{2 \beta}-2r\sigma_1-\frac{r}{7\beta}-28r\beta\sigma_1^2\sigma_2^2 \alpha^2 L^2 \right)\|z^{k}-z^{k+1}\|^{2} \\  & -  14r\beta\|x_r(y(z^{k+1}), z^{k+1})-x_r(y_+^{k}(z^{k+1}),z^{k+1})\|^2.
\end{align*}

Suppose that $r\ge 3L$,
% (which implies $L_{d_{r}}+\frac{L}{2} \leq 5L$)
which implies that $L_{d_{r}}+\frac{L}{2} \leq 5L$. We observe the following:
\begin{itemize}
\item As $\alpha \leq \min\left\{\frac{1}{10L}, \frac{1}{4L\zeta^2} \right\}$, we have $\frac{1}{ \alpha}-L_{d_r}-\frac{L}{2} \ge \frac{1}{2\alpha}$ and 
$
\frac{1}{2\alpha} - L\zeta^2 \ge \frac{1}{4\alpha}. 
$
\item As $\beta\leq\frac{1}{28}$ and $\sigma_1 \leq \frac{3}{2}$, $\sigma_2\leq \frac{5}{2}$, we have 
\begin{align*}
& \frac{(2-\beta)r}{2 \beta}-2r\sigma_1-\frac{r}{7\beta}-28r\beta\sigma_1^2 \sigma_2^2 \alpha^2 L^2  -\frac{1}{2\alpha} \sigma_1^2\alpha^2L^2 \\ 
\ge\ & \frac{6r}{7\beta} -\frac{7r}{2}- 28r\beta\left(\frac{3}{2}\right)^2\left(\frac{5}{2}\right)^2\left(\frac{1}{10L}\right)^2L^2-\frac{r}{24}\\ 
\ge\ &  \frac{r}{\beta}\left(\frac{6}{7}-4\beta- 14\beta^2\right) \ge \frac{4r}{7\beta}.
\end{align*}
\item As $\lambda^{-1} \ge L$ and recalling that $\eta = \alpha L \zeta$, we have 
\[\alpha\leq\frac{1}{4L \zeta^2} = \frac{\lambda^{-1}}{4L^2\zeta^2} \frac{L}{\lambda^{-1}} \leq \frac{\lambda^{-1}}{4L^2\zeta^2}\quad \text{and}\quad 
\frac{\eta^2}{2\alpha} = \frac{\alpha L^2\zeta^2}{2} \leq \frac{1 }{8\lambda}.
\]
Moreover, since $\beta  \leq \tfrac{(r-L)^2}{14 \alpha r (2r-L)^2} = \tfrac{1}{14 \alpha r \sigma_2^2}$,
 we have 
 \[
 28 r\beta \sigma_2^2\eta^2 \leq \frac{ 2\eta^2}{\alpha} \leq \frac{1}{2\lambda}\quad\text{and}\quad \frac{2\lambda^{-1}+r-3L}{2} -28 r\beta \sigma_2^2\eta^2 - \frac{\eta^2}{2\alpha} \ge \frac{3}{8\lambda}.
\]
\end{itemize}
% Together all pieces, we get 
Putting everything together yields
\begin{align*}
&\Phi_r(x^k,y^k,z^k)-\Phi_r(x^{k+1},y^{k+1},z^{k+1}) \\
\ge\ & \frac{3}{8\lambda}\|x^{k}-x^{k+1}\|^{2}+\frac{1}{8\alpha}\|y^{k}-y_{+}^{k}(z^{k+1})\|^{2} +\frac{4r}{7\beta}\|z^{k}-z^{k+1}\|^{2} \\  & -14r\beta\|x_r(y(z^{k+1}), z^{k+1})-x_r(y_{+}^{k}(z^{k+1}), z^{k+1})\|^{2}.
\end{align*}
The proof is complete.
% \end{proof}
\section{Proof of Corollary \ref{prop:dual_eb_KL2}}
\label{app:coro}

\begin{proof}
The proof follows closely that of Proposition \ref{prop:dual_eb_KL}. Let $\psi:\R^n\times \R^n\rightarrow\R$ be the function defined by
\[
\psi(x,z) = \max\limits_{y\in \mathcal{Y}} F_r(x,y,z)
\]
and consider arbitrary $x \in \mathcal{X}$, $y \in \mathcal{Y}$, and $z \in \mathbb{R}^n$. Again, note that  $\psi(\cdot,z)$ is $(r-L)$-strongly convex. This implies that
% for any $x\in \mathcal{X}$, we have
\begin{equation}
\label{OS-GS-better-key1}
\psi(x,z) - \psi(x^\star_r(z),z) \ge \frac{r-L}{2}\|x-x^\star_r(z)\|^2. 
\end{equation}
Moreover, 
% for any $x\in \mathcal{X}$, 
we  have
\begin{align} 
\label{OS-GS-better-key2}
&\psi(x,z)-\psi(x_r^\star(z),z) \notag\\   
\leq \  &\psi(x,z)-F_r(x_r(y_+(z),z),y_+(z),z)\notag \\
= \ & \max_{y^\prime\in \mathcal{Y}} F(x,y^\prime) + \frac{r}{2}\|x-z\|^2-F_r(x_r(y_+(z),z),y_+(z),z)\notag\\
=\,& \max_{y^\prime\in \mathcal{Y}}F(x,y^\prime)-F(x_r(y_+(z),z),y_+(z)) + \frac{r}{2}\|x-z\|^2-\frac{r}{2}\|x_r(y_+(z),z)-z\|^2,
\end{align}
where the  inequality follows from
\begin{align*}
F_r(x_r(y_+(z),z),y_+(z),z) 
=\ & \min_{x'\in\mathcal{X}}\left\{ F(x',y_+(z)) +\frac{r}{2}\|x'-z\|^2\right\} \\
\leq\ & \max_{y' \in \mathcal{Y}}\min_{x'\in\mathcal{X}}\left\{F(x',y') +\frac{r}{2}\|x'-z\|^2\right\} \\
\leq  \ & \min_{x'\in\mathcal{X}}\max_{y' \in \mathcal{Y}}\left\{F(x',y') +\frac{r}{2}\|x'-z\|^2\right\}
= \psi(x_r^\star(z),z).
\end{align*}
As \eqref{OS-GS-better-key1} and \eqref{OS-GS-better-key2} hold for any $x\in\mathcal{X}$, we obtain the intermediate relation
\begin{equation*}
% \label{OS-GS-better-key3} 
    \frac{r-L}{2}\|x_r^\star(z)-x_r(y_+(z),z)\|^2 \leq \max\limits_{y^\prime\in \mathcal{Y}} F(x_r(y_+(z),z),y^\prime) - F(x_r(y_+(z),z),y_+(z)) 
\end{equation*}
by taking $x= x_r(y_+(z),z)$.

The remaining steps of the proof are the same as those of Proposition \ref{prop:dual_eb_KL}. For brevity, we omit them here.
\end{proof}

\section{Dual Error Bound for Concave Case} 
\label{sec:dua-appen}
The following lemma provides a dual error bound for the case where $F(x,\cdot)$ is concave for any $x \in \mathbb{R}^n$. Its proof is based on  Lemma B.10 in \citep{zhang2020single}.

\label{sec:dual}
\begin{lemma}\label{lemma-dual-bd}
For any $y \in \mathcal{Y}$ and $z\in\R^n$, we have
\begin{equation*}
% \label{dual-bd}
\|x^\star_r(z)-x_r(y_{+}(z), z)\|^{2} 
 \leq \kappa \|y-y_{+}(z)\| ,
\end{equation*}
where $\kappa:=\frac{1+ \alpha L\sigma_2+\alpha L}{\alpha(r-L)}\cdot{\rm diam}(\mathcal{Y})$.
\end{lemma}
\begin{proof}
Let $y \in \mathcal{Y}$ and $z \in \mathbb{R}^n$ be arbitrary. Recall that $y(z)$ is an arbitrary vector in $Y(z)$. By the $(r-L)$-strong convexity of $F_r(\cdot, y, z)$, we have
\begin{align}
&F_r(x_r^\star(z), y_{+}(z), z)-F_r(x_r(y_{+}(z), z), y_{+}(z), z)  \geq \frac{r-L}{2}\|x_r(y_{+}(z), z)-x^\star_r(z)\|^{2}, \notag\\
&F_r(x_r(y_{+}(z), z), y(z), z)-F_r(x_r^\star(z), y(z), z)  \geq \frac{r-L}{2}\|x_r(y_{+}(z), z)-x_r^\star(z)\|^{2}.\notag
\end{align}
 Moreover, since $(x_r^\star(z),y(z))$ is a saddle point of the convex-concave function $F_r(\cdot,\cdot,z)$ \citep{sion1958general}, we have $F_r(x_r^\star(z), y(z), z) \geq F_r(x_r^\star(z), y_{+}(z), z)$, which implies that
\begin{equation}\label{dualbd-strcvx}
F_r(x_r(y_{+}(z), z), y(z), z)-F_r(x_r(y_{+}(z), z), y_{+}(z), z) \geq  (r-L)\|x_r(y_{+}(z), z)-x_r^\star(z)\|^{2}.
\end{equation}
Note that $y_{+}(z)$ is the maximizer of 
\[
\max_{y' \in \mathcal{Y}}\big\langle y+\alpha \nabla_{y} F_r(x_r(y, z), y, z)-y_{+}(z) , y'\big\rangle.
\]
For simplicity, we define the function $\xi: \mathbb{R}^d \rightarrow \mathbb{R}$ by
\begin{align*}
\xi(\cdot)
:=\ &\alpha F_r(x_r(y_{+}(z), z),\cdot, z)-\big\langle \alpha\nabla_{y} F_r(x_r(y_{+}(z), z), y_{+}(z), z),\cdot\big\rangle \\ \ &-\big\langle y_{+}(z)-y-\alpha \nabla_{y} F_r(x_r(y, z), y, z), \cdot\big\rangle.
\end{align*}
Then, we have
\begin{align*}
\max_{y' \in \mathcal{Y}}\ \xi(y')
&\leq \alpha F_r(x_r(y_{+}(z), z),y_+(z), z)-\big\langle\alpha\nabla_y F_r(x_r(y_{+}(z), z),y_+(z), z), y_+(z)\big\rangle\, \notag\\
&\quad + \max_{y' \in \mathcal{Y}} \big\langle y+\alpha \nabla_{y} F_r(x_r(y, z), y, z)-y_{+}(z), y'\big\rangle \notag\\
&\leq \xi(y_{+}(z)),
\end{align*}
where the first inequality holds because $F_r(x_r(y_{+}(z), z),\cdot,z)$ is concave. Therefore, we have $\xi(y(z))\leq \xi(y_{+}(z))$, which implies that
\begin{align}\label{dualbd-key}
&F_r(x_r(y_{+}(z), z), y(z), z)-F_r(x_r(y_{+}(z), z), y_{+}(z), z) \notag\\
\leq &\ \frac{1}{\alpha}\Big\langle y(z)-y_{+}(z), \alpha\nabla_{y} F_r(x_r(y_{+}(z), z), y_{+}(z), z)-\alpha \nabla_{y} F_r(x_r(y, z), y, z)\Big\rangle \notag \\ \ & \ + \frac{1}{\alpha}\Big\langle y(z)-y_{+}(z), y_{+}(z)-y\Big\rangle \notag\\
\leq &\ L \|y_{+}(z)-y(z)\|(\|x_r(y_{+}(z), z)-x_r(y, z)\|+ \|y_{+}(z)-y\|)  + \frac{1}{\alpha}\|y_{+}(z)-y(z)\|\cdot \|y-y_{+}(z)\|\notag\\
\leq &\ \left(\frac{1}{\alpha}+L\sigma_2+L\right)\|y_{+}(z)-y(z)\| \cdot\|y-y_{+}(z)\|.
\end{align}
Here, the second inequality is due to the $L$-Lipschitz continuity of $\nabla_y F_r(\cdot,\cdot,z)$ and the last inequality is from \eqref{lip-y}.
Hence, by combining \eqref{dualbd-strcvx} and \eqref{dualbd-key}, we obtain
\[
(r-L)\|x_r^\star(z)-x_r(y_{+}(z), z)\|^{2}\leq \left(\frac{1}{\alpha}+ L\sigma_2+L\right){\rm diam}(\mathcal{Y})\cdot\|y-y_{+}(z)\|,
\]
which proves the desired result.
\end{proof}

\end{appendices}

\end{document}